\documentclass[12pt]{amsart}
\usepackage{hyperref}
\usepackage[utf8]{inputenc}
\usepackage{dsfont}
\usepackage[margin=1.48in]{geometry}
\usepackage{amsthm}
\usepackage{amsfonts}
\usepackage{amssymb}
\usepackage{amsmath}
\usepackage[english]{babel}
\theoremstyle{definition}
\numberwithin{equation}{section}
\newtheorem{conj}[equation]{Conjecture}
\theoremstyle{definition}
\newtheorem{thm}[equation]{Theorem}
\theoremstyle{definition}
\newtheorem{cor}[equation]{Corollary}
\theoremstyle{definition}
\newtheorem{prop}[equation]{Proposition}
\theoremstyle{definition}
\newtheorem{lem}[equation]{Lemma}
\theoremstyle{definition}
\newtheorem{definition}[equation]{Definition}
\theoremstyle{definition}
\newtheorem{rem}[equation]{Remark}
\theoremstyle{definition}

\def\op{\mathrm{op}}
\begin{document}
\title[On Kazhdan-Yom Din Asymptotic Orthogonality]{On Kazhdan-Yom Din Asymptotic Orthogonality for $K$-finite matrix coefficients of tempered representations}
\author{Anne-Marie Aubert} 

\address{Sorbonne Universit\'e and Universit\'e Paris Cit\'e, CNRS, IMJ-PRG, F-75005 Paris, France}

\email{anne-marie.aubert@imj-prg.fr}

\author{Alfio Fabio La Rosa}

\address{Institute for the Advanced Study of Mathematics, Zhejiang University, East No. 7, Zijingang Campus, Hangzhou, China}

\email{alfiofabiolarosa@zju.edu.cn}
\date{}
\keywords{real Lie group, tempered representation, asymptotic Schur orthogonality}

\subjclass{22E50, 20G05}
\maketitle

\begin{abstract} In a recent article, D. Kazhdan and A. Yom Din conjectured the validity of an asymptotic form of Schur orthogonality for 
tempered, irreducible, unitary representations of semisimple groups defined over local fields. In the non-Archimedean case, they established it for $K$-finite matrix coefficients. The purpose of this article is to prove the analogous result in the Archimedean case. 
\end{abstract}

\tableofcontents

\section{Introduction} \label{sec:Introduction}
Let $G$ be a semisimple group over a local field, let $K$ be a maximal compact subgroup of $G$. We fix a Haar measure on $G$, denoted $\,dg$. If $H$ is the Hilbert space underlying a unitary representation of $G$, let $H_{K}$ denote the space of $K$-finite vectors and $H^{\infty}$ the space of smooth vectors.\\

In their recent work \cite{KYD}, D. Kazhdan and A. Yom Din conjectured the validity of an asymptotic version of Schur orthogonality relations. It should hold for matrix coefficients of tempered, irreducible,
unitary representations of $G$, generalising Schur well-known orthogonality relations for discrete series.\\

Following their article, we fix a norm on the Lie algebra $\mathfrak{g}$ of $G$. By \cite[Claim 5.2]{KYD}, we can choose it so that $\text{Ad}K$ acts unitarily on $\mathfrak{g}$. We define the function 
\begin{equation*}\textbf{r}:G\longrightarrow \mathbb{R}_{\geq0},\quad \textbf{r}(g)=\textrm{log} \left( \textrm{max}\{\|\textrm{Ad}(g)\|_{\op},\|\textrm{Ad}(g^{-1})\|_{\op} \}\right)
\end{equation*}
so that, given $r\in\mathbb{R}_{>0}$, we can introduce the corresponding ball 
\begin{equation*}
G_{<r}:=\{ g\in G|\textbf{r}(g)< r\}.
\end{equation*}

Given this set-up, we can state their conjecture.

\begin{conj}[Kazhdan-Yom Din, Asymptotic Schur Orthogonality Relations]
\label{asym}
Let $G$ be a semisimple group over a local field $F$ and let $(\pi, H)$ be a tempered, irreducible, unitary representation of $G$. Then there exist $\textbf{d}(\pi)\in \mathbb{Z}_{\geq0}$ and $\textbf{f}(\pi)\in \mathbb{R}_{>0}$ such that, for all $v_{1},v_{2},v_{3},v_{4}\in H$, the following holds: 
\begin{equation*}
\lim_{r\to\infty}\frac{1}{r^{\textbf{d}(\pi)}}\int_{G_{<r}}\langle \pi(g) v_{1},v_{2}\rangle \overline{\langle \pi(g)v_{3},v_{4}\rangle}\,dg=\frac{1}{\textbf{f}(\pi)}\langle v_{1},v_{3}\rangle \overline{\langle v_{2},v_{4}\rangle} .
\end{equation*}
\end{conj}

Assuming that the matrix coefficients involved are $K$-finite, one has the following result:

\begin{thm}\cite[Theorem 1.7]{KYD} \label{kfin} Let $G$ be a semisimple group defined over a local field $F$ and let $K$ be a maximal compact subgroup of $G$. Let $(\pi,H)$ be a tempered, irreducible, unitary representation of $G$ and let $H_{K}$ denote the space of $K$-finite vectors in $H$. Then there exists $\textbf{d}(\pi)\in\mathbb{Z}_{\geq0}$ such that:
\begin{enumerate}
\item[(1)] If $F$ is non-Archimedean, there exists $\textbf{f}(\pi)\in \mathbb{R}_{>0}$ such that, for all $v_{1},v_{2},v_{3},v_{4}\in H_{K}$, we have 
\begin{equation*}
\lim_{r\to\infty}\frac{1}{r^{\textbf{d}(\pi)}}\int_{G_{<r}}\langle \pi(g) v_{1},v_{2}\rangle \overline{\langle \pi(g)v_{3},v_{4}\rangle}\,dg=\frac{1}{\textbf{f}(\pi)}\langle v_{1},v_{3}\rangle \overline{\langle v_{2},v_{4}\rangle}.
\end{equation*}   
\item[(2)] If $F$ is Archimedean, for any given non-zero $v_{1},v_{2}\in H_{K}$, there exists $C(v_{1},v_{2})>0$ such that 
\begin{equation*}
\lim_{r\to\infty}\frac{1}{r^{\textbf{d}(\pi)}}\int_{G_{<r}} |\langle \pi(g)v_{1},v_{2}\rangle|^{2}\,dg=C(v_{1},v_{2}).
\end{equation*}
\end{enumerate}
\end{thm}

In the non-Archimedean case, the proof of (1) is achieved by first establishing the validity of the analogous version of (2). The polarisation identity allows the authors of \cite{KYD} to define a form 
\begin{equation*}
D(\cdot,\cdot,\cdot,\cdot):H_{K}\times H_{K}\times H_{K}\times H_{K}\longrightarrow \mathbb{C}
\end{equation*}
via the prescription 
\begin{equation*}
D(v_{1},v_{2},v_{3},v_{4}):=\lim_{r\to\infty}\frac{1}{r^{\textbf{d}(\pi)}}\int_{G_{<r}}\langle \pi(g) v_{1},v_{2}\rangle \overline{\langle \pi(g)v_{3},v_{4}\rangle}\,dg.
\end{equation*}
In \cite[Section 4.1]{KYD}, this form is shown to be $G$-invariant and one would like to invoke an appropriate form of  the Schur lemma to argue as in the standard proof of Schur orthogonality relations. That is, for fixed $v_{2},v_{4}\in H_{K}$, one defines the form 
\begin{equation*}
D(\cdot,v_{2},\cdot,v_{4}):H_{K}\times H_{K}\longrightarrow \mathbb{C},
\end{equation*}
and, for fixed $v_{1},v_{3}\in H_{K}$, the form 
\begin{equation*}
D(v_{1},\cdot,v_{3},\cdot):H_{K}\times H_{K}\longrightarrow \mathbb{C}.
\end{equation*}
One applies Schur lemma to these forms, which implies that each such form is a scalar multiple of the inner product on $H$. Upon comparing them, one obtains the desired orthogonality relations.\\

In the non-Archimedean case, it seems to us that the representations considered in \cite{KYD} are implicitly assumed to be smooth \cite[D\'efinition III.1.1]{Ren}, otherwise it is not clear how the theory of asymptotic expansion can be applied.\\

The appropriate version of the Schur lemma in this case is a consequence of Dixmier's lemma \cite[Lemma 0.5.2]{Wal}, which can be applied since in the non-Archimedean setting the subspace of $K$-finite vectors $H_{K}$ and the subspace of smooth vectors $H^{\infty}$ coincide: the latter is irreducible since $H$ itself is irreducible. The required countability of the dimension of $H_{K}$ follows from the admissibility \cite[Th\'eor\`eme VI.2.2]{Ren} of the irreducible smooth unitary representation $(\pi,H)$ and by invoking \cite[Lemma 0.5.2]{Wal} in the proof of \cite[III.1.9]{Ren}.\\

The purpose of this article is to prove that the analogue of (1) in Theorem \ref{kfin} holds in the Archimedean case. As explained in \cite[Section 4.2]{KYD}, it suffices to prove the result for real semisimple groups (Theorem \ref{main3}). 

\begin{thm}\label{Main} Let $(\pi,H)$ be a tempered, irreducible, unitary representation of a connected, semisimple Lie group $G$ with finite centre. Let $K$ be a maximal compact subgroup of $G$. Then there exists $\textbf{f}(\pi)\in\mathbb{R}_{>0}$ such that, for all $v_{1},v_{2},v_{3},v_{4}\in H_{K}$, we have 
\begin{equation*}
\lim_{r\to\infty}\frac{1}{r^{\textbf{d}(\pi)}}\int_{G_{<r}}\langle \pi(g) v_{1},v_{2}\rangle \overline{\langle \pi(g)v_{3},v_{4}\rangle}\,dg=\frac{1}{\textbf{f}(\pi)}\langle v_{1},v_{3}\rangle \overline{\langle v_{2},v_{4}\rangle}.
\end{equation*}
\end{thm}

\begin{rem} It is well-known that an irreducible, tempered representation as in Theorem \ref{Main} is unitarisable. We have chosen the somewhat redundant formulation above to emphasize that the unitarity of the 
representation plays a crucial role in the following. 
From now on, if $(\pi,H)$ is a tempered, irreducible representation we will implicitly assume that $\pi$ acts unitarily with respect to the inner product on $H$.  
\end{rem}

We need to modify the strategy above to account for the fact that the space of $K$-finite vectors of a unitary representation $(\pi,H)$ of a real semisimple group does not afford a representation of $G$. It is, however, an admissible $(\mathfrak{g},K)$-module. \\

Our approach relies crucially on the admissibility of irreducible, unitary representations of reductive Lie groups, a foundational theorem proved by Harish-Chandra. The theory of admissible $(\mathfrak{g},K)$-modules then provides us with the appropriate version of the Schur lemma for $(\mathfrak{g},K)$-invariant forms (Definition \ref{gkinvform}).

Hence, we are reduced to verify that $D(\cdot,v_{2},\cdot,v_{4})$ and $D(v_{1},\cdot,v_{3},\cdot)$ are, indeed, $(\mathfrak{g},K)$-invariant. Having established this, to conclude the proof of Theorem \ref{Main}, we can argue as in \cite[Section 4]{KYD}.  \\

From now on, to make the notation look more compact, given a unitary representation $(\pi,H)$ of $G$ and vectors $v,w\in H$, we set 
\begin{equation*} \phi_{v,w}(g):=\langle\pi(g)v,w\rangle.
\end{equation*}
For connected, semisimple Lie groups with finite centre, $K$-invariance is a consequence of $\mathfrak{g}$-invariance (Proposition \ref{ginvkinv}). Therefore, the problem is establishing the $\mathfrak{g}$-invariance. Explicitly, we prove the following (Proposition \ref{ginv}). 

\begin{prop}\label{ginv2} Let $G$ be a connected, semisimple Lie group with finite centre and let $(\pi,H)$ be a tempered, irreducible, unitary representation of $G$. Then, for all $X\in\mathfrak{g}$, and for all $v_{1},v_{2},v_{3},v_{4}\in H_{K}$, we have
\begin{equation*}
\lim_{r\rightarrow \infty}\frac{1}{r^{\textbf{d}(\pi)}}\int_{G_{<r}}\phi_{\dot{\pi}(X)v_{1},v_{2}}(g)\overline{\phi_{v_{3},v_{4}}(g)}\,dg=-\lim_{r\rightarrow \infty}\frac{1}{r^{\textbf{d}(\pi)}}\int_{G_{<r}}\phi_{v_{1},v_{2}}(g)\overline{\phi_{\dot{\pi}(X)v_{3},v_{4}}(g)}\,dg
\end{equation*}
and 
\begin{equation*}
\lim_{r\rightarrow \infty}\frac{1}{r^{\textbf{d}(\pi)}}\int_{G_{<r}}\phi_{v_{1},\dot{\pi}(X)v_{2}}(g)\overline{\phi_{v_{3},v_{4}}(g)}\,dg=-\lim_{r\rightarrow \infty}\frac{1}{r^{\textbf{d}(\pi)}}\int_{G_{<r}}\phi_{v_{1},v_{2}}(g)\overline{\phi_{v_{3},\dot{\pi}(X)v_{4}}(g)}\,dg.
\end{equation*}
\end{prop}

The key observation is that, by exploiting the theory of asymptotic expansions of matrix coefficients of tempered representations both with respect to a minimal parabolic subgroup $P=MAN$ and with respect to the standard (for $P$) parabolic subgroups of $G$, the expression 
\begin{equation*}
\lim_{r\rightarrow \infty}\frac{1}{r^{\textbf{d}(\pi)}}\int_{G_{<r}}\phi_{\dot{\pi}(X)v_{1},v_{2}}(g)\overline{\phi_{v_{3},v_{4}}(g)}\,dg
\end{equation*}
reduces, roughly, to a sum of finitely many terms of the form 
\begin{equation*}
\int_{K}\langle \Gamma_{\lambda,l}(m_{\lambda},\pi(k)\dot{\pi}(X)v_{1},w_{2}),\Gamma_{\mu,m}(m_{\lambda},\pi(k)v_{3},w_{4})\rangle_{L^{2}(M_{\lambda})}\,dk.
\end{equation*}
Here, $M_{\lambda}$ comes from a standard parabolic subgroup $P_{\lambda}=M_{\lambda}A_{\lambda_{0}}N_{\lambda_{0}}$ of $G$. We denote $\mathfrak{m}_{\lambda},\text{ }\mathfrak{a}_{\lambda_{0}},\text{ }\mathfrak{n}_{\lambda_{0}}$ the Lie algebras of $M_{\lambda},\text{ }A_{\lambda_{0}},\text{ }N_{\lambda_{0}}$, respectively. The pairs $(\lambda,l)$ and $(\mu,m)$ will be introduced precisely in Theorem \ref{asymp2}, we can think of $\lambda,\mu$ as $n$-tuples of complex numbers and of $l,m$ as $n$-tuples of integers. The functions $\Gamma_{\lambda,l}$, $\Gamma_{\mu,m}$ are defined in Equation \eqref{defgamma}. As functions of $m_{\lambda}$, they are analytic and square-integrable and they arise from the asymptotic expansion of the matrix coefficients $\phi_{\dot{\pi}(X)v_{1},v_{2}}$ and $\phi_{v_{3},v_{4}}$, respectively, relative to $P_{\lambda}$ (see Theorem \ref{asymp2}).
The subscript in $P_{\lambda}$ is meant to indicate that the parabolic subgroup is obtained, in an appropriate sense, from the datum of $\lambda$. Moreover, $(\lambda,l)$ and $(\mu,m)$ are related in a precise way (see the discussion after 
Equation~\eqref{asymp1} and the proof of Proposition~\ref{ginv} after Equation~\eqref{asympint2}). \\

We shall elaborate on these points later on. For the moment, let us point out that we reduced the initial problem to showing that, for every $X\in \mathfrak{g}$, and for all relevant pairs $(\lambda,l)$ and $(\mu,m)$, the integral
\begin{equation*}
\int_{K}\langle \Gamma_{\lambda,l}(m_{\lambda},\pi(k)\dot{\pi}(X)v_{1},w_{2}),\Gamma_{\mu,m}(m_{\lambda},\pi(k)v_{3},w_{4})\rangle_{L^{2}(M_{\lambda})}\,dk
\end{equation*}
equals 
\begin{equation*}
-\int_{K}\langle \Gamma_{\lambda,l}(m_{\lambda},\pi(k)v_{1},w_{2}),\Gamma_{\mu,m}(m_{\lambda},\pi(k)\dot{\pi}(X)v_{3},w_{4})\rangle_{L^{2}(M_{\lambda})}\,dk.
\end{equation*}
We will prove that, if $(\lambda,l)$ and $(\mu,m)$ satisfy a certain condition (to be explained below), the functions $\Gamma_{\lambda,l}(\cdot,v_{1},w_{2})$ and $\Gamma_{\mu,m}(\cdot,v_{3},w_{4})$ are, in fact,  $\mathrm{Z}(\mathfrak{g}_{\mathbb{C}})$-finite, with $\mathrm{Z}(\mathfrak{g}_{\mathbb{C}})$ denoting the centre of the universal enveloping algebra of the complexification $\mathfrak{g}_{\mathbb{C}}$ of $\mathfrak{g}$, and $K\cap M_{\lambda}$-finite. It will then follow from a theorem of Harish-Chandra (Theorem \ref{smoothvector}) that they are smooth vectors in the right-regular representation $(R,L^{2}(M_{\lambda}))$ of $M_{\lambda}$ . \\

The idea is to combine this observation with an appropriate form of the Frobenius reciprocity (Theorem \ref{Frobenius}), due to Casselman, to construct $(\mathfrak{g},K)$-invariant maps 
\begin{equation*}
T_{w_{2}}\colon H_{K}\longrightarrow \text{Ind}_{P_{\lambda},K_{\lambda}}(H_{\sigma},\lambda|_{\mathfrak{a}_{\lambda_{0}}}),\text{ }T_{w_{2}}(v)(k)(m_{\lambda}):=\Gamma_{\lambda,l}(m_{\lambda},\pi(k)v,w_{2})
\end{equation*}
and 
\begin{equation*}
T_{w_{2}}\colon H_{K}\longrightarrow \text{Ind}_{P_{\lambda},K_{\lambda}}(H_{\sigma},\lambda|_{\mathfrak{a}_{\lambda_{0}}}),\text{ }T_{w_{4}}(v')(k)(m_{\lambda}):=\Gamma_{\lambda,l}(m_{\lambda},\pi(k)v',w_{4}).
\end{equation*}
Here, the subgroup $\overline{P_{\lambda}}$ is the parabolic subgroup opposite to $P_{\lambda}$. The notation $\text{Ind}_{\overline{P_{\lambda}},K}(H_{\sigma},\lambda|_{\mathfrak{a}_{\lambda_{0}}})$ stands for the space of $K$-finite vectors in the representation induced from the $(\mathfrak{m_{\lambda}}\oplus \mathfrak{a}_{\lambda_{0}},K\cap M_{\lambda})$-module $$H_{\sigma}\otimes \mathbb{C}_{\lambda|_{\mathfrak{a}_{\lambda_{0}}}-\rho_{\lambda_{0}}}$$ for an appropriately chosen admissible, unitary, sub-representation $(\sigma,H_{\sigma})$ of $(R,L^{2}(M_{\lambda}))$. \\

To apply the required form of the Frobenius reciprocity, we need to show that the maps 
\begin{equation*}
S_{w_{2}}\colon H_{K}\longrightarrow H_{\sigma}\otimes \mathbb{C}_{\lambda|_{\mathfrak{a}_{\lambda_{0}}}-\rho_{\lambda_{0}}}, \text{ }S_{w_{2}}(v)(m_{\lambda})\colon =\Gamma_{\lambda,l}(m_{\lambda},v,w_{2})
\end{equation*}
and 
\begin{equation*}
S_{w_{2}}\colon H_{K}\longrightarrow H_{\sigma}\otimes \mathbb{C}_{\lambda|_{\mathfrak{a}_{\lambda_{0}}}-\rho_{\lambda_{0}}}, \text{ }S_{w_{4}}(v')(m_{\lambda}):=\Gamma_{\lambda,l}(m_{\lambda},v',w_{4})
\end{equation*}
descend to  $(\mathfrak{m_{\lambda}}\oplus\mathfrak{a}_{\lambda_{0}},K_{\lambda})$-equivariant maps on $H_{K}/\mathfrak{n}_{\lambda_{0}}H_{K}$. Establishing this result is the technical heart of the article.\\

Assuming it, the integral 
\begin{equation*}\int_{K}\langle \Gamma_{\lambda,l}(m_{\lambda},\pi(k)\dot{\pi}(X)v_{1},w_{2}),\Gamma_{\mu,m}(m_{\lambda},\pi(k)v_{3},w_{4})\rangle_{L^{2}(M_{\lambda})}\,dk
\end{equation*}
is nothing but 
\begin{equation*}
\langle \dot{\text{Ind}}_{\overline{P_{\lambda}}}(\sigma,\lambda|_{\mathfrak{a}_{\lambda_{0}}})(X)\Gamma_{\lambda,l}(m_{\lambda},v_{1},w_{2}),\Gamma_{\mu,m}(m_{\lambda},v_{3},w_{4})\rangle_{\text{Ind}_{\overline{P_{\lambda}}}(\sigma,\lambda|_{\mathfrak{a}_{\lambda_{0}}})},
\end{equation*}
where
\begin{equation*}
\langle\cdot,\cdot\rangle_{\text{Ind}_{\overline{P_{\lambda}}}(\sigma,\lambda|_{\mathfrak{a}_{\lambda_{0}}})}
\end{equation*}
is the inner product on $\text{Ind}_{\overline{P_{\lambda}}}(\sigma,\lambda|_{\mathfrak{a}_{\lambda_{0}}})$. We will see that this makes sense since the inducing data ensure unitarity. The sought equality will then follow from the skew-invariance of the inner product on a unitary representation with respect to the action of the Lie algebra.\\ 

To explain how the functions $\Gamma_{\lambda,l}(\cdot,v_{1},v_{2})$ and $\Gamma_{\mu,m}(\cdot,v_{3},v_{4})$ arise, we need to recall the main features of the asymptotic expansions of $K$-finite matrix coefficients of tempered representations. If $\phi_{v,w}$ is such a matrix coefficient, then its restriction to a certain region of the subgroup $A$ of a minimal parabolic subgroup $P=MAN$ of $G$ admits an asymptotic expansion which can be thought of as a sum indexed by a countable collection 
\begin{equation*}
\Lambda\colon =\{(\lambda,l)\}_{\lambda\in \mathcal{E},\text{ } l\in \mathbb{Z}^{n}_{\geq 0}:|l|\leq l_{0}}.
\end{equation*}
The set $\mathcal{E}$ is a collection of complex-valued real-linear functionals on $\text{Lie}(A)$ depending on $(\pi,H)$ and not on the particular choice of $v,w\in H_{K}$. It is the set of \textbf{exponents} of $(\pi,H)$. The number $n$ is the rank of $G$ and $l_{0}$, too, depends on $(\pi,H)$ only.\\

The term indexed by $(\lambda,l)$ is multiplied by a complex coefficient $c_{\lambda,l}(v,w)$. The choice of $v,w\in H_{K}$ determines the pairs in $\mathcal{C}$ for which $c_{\lambda,l}(v,w)\neq 0$. If $\lambda\in \mathcal{E}$, there exists at least a pair of $v,w\in H_{K}$ such that, for some $l\in \mathbb{Z}_{\geq 0}^{n}$ with $|l|\leq l_{0}$, we have $c_{\lambda,l}(v,w)\neq 0$.\\

For any standard (for $P$) parabolic subgroup $P'=M'A'N'$ of $G$, the restriction of the matrix coefficient $\phi_{v,w}$ to an appropriate region of $A'$ admits a similar asymptotic expansion. It can be thought of as a sum indexed by a countable collection 
\begin{equation*}
\Lambda'\colon =\{(\nu,q)\}_{\nu\in \mathcal{E'},\text{ }q\in \mathbb{Z}_{\geq 0}^{r}: |q|\leq q_{0}}.
\end{equation*}
Here, $r\leq n$ is the dimension of $A'$, the set $\mathcal{E}'$ consists of complex-valued real-linear functionals on $\text{Lie}(A')$. On regions on which both the expansion relative to $P$ and the expansion relative to $P'$ are meaningful, by comparing the two it turns out that the element in $\mathcal{E}'$ are precisely the restrictions to $\text{Lie}(A')$ of the elements in $\mathcal{E}$ and, making the appropriate identifications following from $A'\subset A$, each $q$ is the projection to $\mathbb{Z}_{\geq 0}^{r}$ of an $l$ appearing in the expansion relative to $P$.\\ 

While in the expansion relative to $P$ the term indexed by $(\lambda,l)$ is multiplied by the complex coefficient $c_{\lambda,l}(v,w)$, the term indexed by $(\nu,q)$ in the expansion relative to $P'$ is mutiplied by a real-analytic function 
\begin{equation}
c^{P'}_{\nu,q}(\cdot,v,w):M'\longrightarrow \mathbb{C}.
\end{equation}
We need one more piece of information to explain how $\Gamma_{\lambda,l}(\cdot,v_{1},v_{2})$ and $\Gamma_{\mu,m}(\cdot,v_{3},v_{4})$ arise: the construction of $\textbf{d}(\pi)$ in \cite{KYD}. The idea is as follows. We can think of $\lambda\in \mathcal{E}$ as an $n$-tuple of complex numbers $(\lambda_{1},\cdots,\lambda_{n})$. It can be shown that there exist a finite sub-collection $\mathcal{E}_{0}\subset \mathcal{E}$ such that, for every $\lambda\in \mathcal{E}$, there exists $\hat{\lambda}\in \mathcal{E}_{0}$ such that 
\begin{equation*}
\hat{\lambda}-\lambda\in \mathbb{Z}_{\geq 0}^{n}.
\end{equation*}
Moreover, any two distinct elements in $\mathcal{E}_{0}$ are integrally inequivalent: their difference does not belong to $\mathbb{Z}^{n}$. By a result of Casselman (Theorem \ref{tempcriterion}), for every $\hat{\lambda}\in \mathcal{E}_{0}$ and for every $i\in \{1,\dots,n\}$, we have 
\begin{equation*}
\text{Re}\hat{\lambda}_{i}\leq 0 ,
\end{equation*}
and it is clear that this holds for every $\lambda\in \mathcal{E}$.\\

For $(\lambda,l)\in \Lambda$, we introduce the set $I_{\lambda}:=\{i\in \{1,\dots,n\}|\text{Re}\lambda_{i}<0\}$, we define 
\begin{equation}\label{defdp}
\textbf{d}_{P}(\lambda,l):=|I^{c}_{\lambda}|+\sum_{i\in I^{c}_{\lambda}}2l_{i},
\end{equation}
and we take the maximum, $\textbf{d}_{P}$, as $(\lambda,l)$ ranges over all the pairs with $\lambda\in \mathcal{E}_{0}$.\\

We can proceed analogously for every standard parabolic $P'$ and obtain a non-negative integer $\textbf{d}_{P'}$. The maximum over all $P'$ is $\textbf{d}(\pi)$.\\

Now, given $\lambda\in \mathcal{E}_{0}$, identifying $I_{\lambda}$ with a subset of the simple roots determined by an order on the root system attached to the pair $(\mathfrak{g},\mathfrak{a})$, we can construct a standard (for $P$) parabolic subgroup $P_{\lambda}=M_{\lambda}A_{\lambda_{0}}N_{\lambda_{0}}$ associated to $I_{\lambda}$. We will show that if $(\lambda,l)\in \Lambda$ satisfies $\lambda\in \mathcal{E}_{0}$ and $\textbf{d}_{P}(\lambda,l)=\textbf{d}(\pi)$, then $\Gamma_{\lambda,l}(\cdot,v_{1},v_{2})$ is precisely the function $c^{P'}_{\nu,q}(\cdot,v_{1},v_{2})$ with $\nu:=\lambda|_{\mathfrak{a}_{\lambda_{0}}}$, where $\mathfrak{a}_{\lambda_{0}}:=\text{Lie}(A_{\lambda_{0}})$, and $q$ equal to the projection of $l$ to $\mathbb{Z}^{I^{c}_{\lambda}}_{\geq 0}$.\\ 

Finally, we mentioned that in the integral
\begin{equation*}
\int_{K}\langle \Gamma_{\lambda,l}(m_{\lambda},\pi(k)v_{1},w_{2}),\Gamma_{\mu,m}(m_{\lambda},\pi(k)v_{3},w_{4})\rangle_{L^{2}(M_{\lambda})}\,dk
\end{equation*}
the pairs $(\lambda,l)$ and $(\mu,m)$ must be related in a precise way. First of all, $(\mu,m)\in \Lambda$ satisfies $\mu\in \mathcal{E}_{0}$ and $\textbf{d}_{P}(\mu,m)=\textbf{d}(\pi)$. In addition, we must have $I_{\lambda}=I_{\mu}$ (so that $P_{\lambda}=P_{\mu}$) and $\lambda|_{\mathfrak{a}_{\lambda_{0}}}=\mu|_{\mathfrak{a}_{\lambda_{0}}}$. The last condition, together with the unitarity of the representation $(\sigma,H_{\sigma})$ introduced above, is precisely what ensures that $\text{Ind}_{P_{\lambda}}(\sigma,\lambda|_{\mathfrak{a}_{\lambda_{0}}})$ is unitary.\\

Implementing the strategy sketched above requires gathering a number of intermediate results. Several are inspired from the chapter in \cite{Knapp} on the Langlands classification of tempered representations. Here is a more detailed outline of the article.  \\

\textbf{Section~\ref{sec:2}:} The first part includes a discussion of the $(\mathfrak{g},K)$-module version of the Schur lemma (Corollary~\ref{Schur}). In the second part, we recall the result of Harish-Chandra establishing that smooth, $\mathrm{Z}(\mathfrak{g}_{\mathbb{C}})$-finite, $K$-finite, square-integrable functions on reductive groups are smooth vectors in the right-regular representation (Theorem~\ref{smoothvector}). As a consequence, we prove that, on such a function, the action of $\mathfrak{g}$ through differentiation is the same as the action of the Lie algebra through the right-regular representation (Proposition \ref{skewinv}). After stating the basic facts on parabolically induced representations that we need, we discuss Casselman's version of the Frobenius reciprocity 
(Theorem \ref{Frobenius}). \\

\textbf{Section~\ref{sec:3}:} In the first part, we recall the theory of asymptotic expansions of matrix coefficients of tempered representations both with respect to a minimal parabolic subgroup and with respect to standard parabolic subgroups. We then explain in detail how the functions $\Gamma_{\lambda,l}(\cdot,v_{1},v_{2})$, $\Gamma_{\mu,m}(\cdot,v_{3},v_{4})$ arise. We begin by introducing an equivalence relation on the data indexing the asymptotic expansion relative to $P$ of the $K$-finite matrix coefficients of a tempered, irreducible,  representation $(\pi,H)$. This equivalence relation is motivated by the construction of $\textbf{d}(\pi)$ in \cite{KYD} and it is meant to exploit the criteria for the computation of asymptotic integrals in \cite[Appendix A]{KYD}. Imposing the conditions on $(\lambda,l)$ and $(\mu,m)$ that we discussed above, we identify the functions $\Gamma_{\lambda,l}(\cdot,v_{1},v_{2})$ and $\Gamma_{\mu,m}(\cdot,v_{3},v_{4})$ with the coefficient functions in the asymptotic expansion relative to $P_{\lambda}$ of $\phi_{v_{1},v_{2}}$ and $\phi_{v_{3},v_{4}}$ (Proposition \ref{Gamma}). We then prove that they are smooth vectors in $(R,L^{2}(M_{\lambda}))$ (Proposition \ref{GammaGar}). Combining Proposition \ref{GammaGar} with the technical Lemma \ref{computationa} and 
Lemma \ref{computationn}, we can construct unitary, admissible, finitely generated representations $(\sigma_{1}, H_{\sigma_{1}})$ and $(\sigma_{2},H_{\sigma_{2}})$  whose direct sum is the unitary, admissible, finitely generated representation $(\sigma,H_{\sigma})$ introduced above (Proposition \ref{applyfrob}).     \\   

\textbf{Section~\ref{sec:4}:} Having gathered the results we need, we prove Proposition \ref{ginv2} (Proposition \ref{ginv}). This consists in an application of the considerations in \cite[Appendix A]{KYD} to show that the integral 
\begin{equation*}
\lim_{r\rightarrow \infty }\frac{1}{r^{\textbf{d}(\pi)}}\int_{G_{<r}}\phi_{v_{1},v_{2}}(g)\overline{\phi_{v_{3},v_{4}}(g)}\,dg
\end{equation*}
can be computed in terms of a sum of integrals of the form 
\begin{equation*}
\int_{K}\langle \Gamma_{\lambda,l}(m_{\lambda},\pi(k)v_{1},w_{2}),\Gamma_{\mu,m}(m_{\lambda},\pi(k)v_{3},w_{4}) \rangle\,dk
\end{equation*}
with the pairs $(\lambda,l)$ and $(\mu,m)$ both belonging to $\Lambda$ with $\lambda,\mu\in \mathcal{E}_{0}$, $I_{\lambda}=I_{\mu}$, $\lambda|_{\mathfrak{a}_{\lambda_{0}}}=\mu|_{\mathfrak{a}_{\lambda_{0}}}$ and
\begin{equation*}
\textbf{d}_{P}(\lambda,l)=\textbf{d}_{P}(\mu,m)=\textbf{d}(\pi).
\end{equation*}

At this point, the representation theoretic arguments explained in the Introduction and proved in Section~\ref{sec:3} conclude the proof of Proposition \ref{ginv2}.\\

Finally, we proceed as explained in the first part of the Introduction to prove Theorem \ref{Main} (Theorem \ref{main3}).\\

\textbf{Acknowledgements:} We would like to thank the referees for pointing out a serious gap in a previous version of the article and for suggesting the improvements that led to this final version. The second author began to work on this article as a Doctoral Researcher at the University of Luxembourg and completed it as a Postdoctoral Researcher at the IASM of Zhejiang University. He is grateful to his Ph.D. supervisor, G. Wiese, for his support; to J. R. Getz and G. Chenevier for their comments; and to the members of the IASM for their generous welcome.

 
\section{Recollections on Representation Theory} \label{sec:2}

Our presentation of the theory of $(\mathfrak{g},K)$-modules follows \cite{Wal}. To discuss its basic features, we need to gather some results on unitary representations of compact groups. We begin by recalling the basic notions in the study of representations of topological groups, which we always assume to be Hausdorff. \\

First, following \cite[Section 1.1]{Wal}, let $G$ denote a second-countable, locally compact group, equipped with a left Haar measure $\,dg$, and let $V$ denote a complex topological vector space. We denote by $\mathrm{GL}(V)$ the group of invertible continuous endomorphisms of $V$. A \textbf{representation} of $G$ on $V$ is a strongly continuous homomorphism $\pi\colon G\longrightarrow \mathrm{GL}(V)$. Let $(\pi,V)$ denote the datum of a representation of $G$. A subspace of $V$ which is stable under the action of $G$ through $\pi$ is called an \textbf{invariant subspace}. A representation $(\pi,V)$, with $V\neq {0}$, is said to be \textbf{irreducible} if the only closed invariant subspaces are the trivial subspace and $V$ itself. \\

If $(H,\langle \cdot,\cdot\rangle)$ is a separable Hilbert space, a representation $\pi$ of $G$ on $H$ is termed a \textbf{Hilbert representation}. If, in addition, $G$ acts by unitary operators through $\pi$, the representation is said to be \textbf{unitary}. \\

Next, following \cite[Section 10]{KL}, we introduce the basic features of the theory of vector-valued integration. 

Let $(X,\,dx)$ be a Radon measure space, let $H$ be a Hilbert space and assume that 
\begin{equation*}
f\colon X\longrightarrow H
\end{equation*}
is measurable. The function $f$ is \textbf{integrable} if it satisfies the following two conditions:
\begin{enumerate}
\item[(1)] For all $v\in H$, $$\int_{X} |\langle f(x), v\rangle|\,dx<\infty.$$
\item[(2)] The map $$v\mapsto \int_{X} \langle f(x), v\rangle\,dx$$ is a bounded conjugate-linear functional.
\end{enumerate}

If $f\colon X\longrightarrow H$ is integrable, then, by the Riesz representation theorem, there exists a unique element in $H$, denoted 
\begin{equation*}
\int_{X}f(x)\,dx,
\end{equation*}
such that, for all $v\in H$, we have 
\begin{equation*}
\Bigl \langle\int_{X}f(x)\,dx,v\Bigr \rangle=\int_{X}\langle f(x),v\rangle\,dx.
\end{equation*}

\begin{prop} Let $(X,\,dx)$ be as above. Let $H$, $E$ be Hilbert spaces,  $f\colon X\longrightarrow H$ a measurable function and $T\colon H\longrightarrow E$ a bounded linear operator. Then the following holds:
\begin{enumerate} 
\item[(1)] If 
\begin{equation*}
\int_{X}\|f(x)\|\,dx<\infty,
\end{equation*}
then $f\colon X\longrightarrow H$ is integrable.
\item[(2)] If $f\colon X\longrightarrow H$ is integrable, then so is $Tf\colon X\longrightarrow E$. Moreover, 
\begin{equation*}
T \left( \int_{X}f(x)\,dx \right)=\int_{X}Tf(x)\,dx.
\end{equation*}
\end{enumerate}
\end{prop}
\begin{proof}
See \cite[Proposition 10.8 and Proposition 10.9]{KL}.
\end{proof}

Now, let $(\pi,H)$ be a unitary representation of $G$. Let $v\in H$ and $f\colon G\longrightarrow H$ be such that the map 
\begin{equation*} 
g\mapsto f(g)\pi(g)v
\end{equation*}
is integrable. Let $\pi(f)v$ denote the unique element in $H$ such that, for all $w\in H$, we have 
\begin{equation*}
\langle\pi(f)v,w\rangle=\int_{G}f(g)\langle\pi(g)v,w\rangle\,dg.
\end{equation*}

\begin{prop}\label{intop} Let $(\pi,H)$ be as above. If $f\in L^{1}(G)$, then, for all $v\in H$, the map $g\mapsto f(g)\pi(g)v$ is integrable and the prescription $$\pi(f)\colon H\longrightarrow H, \text{ } v\mapsto \pi(f)v$$ defines a bounded linear operator.
\end{prop}
\begin{proof} See \cite[Proposition 10.20]{KL}.
\end{proof}

With the integral operators introduced in Proposition \ref{intop} at our disposal, we have all the tools needed to state the main results on the unitary representations of compact groups. \\

Let $K$ be a compact group. Let $\widehat{K}$ denote the set of equivalence classes of irreducible unitary representations of $K$. If $(\pi,H)$ is a unitary representation, for each $[\gamma]\in \widehat{K}$ let $H(\gamma)$ denote the closure of the sum of all the closed invariant subspaces of $H$ in the equivalence class of $\gamma$. We refer to $H(\gamma)$ as the $\gamma$-isotypic component of $H$. This notion is independent of the choice of representative for the equivalence class.

\begin{prop}\label{findim} Let $K$ be a compact group. Let $(\pi,H)$ be an irreducible unitary representation of $K$. Then $H$ is finite-dimensional. 
\end{prop}
\begin{proof} See \cite[Proposition 1.4.2]{Wal}.
\end{proof}

Given Proposition \ref{findim}, we can associate, to each irreducible representation $\gamma$ of $K$, the function 
\begin{equation*}
\chi_{\gamma}\colon K\longrightarrow \mathbb{C},\text{ }\chi_{\gamma}(g):=\mathrm{tr}\gamma(g),
\end{equation*}
the \textbf{character} of $\gamma$. A standard argument proves that equivalent representations have the same character. \\

Recall that if $\{(\pi_{i},H_{i})| i\in I\}$ is a countable family of unitary representations of a topological group $G$, we can construct a new unitary representation of $G$, the \textbf{direct sum}, on the Hilbert space completion of the algebraic direct sum of the $H_{i}$'s. We refer the reader to \cite[Section 1.4.1]{Wal}, for the details of this construction. We let 
\begin{equation*}
\bigoplus_{i\in I}H_{i}
\end{equation*}
denote the direct sum of the family $\{(\pi_{i},H_{i})| i\in I\}$, dropping explicit reference to the $\pi_{i}$'s. 

\begin{prop}\label{PW} Let $K$ be a compact group. Let $(\pi,H)$ be a unitary representation of $K$. Then $(\pi,H)$ is the direct sum representation of its $K$-isotypic components; that is,
\begin{equation*}
H=\bigoplus_{[\gamma]\in \widehat{K}}H(\gamma).
\end{equation*}
 Moreover, let $\alpha_{\gamma}$ denote the function 
\begin{equation*}
\alpha_{\gamma}(k):=\text{dim}(\gamma)\overline{\chi_{\gamma}(k)}.
\end{equation*}
Then the following holds: 
\begin{equation*}
H(\gamma)=\pi(\alpha_{\gamma})H.
\end{equation*}
\end{prop}
\begin{proof} See \cite[Lemma 1.4.7]{Wal}.
\end{proof}

\begin{prop}\label{unitarian} Let $K$ be a compact group. If $(\pi,H)$ is a Hilbert space representation of $K$, then there exists an inner product on $H$ that induces the original topology on $H$ and for which $K$ acts unitarily through $\pi$.
\end{prop}
\begin{proof} See \cite[Lemma 1.4.8]{Wal}.
\end{proof}

We are finally ready to introduce $(\mathfrak{g},K)$-modules.  

\begin{definition} Let $G$ be a connected, semisimple Lie group with finite centre. Let $\mathfrak{g}$ denote its Lie algebra. Let $K$ be a maximal compact subgroup of $G$, which we fix from now on, with Lie algebra $\mathfrak{k}$. A vector space $V$, equipped with the structure of $\mathfrak{g}$-module and $K$-module, is called a $(\mathfrak{g},K)$-\textbf{module} if the following conditions hold:
\begin{enumerate}
\item[(1)] For all $v\in V$, for all $X\in \mathfrak{g}$, for all $k\in K$, 
\begin{equation*}
kXv=\text{Ad}(k)Xkv
\end{equation*}
\item[(2)] For all $v\in V$, the span of the set 
\begin{equation*}
Kv:=\{kv|k\in K\}
\end{equation*}
is a finite-dimensional
subspace of $V$, on which the action of $K$ is continuous.
\item[(3)] For all $v\in V$, for all $Y\in\mathfrak{k}$, 
\begin{equation*}
\frac{d}{dt}\mathrm{exp}(tY)v|_{t=0}=Yv.
\end{equation*}
\end{enumerate}
\end{definition}

We remark that (3) implicitly uses the smoothness of the action of $K$ on the span of $Kv$. This follows from the fact that a continuous group homomorphism between Lie groups is automatically smooth. \\

Let $V$ and $W$ be $(\mathfrak{g},K)$-modules and let $\text{Hom}_{\mathfrak{g},K}(V,W)$ denote the space of $\mathfrak{g}$-morphisms that are also $K$-equivariant. Then $V$ and $W$ are said to be \textbf{equivalent} if $\text{Hom}_{\mathfrak{g},K}(V,W)$ contains an invertible element.\\

A $(\mathfrak{g}, K)$-module $V$ is called \textbf{irreducible} if the only subspaces that are invariant under the actions of $\mathfrak{g}$ and $K$ are the trivial subspace and $V$ itself. In this case, we have the following theorem:

\begin{thm}\label{g,k} Let $V$ be an irreducible $(\mathfrak{g},K)$-module. Then the space  $\mathrm{Hom}_{\mathfrak{g},K}(V,V)$ is $1$-dimensional.   
\end{thm}
\begin{proof} This is the result actually proved in \cite[Lemma 3.3.2]{Wal}, although the statement there says $\text{Hom}_{\mathfrak{g},K}(V,W)$, for an unspecified $W$. We believe it is a typo. 
\end{proof}

Let $V$ be a $(\mathfrak{g},K)$-module. Since, given each $v\in V$, the span of $Kv$, say $W_{v}$, is a finite-dimensional continuous representation of $K$, we can use Proposition \ref{unitarian} and then apply Proposition \ref{PW}, thus decomposing $W_{v}$ into a finite sum of finite-dimensional $K$-invariant subspaces of $V$. For $\gamma\in \widehat{K}$, we let $V(\gamma)$ denote the sum of all the $K$-invariant finite dimensional subspaces in the equivalence class of $\gamma$. Then the discussion above implies that 
\begin{equation*}
V=\bigoplus_{\gamma\in \widehat{K}}V(\gamma)
\end{equation*}
as a $K$-module, with the direct sum indicating the algebraic direct sum. A $(\mathfrak{g},K)$-module $V$ is called \textbf{admissible} if, for all $\gamma\in \widehat{K}$, $V(\gamma)$ is finite-dimensional. \\

Given a unitary representation $(\pi,H)$, there exists a $(\mathfrak{g},K)$-module naturally associated to it. To define it, recall that a vector $v\in H$ is called \textbf{smooth} if the map 
\begin{equation*}
g\mapsto \pi(g)v
\end{equation*}
is smooth. Let $H^{\infty}$ denote the subspace of smooth vectors of $H$. It is a standard fact that the prescription 
\begin{equation*} 
\dot{\pi}(X):=\frac{d}{dt}\pi(\mathrm{exp}(tX))v|_{t=0}, 
\end{equation*} 
for $v\in H^{\infty}$ and $X\in \mathfrak{g}$, defines an action of $\mathfrak{g}$ on $H^{\infty}$. Recall that a vector $v\in H$ is $K$-\textbf{finite} if the span of the set 
\begin{equation*}
\pi(K)v:=\{\pi(k)v|k\in K\}
\end{equation*}
is finite-dimensional. Let $H_{K}$ denote the subspace of $K$-finite vectors of $H$. By \cite[Lemma 3.3.5]{Wal}, with the action of $\mathfrak{g}$ so defined and with the action of $K$ through $\pi$, the space $H_{K}\cap H^{\infty}$ is a $(\mathfrak{g},K)$-module. The representation $(\pi,H)$ is said to be \textbf{admissible} if $H_{K}\cap H^{\infty}$ is admissible as a $(\mathfrak{g},K)$-module and $(\pi,H)$ is called \textbf{infinitesimally irreducible} if $H_{K}\cap H^{\infty}$ is irreducible as a $(\mathfrak{g},K)$-module.  It is in general not true that a $K$-finite vector is smooth. However, if $(\pi,H)$ is admissible, we have the following result:

\begin{thm}\label{Ksmooth} Let $G$ be a connected, semisimple Lie group with finite centre. Let $(\pi,H)$ be an admissible representation of $G$. Then every $K$-finite vector is smooth.
\end{thm}
\begin{proof} See the proof \cite[ Theorem 3.4.10]{Wal}.
\end{proof}

In light of the following fundamental result of Harish-Chandra, Theorem~\ref{Ksmooth} will play an important role in this article.

\begin{thm}\label{HC} Let $G$ be a connected, semisimple Lie group with finite centre. Let $(\pi,H)$ be an irreducible, Hilbert representation of $G$. Then $(\pi,H)$ is admissible. 
\end{thm} 
\begin{proof} See \cite[Theorem 7.204]{KnappVogan}. 
\end{proof}

In the following, given a unitary representation $(\pi,H)$, we will write $H_{K}$ for the $(\mathfrak{g},K)$-module $H_{K}\cap H^{\infty}$ even if $(\pi,H)$ is not admissible. We believe it will not cause any confusion.

We are now in position to prove the version of the Schur lemma for sesquilinear forms that we will use in Section~\ref{sec:3}. It is given as Corollary~\ref{Schur} below. First, we need:

\begin{thm}\label{irred} Let $G$ be a connected, semisimple Lie group with finite centre. Let $(\pi,H)$ be an admissible Hilbert representation of $G$. Then $(\pi,H)$ is irreducible if and only if it is infinitesimally irreducible.
\end{thm}
\begin{proof} See \cite[Theorem 3.4.11]{Wal}.
\end{proof}

\begin{definition}\label{gkinvform} Let $V$ and $W$ be $(\mathfrak{g},K)$-modules. A sesquilinear form 
\begin{equation*}
B(\cdot,\cdot):V\times W\longrightarrow \mathbb{C}
\end{equation*} 
is $(\mathfrak{g},K)$-invariant if it satisfies the following two conditions:
\begin{enumerate}
\item[(i)] For all $k_{1},k_{2}\in K$ and all $v,w\in V$ we have  
\begin{equation*}
B(k_{1}v,k_{2}w)=B(v,w).
\end{equation*}
\item[(ii)] For all $X\in \mathfrak{g}$ and all $v,w\in V$ we have 
\begin{equation*}
B(Xv,w)=-B(v,Xw).
\end{equation*}
\end{enumerate}
\end{definition}

\begin{thm}\label{dual} Let $G$ be a connected, semisimple Lie group with finite centre. Let $V$ be an admissible $(\mathfrak{g},K)$-module. Suppose that there exist a $(\mathfrak{g},K)$-module $W$ and a non-degenerate $(\mathfrak{g},K)$-invariant sesquilinear form 
\begin{equation*}
B(\cdot,\cdot):V\times W\longrightarrow \mathbb{C}.
\end{equation*}
Then $W$ is $(\mathfrak{g},K)$-isomorphic to $\overline{V}$.
\end{thm}
\begin{proof} This is \cite[Lemma 4.5.1]{Wal},  except for the fact that our form is sesquilinear. To account for it, we modify the definition of the map $T$ in the reference by setting, for a given $w\in W$, $T(w)(v)=B(w,v)$ for all $v\in V$. This defines a map from $W$ to $\overline{V}$ obtained by sending $w$ to $T(w)$ which, by the argument in the reference, is a $(\mathfrak{g},K)$-isomorphism. 
\end{proof}

The next corollary is proved by adapting to our case the argument in \cite[Proposition 8.5.12]{Gold} and using the beginning of the proof of \cite[Proposition 9.1]{Knapp}.

\begin{cor}\label{Schur} Let $G$ be a connected, semisimple Lie group with finite centre. Let $(\pi,H)$ be an irreducible, Hilbert representation of $G$. Then, up to a constant, there exists at most one non-zero $(\mathfrak{g},K)$-invariant sesquilinear form on $H_{K}$. In particular, if $(\pi,H)$ is irreducible unitary, then every such form is a constant multiple of $\langle\cdot,\cdot\rangle$. 
\end{cor}
\begin{proof} The irreducibility of $(\pi,H)$ implies that of $H_{K}$, by Theorem \ref{irred} and by Theorem \ref{Ksmooth}. Let $B(\cdot,\cdot)$ be a $(\mathfrak{g},K)$-invariant sesquilinear form. Consider the linear subspace $V_{0}$ of $H_{K}$ defined as $$V_{0}:=\{v\in H_{K}| B(v,w)=0 \text{ for all }w\in H_{K}\}.$$ 
Since $B(\cdot,\cdot)$ is non-zero, $V_{0}$ is a proper subspace of $H_{K}$. Since $B(\cdot,\cdot)$ is moreover $(\mathfrak{g},K)$-invariant, it follows that $V_{0}$ is a $(\mathfrak{g},K)$-invariant subspace of $H_{K}$, hence, by the irreducibility of $H_{K}$, it must be zero. Analogous considerations for the subspace 
\begin{equation*}
V^{0}:=\{w\in H_{K}| B(v,w)=0 \text{ for all } v\in H_{K}\}
\end{equation*}
imply that $B(\cdot,\cdot)$ is non-degenerate. By Theorem \ref{dual}, the map $v\mapsto T(v)$, $T(v)(\cdot):=B(v,\cdot)$, is a $(\mathfrak{g},K)$-isomorphism. Since $H_{K}$ is irreducible, the space $\text{Hom}_{\mathfrak{g},K}(H_{K},H_{K})$ is $1$-dimensional by Theorem \ref{g,k}. Now, let $B'(\cdot,\cdot)$ be another such form, with associated isomorphism $T'$. Then $T(T')^{-1}=cI$, for some $c\in\mathbb{C}$. For the last statement, the unitarity of $(\pi,H)$ implies that $\langle\cdot,\cdot\rangle$ is a $(\mathfrak{g},K)$-invariant non-degenerate sesquilinear form and Theorem \ref{HC}, with the discussion above, implies the result. 
\end{proof}

Since we are assuming that $G$ is connected, proving $(\mathfrak{g},K)$-invariance reduces to proving $\mathfrak{g}$-invariance. Indeed, by \cite[Theorem 2.2, p. 256]{Helg}, any maximal compact subgroup $K$ of $G$ is connected. Therefore, by  \cite[Corollary 4.48]{Knapp2}, the exponential map 
\begin{equation*}
\text{exp}\colon \mathfrak{k}\longrightarrow K
\end{equation*}
is surjective. 

\begin{prop}\label{ginvkinv} Let $G$ be a connected, semisimple Lie group with finite centre. Let $V$ be a $(\mathfrak{g},K)$-module, let 
\begin{equation*}
B(\cdot,\cdot):V\times V\longrightarrow \mathbb{C}
\end{equation*}
be a $\mathfrak{g}$-invariant sesquilinear form. Then $B(\cdot,\cdot)$ is $K$-invariant.
\end{prop}
\begin{proof} Given any pair of vectors $v,w\in V$, we can find a finite-dimensional subspace of $V$, say $W$, which contains both and on which $K$ acts continuously through a representation $\pi$. The restriction of the bilinear form $B(\cdot,\cdot)$ to $W$ is continuous.
To prove that $B(\pi(k)v,\pi(k)w)=B(v,w)$ for all $k\in K$, it suffices to prove that $B(\pi(k)v,w)=B(v,\pi(k^{-1})w)$ for all $k\in K$. Given $k\in K$, let $X\in\mathfrak{k}$ be such that $k=\text{exp}(X)$. We begin by writing 
\begin{equation*}
B(\pi(k)v,w)=B(\pi(\text{exp}X)v,w).
\end{equation*}

Since $\pi(\text{exp}X)=\text{exp}\dot{\pi}(X)v$, we obtain 
\begin{equation*}
B(\pi(\text{exp}X)v,w)=B(\text{exp}\dot{\pi}(X)v,w).
\end{equation*}
The continuity of $B(\cdot,\cdot)$ on $V$ gives 
\begin{equation*}
B(\text{exp}\dot{\pi}(X)v,w)=\text{exp}B(\dot{\pi}(X)v,w).
\end{equation*}
By the $\mathfrak{g}$-invariance of $B(\cdot,\cdot)$, we have 
\begin{equation*}
\text{exp}B(\dot{\pi}(X)v,w)=\text{exp}B(v,\dot{\pi}(-X)w)
\end{equation*}
and, finally, 
\begin{equation*}
\text{exp}B(v,\dot{\pi}(-X)w)=B(v,\pi(\text{exp}(-X))w).
\end{equation*}
This is precisely
\begin{equation*} 
B(\pi(k)v,w)=B(v,\pi(k^{-1})w).
\end{equation*}
\end{proof}

Let us recall that any locally compact Hausdorff group $G$ acts on the Hilbert space $L^{2}(G)$ by the prescription 
\begin{equation*}
R(g)f(x):=f(xg).
\end{equation*}
The representation so obtained is unitary and if $G$ is a Lie group the notion of smooth vectors in $L^{2}(G)$ makes sense. In the next section, we will need a criterion to establish that certain functions are smooth vectors in $L^{2}(G)$. We will make use of the following notion:
\begin{definition} Let $G$ be a Lie group and let $(\pi,H)$ be a Hilbert representation of $G$. The \textbf{G\aa rding subspace} of $H$ is the vector subspace of $H$ spanned by the set 
\begin{equation*}
\{\pi(f)v|v\in H,\text{ } f\in C^{\infty}_{c}(G)\}.
\end{equation*}
\end{definition}

\begin{prop}\label{garsmooth} Let $G$ be a Lie group with finitely many connected components, let $(\pi,H)$ be a Hilbert representation of $G$. Then every vector in the G\aa rding subspace of $H$ is a smooth vector in $H$.
\end{prop}
\begin{proof} See \cite[Lemma 1.6.1]{Wal}.
\end{proof}

Recall that $f\in C^{\infty}(G)$ is called $\mathrm{Z}(\mathfrak{g}_{\mathbb{C}})$-\textbf{finite} if it is annihilated by an ideal of $\mathrm{Z}(\mathfrak{g}_{\mathbb{C}})$ of finite codimension. The criterion we need is the following result of Harish-Chandra:

\begin{thm}\label{smoothvector} Let $G$ be a group in the class $\mathcal{H}$ as in \cite[p. 192]{V}. Let $f\in C^{\infty}(G)$ be $K$-finite and $\mathrm{Z}(\mathfrak{g}_{\mathbb{C}})$-finite. Then there exists a function $h\in C_{c}^{\infty}(G)$ which satisfies $h(kgk^{-1})=h(g)$ for all $k\in K$ and for all $g\in G$ and such that $f*h=f$. If $f\in C^{\infty}(G)$, in addition, is square-integrable, then $f$ is a smooth vector in $L^{2}(G)$. 
\end{thm}
\begin{proof} The first statement is \cite[Proposition 14, p.~352]{V}.
The second conclusion follows from the observation found at the beginning of the proof of  \cite[Corollary 8.42]{Knapp} that $f$ is in the G\aa rding subspace of $L^{2}(G)$ and it is therefore smooth by Proposition\ref{garsmooth}. That $f$ is indeed in the G\aa rding subspace of $L^{2}(G)$ follows from the standard fact that  

\begin{equation}\label{Rconv}
R(\tilde{\psi})f=f*\psi,
\end{equation}
for every $\psi\in C^{\infty}_{c}(G)$. Here, $\tilde{\psi}(x):=\psi(x^{-1})$. The first statement then gives
\begin{equation} 
R(\tilde{h})f=f*h=f,
\end{equation}
concluding the proof.
\end{proof}

\begin{prop}\label{skewinv} Let $G$ be a group in the class $\mathcal{H}$. Let $f\in C^{\infty}(G)$ be $K$-finite, $\mathrm{Z}(\mathfrak{g}_{\mathbb{C}})$-finite and square-integrable. Then, for every $X\in \mathfrak{g}$, we have 
\begin{equation*} 
Xf=\dot{R}(X)f,
\end{equation*}
where $Xf\colon G\longrightarrow \mathbb{C}$ is defined as 
\begin{equation} Xf(g):=\frac{d}{dt}\left[f(g\text{exp}(tX))\right]|_{t=0}.
\end{equation}
\end{prop}
\begin{proof} By Theorem \ref{smoothvector}, there exists $h\in C_{c}^{\infty}(G)$ such that 
\begin{equation*}
f=f*h.
\end{equation*}
From the equalities 
\begin{equation*}
Xf=X(f*h)=f*Xh
\end{equation*}
and 
\begin{equation*}
f*Xh=\dot{R}(\widetilde{Xh})f,
\end{equation*}
the latter being an application of \ref{Rconv}, we obtain 
\begin{equation*}
Xf=\dot{R}(\widetilde{Xh})f.
\end{equation*}
Since 
\begin{equation*}
\dot{R}(\widetilde{Xh})f=\dot{R}(X)R(\Tilde{h})f
\end{equation*}
and 
\begin{equation*}
R(\Tilde{h})f=f*h=f,
\end{equation*}
we conclude
\begin{equation*}
Xf=\dot{R}(X)f.
\end{equation*}
\end{proof}

We will apply Proposition \ref{skewinv} to the group $M$ in the Langlands decomposition of a parabolic subgroup $P=MAN$ of a connected semisimple Lie group with finite centre. A group $M$ of this form will not be connected, semisimple in general. However, it belongs to the class $\mathcal{H}$ by \cite[Lemma 9, p. 108]{Harish-Chandra}.\\

We briefly recall the construction of parabolically induced representations. We refer the reader to \cite[Chapter XI]{KnappVogan}, for a more thorough account. \\

Let $G$ be a connected, semisimple Lie group with finite centre and let $P=MAN$ be a parabolic subgroup of $G$. The group $K_{M}:=K\cap M$ is a maximal compact subgroup of $M$. Let $\lambda$ be a complex-valued real-linear functional on $\mathfrak{a}$ and let $(\sigma,H_{\sigma})$ be a Hilbert representation of $M$. We define an action of $G$ on the space of functions 
\begin{equation*}\{f\in C(K,H_{\sigma})|\text{ }f(mk)=\sigma(m)f(k)\text{ for all }m\in K_{M}\text{ and all }k\in K\}
\end{equation*}
by declaring 
\begin{equation*}
\text{Ind}_{P}(\sigma,\lambda,g)f(k):=e^{(\lambda+\rho)(\textbf{h}(kg))}\sigma(\textbf{m}(kg))f(\textbf{k}(kg)),
\end{equation*}
where, if $g=kman$ for some $k\in K$, $m\in M$, $a\in A$, $n\in N$, we set $\textbf{k}(g):=k$, $\textbf{m}(g):=m$, $\textbf{h}(g):=\text{log}(a)$, $\textbf{n}(g):=n$. The symbol $\rho$ denotes half of the sum of the positive restricted roots determined by $\mathfrak{a}$ counted with multiplicities. On this space of functions, we introduce the norm 
\begin{equation*}
\|f\|_{\text{Ind}_{P}(\sigma,\lambda)}:=(\int_{K}\|f(k)\|_{\sigma}^{2}\,dk)^{\frac{1}{2}}
\end{equation*}
and, upon completing, we obtain a Hilbert representation of $G$ which we denote $\text{Ind}_{P}(\sigma,\lambda)$. We will denote $\text{Ind}_{P,K_{M}}(\sigma,\lambda)$ the space of $K_{M}$-finite vectors in $\text{Ind}_{P}(\sigma,\lambda)$.\\

\begin{prop}\label{parabunit} Let $G$ be a connected, semisimple Lie group with finite centre and let $P=MAN$ be a parabolic subgroup of $G$. Let $\lambda$ be a complex-valued, real-linear, totally imaginary functional on $\mathfrak{a}$ and let $(\sigma,H_{\sigma})$ be a unitary representation of $M$. Then $\text{Ind}_{P}(\sigma,\lambda)$ is a unitary representation of $G$. 
\end{prop}
\begin{proof} See \cite[Corollary 11.39]{KnappVogan}.
\end{proof}

\begin{cor}\label{corparabunit} Let $G$ be a connected, semisimple Lie group with finite centre and let $P=MAN$ be a parabolic subgroup of $G$. Let $\lambda$ be a complex-valued, real-linear, totally imaginary functional on $\mathfrak{a}$ and let $(\sigma,H_{\sigma})$ be a unitary representation of $M$. Then, for every $f_{1},f_{2}\in \text{Ind}_{P,K_{M}}(\sigma,\lambda)$ and for every $X\in \mathfrak{g}$, we have 
\begin{equation*}
\langle \dot{\text{Ind}}_{P}(\sigma,\lambda,X)f_{1},f_{2} \rangle_{\text{Ind}_{P}(\sigma,\lambda)}=-\langle f_{1}, \dot{\text{Ind}}_{P}(\sigma,\lambda,X)f_{2}\rangle_{\text{Ind}_{P}(\sigma,\lambda)}.
\end{equation*}
\end{cor}
\begin{proof} This is a consequence of Proposition \ref{parabunit} and the skew-invariance of the inner product on a unitary representation with respect to the action of the Lie algebra on the space of smooth vectors \cite[p. 266]{Warner}.
\end{proof}

Next, we recall a form of the Frobenius reciprocity originally observed by Casselman. We first need some preparation.\\

First of all, we record the following.

\begin{lem}\label{nV} Let $G$ be a connected, semisimple Lie group with finite centre and let $P=MAN$ be a parabolic subgroup of $G$. If $V$ is a $(\mathfrak{g},K)$-module, then the $(\mathfrak{g},K)$-module structure on $V$ induces a structure of  $(\mathfrak{m}\oplus\mathfrak{a},K_{M})$-module on $V\backslash \mathfrak{n}V$ in such a way that the quotient map 
\begin{equation*}
q\colon V\longrightarrow V/\mathfrak{n}V
\end{equation*}
is $(\mathfrak{m}\oplus\mathfrak{a},K_{M})$-equivariant.
\end{lem}
\begin{proof} It suffices to show that if $v\in V$ is of the form $$v=Xw$$ for some $w\in V$ and $X\in \mathfrak{n}$, then, for all $\xi\in K_{M}$, we have $$\xi v\in \mathfrak{n}V,$$ and, for all $Y\in\mathfrak{m}\oplus \mathfrak{a}$, we have 
\begin{equation*}
Yv\in \mathfrak{n}V.
\end{equation*}

Let $\xi\in K_{M}$. We have $$\xi v=\xi Xw=\text{Ad}(\xi)X \xi w$$ and, since $K_{M}$, being contained in $M$, normalises $\mathfrak{n}$ by \cite[Proposition 7.83]{Knapp2}, it follows that $\text{Ad}(\xi)X\in\mathfrak{n}$.\\

Let $Y\in \mathfrak{m}\oplus\mathfrak{a}$. We have 
\begin{equation*}
Yv=YXw=\left[Y,X\right]w+XYw.
\end{equation*}
The second term in the RHS belongs to $\mathfrak{n}V$ because $X\in\mathfrak{n}$ and the first belongs to $\mathfrak{n}V$ because $\mathfrak{n}$ is an ideal in $\mathfrak{p}=\mathfrak{m}\oplus\mathfrak{a}\oplus\mathfrak{n}$ by \cite[Proposition 7.78]{Knapp2}.
\end{proof}

Let us recall that a $(\mathfrak{g},K)$-module is \textbf{finitely generated} if it is a finitely generated $U(\mathfrak{g}_{\mathbb{C}})$-module. We say that a Hilbert representation $(\pi,H)$ of $G$ is finitely generated if $H_{K}$ is finitely generated. We record the following result of Casselman.

\begin{thm}\label{quotadfin} Let $G$ be a connected, semisimple Lie group with finite centre and let $P=MAN$ be a parabolic subgroup of $G$. Let $V$ be an admissible, finitely generated $(\mathfrak{g},K)$-module. Then $V/\mathfrak{n}V$ is an admissible, finitely generated $(\mathfrak{m}\oplus \mathfrak{a},K_{M})$-module.
\end{thm}
\begin{proof} See \cite[Lemma 4.3.1]{Wal}.
\end{proof}

If $V$ is an irreducible $(\mathfrak{g},K)$-module, we say that $V$ admits an infinitesimal character if the center $\mathrm{Z}(\mathfrak{g}_{\mathbb{C}})$ of the universal enveloping algebra $U(\mathfrak{g}_{\mathbb{C}})$ of the complexification $\mathfrak{g}_{\mathbb{C}}$ of $\mathfrak{g}$, acts on $V$ by a character, that is: for every $Z\in \mathrm{Z}(\mathfrak{g}_{\mathbb{C}})$ and for every $v\in V$, we have
\begin{equation*}
Zv=\chi(Z)v,    
\end{equation*}
where $\chi\colon \mathrm{Z}(\mathfrak{g}_{\mathbb{C}})\longrightarrow \mathbb{C}$ is a morphism of complex, unital algebras. The action of $\mathrm{Z}(\mathfrak{g}_{\mathbb{C}})$ on $V$ in question is the one obtained by first extending the action of $\mathfrak{g}$ to an action of $\mathfrak{g}_{\mathbb{C}}$ and then to an action of $U(\mathfrak{g}_{\mathbb{C}})$ using the PBW Theorem.

\begin{cor}\label{corquotadfin} Let $G$ be a connected, semisimple Lie group with finite centre and let $P=MAN$ be a parabolic subgroup of $G$. If $V$ is an irreducible $(\mathfrak{g},K)$-module admitting an infinitesimal character, then $V/\mathfrak{n}V$ is an admissible, finitely generated $(\mathfrak{m}\oplus\mathfrak{a},K_{M})$-module.
\end{cor}
\begin{proof} By \cite[Theorem 2.2, p. 256]{Helg}, $K$ is connected. By \cite[Theorem 7.204]{KnappVogan}, $V$ is admissible. Combining \cite[Example 1, p. 442]{KnappVogan} and \cite[Corollary 7.207]{KnappVogan}, it follows that $V$ is finitely generated. The result now follows from Theorem \ref{quotadfin}. 
\end{proof}

Let $\mathfrak{p}$, $\mathfrak{m}$, $\mathfrak{a}$ and $\mathfrak{n}$ denote the Lie algebras of $P$, $M$, $A$ and $N$, respectively.\\

Let $(\sigma,H_{\sigma})$ be an admissible and finitely generated Hilbert representation of $M$ which is unitary when restricted to $K_{M}$. Let $\lambda$ be a complex-valued real-linear functional on $\mathfrak{a}$. Consider the $(\mathfrak{m}\oplus \mathfrak{a},K_{M})$-module $H^{\lambda}_{\sigma,K_{M}}$ defined as 
\begin{equation*} H^{\lambda}_{\sigma,K_{M}}:=H_{\sigma,K_{M}}\otimes \mathbb{C}_{\lambda+\rho}
\end{equation*}
where the pair $(\mathfrak{m},K_{M})$ acts on $H_{\sigma,K_{M}}$ and $\mathfrak{a}$ acts on $\mathbb{C}_{\lambda+\rho}$ via the functional $\lambda+\rho$.\\

If $V$ is a $(\mathfrak{g},K)$-module and $T\in \text{Hom}_{\mathfrak{g},K}(V,\text{Ind}_{P,K_{M}}(\sigma,\lambda))$, then we can define an element $\hat{T}\in \text{Hom}_{\mathfrak{\mathfrak{m}\oplus\mathfrak{a}},K_{M}}(V/{\mathfrak{n}V},H^{\lambda}_{\sigma,K_{M}})$ by setting 
\begin{equation*}
\hat{T}(v):=T(v)(1).
\end{equation*}

\begin{thm}\label{Frobenius} Let $G$ be a connected, semisimple Lie group with finite centre. Let $V$ be a $(\mathfrak{g},K)$-module. Let $(\sigma,H_{\sigma})$ be an admissible and finitely generated Hilbert representation of $M$ which is unitary when restricted to $K_{M}$ and let $\lambda$ be a complex-valued real-linear functional on $\mathfrak{a}$. Consider the $(\mathfrak{m}\oplus \mathfrak{a},K_{M})$-module $H^{\lambda}_{\sigma,K_{M}}$. Then the map 
\begin{equation*}
\text{Hom}_{\mathfrak{g},K}(V,\text{Ind}_{P,K_{M}}(\sigma,\lambda))\longrightarrow \text{Hom}_{\mathfrak{m}\oplus \mathfrak{a},K_{M}}(V/{\mathfrak{n}V},H^{\lambda}_{\sigma,K_{M}}),\text{ }T\mapsto \hat{T}
\end{equation*}
is a bijection.
\end{thm}
\begin{proof} See \cite[ Lemma 5.2.3]{Wal} and the discussion preceding it. 
\end{proof}

For clarity, we point out that the formulation in \cite{Wal} seems to contain some typos and so we modified it following \cite[Theorem 4.9]{HechtSchmidt}.\\

The inverse of the map $T\mapsto \hat{T}$ is constructed as follows (see \cite[Lemma 5.2.3 and Lemma 3.8.2]{Wal}. Alternatively, \cite[Theorem 4.9]{HechtSchmidt}). Let $S\in \text{Hom}_{\mathfrak{m}\oplus\mathfrak{a},K_{M}}(V/{\mathfrak{n}V},H^{\lambda}_{\sigma,K_{M}})$. Then we obtain an element $\Tilde{S}\in \text{Hom}_{\mathfrak{g},K}(V,\text{Ind}_{P,K_{M}}(\sigma,\lambda))$ by setting 
\begin{equation*}
\Tilde{S}(v)(k):=S(q(kv)), 
\end{equation*}
where $q\colon V\longrightarrow V/\mathfrak{n}V$ denotes the quotient map. Then the inverse of $T\mapsto \hat{T}$ is given by the map 
\begin{equation*}
\text{Hom}_{\mathfrak{m}\oplus \mathfrak{a},K_{M}}(V/{\mathfrak{n}V},H^{\lambda}_{\sigma,K_{M}})\longrightarrow \text{Hom}_{\mathfrak{g},K}(V,\text{Ind}_{P,K_{M}}(\sigma,\lambda)),\text{ }S\mapsto \Tilde{S}.
\end{equation*}


\section{Asymptotic behaviour of representations} \label{sec:3}

\subsection{Asymptotic expansions of matrix coefficients}

We begin by collecting the fundamental facts concerning asymptotic expansions of matrix coefficients of tempered representations. 
We refer the reader to \cite[Chapter VIII]{Knapp} for a more thorough exposition of the topic. \\

Let $G$ be a connected, semisimple Lie group with finite centre, let $K$ be a fixed maximal compact subgroup of $G$ and let $\mathfrak{k}$ be its Lie algebra. Let $P=MAN$ denote the minimal parabolic subgroup of $G$ with Lie algebra $\mathfrak{p}$. Given a maximal abelian subspace $\mathfrak{a}$ of $\mathfrak{p}$, we call $A$ the corresponding subgroup of $P$ and $M$ the centraliser of $A$ in $K$.
We fix a system $\Delta$ of simple roots of the restricted root system attached to $(\mathfrak{g},\mathfrak{a})$, we use $\Delta^{+}$ to denote the corresponding set of positive roots.\\

Let $\mathfrak{a}^{+}$ denote the set $\{H\in\mathfrak{a}|\alpha(H)>0\text{ for all }\alpha\in \Delta\}$. Then the subset of regular elements $G^{\text{reg}}$ of $G$ admits a decomposition as $G^{\text{reg}}=K\text{exp}(\mathfrak{a}^{+})K$ and $G$ itself admits a decomposition $G=K\overline{\text{exp}(\mathfrak{a}^{+})}K$.\\

We write $\Delta=\{\alpha_{1},\dots,\alpha_{n}\}$ and we identify it with the ordered set $\{1,\dots,n\}$ in the obvious way. We adopt the following notation to simplify the appearance of the expansions we are going to work with. \\

For $H\in\mathfrak{a}$ and $l\in \mathbb{Z}^{n}_{\geq 0}$, we set $\alpha(H)^{l}:=\prod^{n}_{i=1}\alpha_{i}(H)^{l_{i}}$.\\

If $\lambda$ is a real-linear complex-valued functional on $\mathfrak{a}$, since, for every $H\in \mathfrak{a}$, we have 

\begin{equation*}
\lambda(H)=\sum^{n}_{i=1}\lambda_{i}\alpha_{i}(H)
\end{equation*}

for some $\lambda_{1},\dots,\lambda_{n}\in\mathbb{C}$ , we will often identify $\lambda$ with the $n$-tuple $(\lambda_{1},\dots,\lambda_{n})$.\\  

The next result is concerned with the expansion of $K$-finite matrix coefficients relative to $P$. 

\begin{thm}\label{asymp1} Let $G$ be a connected, semisimple Lie group with finite centre and let $(\pi,H)$ be an irreducible, Hilbert representation of $G$. Then there exist a non-negative integer $l_{0}$ and a finite set of real-linear complex-valued functionals on $\mathfrak{a}$, denoted $\mathcal{E}_{0}$, such that, for every $v,w\in H_{K}$, the restriction to $\text{exp}(\mathfrak{a}^{+})$ of the matrix coefficient 
\begin{equation*}
\phi_{v,w}(g)=\langle \pi(g)v,w\rangle
\end{equation*}
admits a uniformly and absolutely convergent expansion as 
\begin{equation*}
\phi_{v,w}(\text{exp}H)=e^{-\rho(H)}\sum_{\lambda\in \mathcal{E}_{0}}\text{ }\sum_{l\in \mathbb{Z}^{n}_{\geq 0}:|l|\leq l_{0}}\text{ }\sum_{k\in\mathbb{Z}^{n}_{\geq 0}}\alpha(H)^{l}e^{(\lambda-k)(H)}\langle c_{\lambda-k,l}(v),w\rangle,
\end{equation*}
where each $c_{\lambda-k,l}\colon H_{K}\longrightarrow H_{K}$ is a complex-linear map and $\rho_{\mathfrak{p}}$ denotes half of the sum of the elements in $\Delta^{+}$ counted with multiplicities. The maps $c_{\lambda-k,l}$ are completely determined by the representation $(\pi,H)$.
\end{thm}
\begin{proof} By Theorem \ref{HC}, $(\pi,H)$ is admissible and therefore has an infinitesimal character. By \cite[Theorem 8.32]{Knapp},  we have the stated expansion for any $\tau$-spherical function (in the sense of \cite[p.215]{Knapp}) $F$ on $G$ of the form 
\begin{equation*}
F(g)=E_{2}\pi(g)E_{1},
\end{equation*}
where $\tau_{1}$ and $\tau_{2}$ are sub-representations of 
\begin{equation*}
\pi|_{K}\cong \bigoplus_{\gamma\in \Tilde{K}}n_{\gamma}\gamma
\end{equation*}
of the form 
\begin{equation*}
\tau_{1}:=\bigoplus_{\gamma\in \Theta_{1}}n_{\gamma}\gamma \text{ and }\tau_{2}:=\bigoplus_{\gamma\in \Theta_{2}}n_{\gamma}\gamma
\end{equation*}
for finite collections $\Theta_{1},\Theta_{2}\in \hat{K}$, and $E_{1}$, $E_{2}$ are the orthogonal projections to $\tau_{1}$, $\tau_{2}$, respectively. In this expansion, the set $\mathcal{E}_{0}$, the maps $c_{\lambda-k,l}$ and the number $l_{0}$ depend on $\tau=(\tau_{1},\tau_{2})$ and we can expand $\phi_{v,w}$ provided that $v\in \tau_{1}$ and $w\in \tau_{2}$. To obtain an expansion valid for every $v,w\in H_{K}$ and with $\mathcal{E}_{0}$, $l_{0}$ and the $c_{\lambda-k,l}$ independent of $\tau$, we appeal to the theory developed in \cite[Section 8]{CassMil}, which we can apply since $(\pi,H)$ is finitely generated by \cite[Corollary 7.207]{KnappVogan} (for clarity, we should mention that the set-up in \cite{Knapp} is different, but entirely equivalent to 
that in \cite{CassMil}, translating between the two is just book-keeping). First, by \cite[Theorem 8.7]{CassMil}, the representation $(\pi,H)$ has a unique matrix coefficient map in the sense  of \cite[p. 907]{CassMil}. The required expansion on the region $\text{exp}(\mathfrak{a}^{+})$ of the matrix coefficient map is given by \cite[Theorem 8.8]{CassMil}. For completeness, the relation between the $\tau$-dependent expansion and the expansion in our statement is given in \cite[Lemma 8.3]{CassMil}. 
\end{proof}

We recall that if $\nu,\nu'$ are real-linear complex-valued functionals on $\mathfrak{a}$ such that $\nu-\nu'$ is an integral linear combination of the simple roots, then we say that $\nu$ and $\nu'$ are \textbf{integrally equivalent}. \\

The set $\mathcal{E}_{0}$ has the property that if $\lambda,\lambda'\in \mathcal{E}_{0}$ with $\lambda\neq \lambda'$, then $\lambda$ and $\lambda'$ are not integrally equivalent.\\

If $\nu$ and $\nu'$ are integrally equivalent and $\nu-\nu'$ is a non-negative integral combination of the simple roots, we write $\nu\geq\nu'$, thus introducing an order relation among integrally equivalent functionals on $\mathfrak{a}$. \\

If $k\in \mathbb{Z}^{n}_{\geq 0}$ is such that the term 
\begin{equation*}\alpha(H)^{l}e^{(\lambda-k)(H)}\langle c_{\lambda-k,l}(v),w\rangle
\end{equation*}
is non-zero for some $\lambda\in\mathcal{E}_{0}$ and for some $v,w\in H_{K}$, then we say that $\nu:=(\lambda-k)$ is an \textbf{exponent} and we let $\mathcal{E}$ denote the set of exponents. The exponents which are maximal with respect to the order relation introduced above are called \textbf{leading exponents}: $\mathcal{E}_{0}$ is precisely the set of leading exponents. \\

The following result is used crucially in \cite{KYD} and in the following. 

\begin{thm}\label{tempcriterion} Let $(\pi,H)$ be an irreducible, tempered, Hilbert representation of $G$. Then every $\lambda\in \mathcal{E}_{0}$ satisfies 
\begin{equation*}
\text{Re}\lambda_{i}\leq 0
\end{equation*}
for every $i\in \{1,\dots,n\}$.
\end{thm}
\begin{proof} See \cite[Theorem 8.53]{Knapp}. Strictly speaking, in loc. cit. the theorem is formulated under some restrictions on $G$, but it is a convenient reference since we are adopting the same normalisation of the exponents. See \cite[Proposition 3.7, p. 83]{BorelWallach} or \cite[Corollary 8.12]{CassMil},  for proofs for more general groups.  
\end{proof}

We now turn to asymptotic expansions of matrix coefficients of $(\pi,H)$ relative to standard (for $P$) parabolic subgroups of $G$. We follow \cite[Chapter VIII, Section 12]{Knapp}. \\

Given a subset $I\subset \{1,...,n\}$, and recalling that we identified $\Delta$ with $\{1,\dots,n\}$, we can associate to it a parabolic subgroup 
\begin{equation*}
P_{I}=M_{I}A_{I^{c}}N_{I^{c}}
\end{equation*}
of $G$ containing $P$ in such a way that the restricted root space $\mathfrak{g}_{-\alpha}$ satisfies $\mathfrak{g}_{-\alpha}\subset \mathfrak{m}_{I}$ if and only if $\alpha\in I$ (with $\mathfrak{m}_{I}$ denoting the Lie algebra of $M_{I}$). For the details, we refer the reader to \cite[Chapter VII]{Knapp2} and \cite[Proposition 5.23]{Knapp}. \\ 

First, we introduce the basis $\{H_{1},\dots,H_{n}\}$ of $\mathfrak{a}$ dual to $\Delta$. We define the Lie algebra $\mathfrak{a}_{I}$ as 
\begin{equation*}
\mathfrak{a}_{I}:=\sum_{i\in I}\mathbb{R}H_{i}
\end{equation*}
and the group $A_{I}$ as 
\begin{equation*}
A_{I}:=\text{exp}(\sum_{i\in I}\mathbb{R}\alpha_{i}).
\end{equation*}
We can then write 
\begin{equation*}
\mathfrak{a}=\mathfrak{a}_{I}\oplus \mathfrak{a}_{I^{c}}\quad\text{and}\quad
A=A_{I}A_{I^{c}}.
\end{equation*}
The groups $N_{I}$ and $N_{I^{c}}$ are the analytic subgroups of $G$ corresponding to the Lie algebras 
\begin{equation*}
\mathfrak{n}_{I}:=\sum_{\beta\in \Delta^{+}:\beta|_{\mathfrak{a}_{I^{c}}}=0}\mathfrak{g}_{\beta}\quad \text{and}\quad\mathfrak{n}_{I^{c}}:=\sum_{\beta\in \Delta^{+}:\beta|_{\mathfrak{a}_{I^{c}}}\neq 0}\mathfrak{g}_{\beta}.
\end{equation*}
We have 
\begin{equation*}
\rho=\rho_{I}+\rho_{I^{c}}
\end{equation*}
with 
\begin{equation*}
\rho_{I}:=\frac{1}{2}\sum_{\beta\in \Delta^{+}:\beta|_{\mathfrak{a}_{I^{c}}}=0}(\text{dim}\mathfrak{g}_{\beta})\beta
\end{equation*}
and analogously for $\rho_{I^{c}}$. Denoting $M_{0,I}$ the analytic subgroup of $G$ corresponding to the Lie algebra 
\begin{equation*}
\mathfrak{m}_{I}=\mathfrak{m}\oplus \mathfrak{a}_{I}\oplus \mathfrak{n}_{I}\oplus \overline{\mathfrak{n}_{I}},
\end{equation*}
the group $M_{I}$ is then given as 
\begin{equation*}
M_{I}:=Z_{K}(\mathfrak{a}_{I^{c}})M_{0,I}.
\end{equation*}
Finally, $K_{I}:=K\cap M_{I}$ is a maximal compact subgroup of $M_{I}$ and $MA_{I}N_{I}$ is a minimal parabolic subgroup of $M_{I}$.\\

\begin{thm}\label{asymp2} Let $G$ be a connected, semisimple Lie group with finite centre and let $(\pi,H)$ be an irreducible, Hilbert representation of $G$. Let $C$ be a compact subset of $M_{I}$ satisfying $K_{I}CK_{I}=C$. Then there exists a positive real number $R$ depending on $C$ such that, for every $m\in C$ and for every $a=\text{exp}H\in A_{I^{c}}$ which satisfies $\alpha_{i}(H)>\text{log}R$ for every $i\in I^{c}$, 
we have 
\begin{equation*}
\phi_{v,w}(m\,\text{exp}H)=e^{-\rho_{I^{c}}(H)}\sum_{\nu\in \mathcal{E}_{I}}\text{ }\sum_{q\in \mathbb{Z}^{I^{c}}_{\geq 0}:|q|\leq q_{0}}\alpha(H)^{q}e^{\nu(H)}c_{\nu,q}^{P_{I}}(m,v,w)
\end{equation*}
for every $v,w\in H_{K}$. Here, $\mathcal{E}_{I}$ is a countable set of real-linear complex-valued functionals on $\mathfrak{a}_{I^{c}}$, each $c^{P_{I}}_{\nu,q}$ extends to a real analytic function on $M_{I}$ and satisfies 
\begin{equation*}
c^{P_{I}}_{\nu,q}(\xi_{2}m\xi_{1},v,w)=c^{P_{I}}_{\nu,q}(m,\pi(\xi_{1})v,\pi(\xi_{2}^{-1})w)
\end{equation*}
for every $\xi_{1},\xi_{2}\in K_{I}$. Moreover, for every $m\in M_{I}$ and $w\in H_{K}$, the map 
\begin{equation*}
H_{K}\longrightarrow \mathbb{C},\text{ }v\mapsto c^{P_{I}}_{\nu,q}(m,v,w)
\end{equation*}
is complex-linear and, for every $m\in M_{I}$ and $v\in H_{K}$, the map 
\begin{equation*}
H_{K}\longrightarrow \mathbb{C},\text{ }w\mapsto c^{P_{I}}_{\nu,q}(m,v,w)
\end{equation*}
is conjugate-linear. The maps $c_{\nu,q}^{P_{I}}\colon M_{I}\times H_K\times H_K\longrightarrow \mathbb{C}$ are completely determined by the representation $(\pi,H)$.
\end{thm}
\begin{proof} For a $\tau$-spherical function $F$ as in the proof of Theorem \ref{asymp1}, the result follows from \cite[Theorem 8.45]{Knapp}. To obtain an expansion independent of $\tau$, it suffices to prove that each $F_{\nu-\rho_{I^{c}_{\lambda}}}$ is independent of $\tau$. \\ 
Let $m\in M_{I}$ and write $m=\xi_{2}a_{I}\xi_{2}$ for some $a_{I}\in \overline{A_{I}^{+}}$, where $A^{+}_{I}$ is the positive Weyl chamber, and some $\xi_{1},\xi_{2}\in K_{I}$. Since 
\begin{equation*}
F_{\nu-\rho_{I^{c}}}(ma,v,w)=F_{\nu-\rho_{I^{c}}}(a_{I}a,\pi(\xi_{1})v,\pi(\xi_{2}^{-1})w),
\end{equation*}
re-labeling things, it suffices to prove that $F_{\nu-\rho_{I^{c}}}(\cdot,v,w)$ is independent of $\tau$ as a function on $\overline{A^{+}_{I}}A_{I^{c}}$. By \cite[ Corollary 8.46]{Knapp}, the functional $\nu\in \mathcal{E_{I}}$ is the restriction of an element in the set of exponents $\mathcal{E}$ in the expansion relative to $P$ and this set is independent of $\tau$ by \cite[Theorem 8.8]{CassMil}. Therefore, it remains to prove that each $c^{P_{\lambda}}_{\nu,q}$ is independent of $\tau$. Since $c^{P_{\lambda}}_{\nu,q}$ is analytic on $M_{I}$, it suffices to prove that $c^{P_{\lambda}}_{\nu,q}(\cdot,v,w)$ as a function on $A_{I}^{+}$ is independent of $\tau$. Given $a_{I}\in A^{+}_{I}$, we can find a compact subset $C$ of $M_{I}$ containing and $a_{I}$ such that $K_{I}CK_{I}=C$, and a positive $R$ depending on $C$, such that for every $H\in \mathfrak{a}_{I^{c}}$ satisfying $\alpha_{i}(H)>\text{log}R$ for every $i\in I^{c}$, the expansion of $\phi_{v,w}(a_{I}a)$ relative to $P$ and the expansion relative to $P_{I}$ are both valid. 
Comparing them as in \cite[p. 251]{Knapp}, it follows that expansion relative to $P_{I}$ is completely determined by the expansion relative to $P$ and the latter is independent of $\tau$ by Theorem \ref{asymp1}.
\end{proof}

For every $\nu\in \mathcal{E}_{I}$, the term 
\begin{equation*}
\alpha(H)^{q}e^{(\nu-\rho_{I^{c}})(H)}c^{P_{I}}_{\nu,q}(m,v,w)
\end{equation*}
is non-zero for some $v,w\in H_{K}$ and some $m\in M$. The set $\mathcal{E}_{I}$ is the set of \textbf{exponents relative to $P_{I}$}.\\

To define the functions of the form $\Gamma_{\lambda,l}$ discussed in the Introduction, the first step consists in associating a standard (for $P$) parabolic subgroup of $G$ to each $\lambda\in \mathcal{E}_{0}$. \\

Let $(\pi,H)$ be an irreducible, tempered, Hilbert representation of $G$ and let $\lambda\in \mathcal{E}_{0}$. We set $I_{\lambda}:=\{i\in \{1,\dots,n\}|\text{Re}\lambda_{i}<0 \}$ which we identify with the subset $\Delta_{\lambda}$ of $\Delta$ defined as $$\Delta_{\lambda}:=\{\alpha_{i}\in \Delta|i\in I_{\lambda}\}.$$ The construction of standard parabolic subgroups from the datum of a subset of $\Delta$ assigns to $I_{\lambda}$ the standard parabolic subgroup $P_{\lambda}$ defined as $$P_{\lambda}:=P_{I_{\lambda}}.$$ It admits a decomposition 
\begin{equation*}
P_{\lambda}=M_{\lambda}A_{\lambda_{0}}N_{\lambda_{0}},
\end{equation*}
where  
\begin{equation*}
A_{\lambda_{0}}:=A_{I^{c}_{\lambda}}.
\end{equation*}
The subgroup $M$ admits a decomposition 
\begin{equation*}
M_{\lambda}=K_{\lambda}A_{\lambda}K_{\lambda},
\end{equation*}
where 
\begin{equation*}
A_{\lambda}:=A_{I_{\lambda}}\quad\text{and }\quad
K_{\lambda}:=K\cap M_{\lambda}.
\end{equation*}
The group $A$ decomposes as $A=A_{\lambda}A_{\lambda_{0}}$. We write $\mathfrak{a}_{\lambda}$ and $\mathfrak{a}_{\lambda_{0}}$ for $\mathfrak{a}_{I_{\lambda}}$ and $\mathfrak{a}_{I^{c}_{\lambda}}$, respectively. Similarly, we write $\rho_{\lambda}$ and $\rho_{\lambda_{0}}$ for $\rho_{I_{\lambda}}$ and $\rho_{I^{c}_{\lambda}}$, respectively.\\

\begin{rem} The theory recalled so far is sufficient to prove that tempered, irreducible, Hilbert representations are unitarisable. From now on, given a tempered, irreducible, Hilbert representation $(\pi,H)$, we will implicitly assume that is unitary and we will refer to it simply as a tempered, irreducible representation.
\end{rem}

\subsection{The functions $\Gamma_{\lambda,l}$}

We are going to introduce an equivalence relation on the data indexing the expansion of $\phi_{v,w}$ relative to $P$. The definition is motivated by the construction of $\textbf{d}(\pi)$ in \cite{KYD}. Let $v,w\in H_{K}$. We have  
\begin{equation*}
\begin{split}
\phi_{v,w}(\text{exp}H)&=e^{-\rho(H)}\sum_{\lambda\in \mathcal{E}_{0}}\text{ }\sum_{l\in \mathbb{Z}^{n}_{\geq 0}:|l|\leq l_{0}}\alpha(H)^{l}e^{\lambda(H)}\Phi^{v,w}_{\lambda,l}(H),
\end{split}
\end{equation*}
where 
\begin{equation*}
\Phi^{v,w}_{\lambda,l}(H):=\sum_{k\in\mathbb{Z}^{n}_{\geq0}}e^{-k(H)}\langle c_{\lambda-k,l}(v_{1}),v_{2}\rangle.
\end{equation*}
The terms in this expansion are indexed by the finite set 
\begin{equation*}\mathcal{C}:=\{(\lambda,l)\}_{\lambda\in \mathcal{E}_{0},\text{ } l\in \mathbb{Z}^{n}_{\geq 0}:|l|\leq l_{0}}.
\end{equation*}
We introduce a relation on $\mathcal{C}$ by declaring that $(\lambda,l)\sim (\mu,m)$ if $I_{\lambda}=I_{\mu}$, $\lambda|_{\mathfrak{a}_{\lambda_{0}}}=\mu|_{\mathfrak{a}_{\lambda_{0}}}$ and $\text{res}_{I^{c}_{\lambda}}l=\text{res}_{I^{c}_{\mu}}m$. To define this relation we have implicitly used the identification of $I_{\lambda}$ with the subset $\Delta_{\lambda}$ of $\Delta$ at the end of the previous subsection.\\

It is clear that $\sim$ is an equivalence relation. We denote $\left[\lambda,l\right]$ the equivalence class containing $(\lambda,l)$.\\

We can therefore re-group the expansion of $\phi_{v,w}$ as follows: 

\begin{equation*}
\begin{split}
\phi_{v,w}(\text{exp}H)&=e^{-\rho(H)}\sum_{\left[\lambda,l\right]\in \mathcal{C}/{\sim}}\alpha(H_{\lambda_{0}})^{l_{\lambda_{0}}}e^{\lambda|_{\mathfrak{a}_{\lambda_{0}}}(H_{\lambda_{0}})}\sum_{(\lambda',l')\in \left[\lambda,l\right]}\alpha(H_{\lambda})^{l'_{\lambda}}e^{\lambda'|_{\mathfrak{a}_{\lambda}}(H_{\lambda})}\Phi^{v,w}_{\lambda',l'}(H),
\end{split}
\end{equation*}
where 
\begin{equation*}
l_{\lambda_{0}}:=\text{res}_{I^{c}_{\lambda}}l,\text{ }\alpha(H_{\lambda_{0}})^{l_{\lambda_{0}}}:=\prod_{i\in I^{c}_{\lambda}}\alpha_{i}(H_{\lambda_{0}})^{l_{i}},\text{ } l'_{\lambda}:=\text{res}_{I_{\lambda}}l',\text{ } \alpha(H_{\lambda})^{l'_{\lambda}}:=\prod_{i\in I_{\lambda}}\alpha(H_{\lambda})^{l'_{i}}
\end{equation*}
and $H=H_{\lambda_{0}}+H_{\lambda}$ corresponds to the decomposition 
\begin{equation*}\mathfrak{a}^{+}=\mathfrak{a}^{+}_{\lambda_{0}}\oplus\mathfrak{a}^{+}_{\lambda}.
\end{equation*}

We are also implicitly using the fact that $\alpha(H)^{l}=\alpha(H_{\lambda})^{l_{\lambda}}\alpha(H_{\lambda_{0}})^{l_{\lambda_{0}}}$ which follows from writing $H$ with respect to the basis dual to $\Delta$.\\

To proceed, we need to isolate certain equivalence classes in $\mathcal{C}/{\sim}$. First, we recall from the Introduction how the quantity $\textbf{d}_{P}(\lambda,l)$, for $(\lambda,l)\in \mathcal{C}$ and $P$ a fixed minimal parabolic subgroup of $G$, and the quantity $\textbf{d}(\pi)$ are defined.\\

For $(\lambda,l)\in \mathcal{C}$, we set 
\begin{equation*}
\textbf{d}_{P}(\lambda,l):=|I^{c}_{\lambda}|+\sum_{i\in I^{c}_{\lambda}}2l_{i}
\end{equation*}
and we observe that this number only depends on the equivalence class of $(\lambda,l)$. Then we take the maximum, $\textbf{d}_{P}$, as $(\lambda,l)$ ranges over $\mathcal{C}$. We can proceed analogously for every standard (for $P$) parabolic subgroup of $P'$ of $G$ to obtain a non-negative integer $\textbf{d}_{P'}$. Then $\textbf{d}(\pi)$ is defined to be the maximum over all $P'$ of the quantities $\textbf{d}_{P'}$.

\begin{definition} Let $\left[\lambda,l\right]\in \mathcal{C}/{\sim}$. We say that $\left[\lambda,l\right]$ is \textbf{relevant} if it satisfies 
\begin{equation*}
\textbf{d}_{P}(\lambda,l)=\textbf{d}(\pi),
\end{equation*}
where $\textbf{d}_{P}(\lambda,l)$ is defined by \eqref{defdp}.   
\end{definition}

Let $\left[\lambda,l\right]\in \mathcal{C}/{\sim}$ be a relevant equivalence class. For $H_{\lambda}\in \mathfrak{a}^{+}_{\lambda}$, we set 
\begin{equation}\label{defgamma}
\Gamma_{\lambda,l}(\text{exp}H_{\lambda},v,w):=e^{-\rho(H)}\sum_{(\lambda',l')\in\left[\lambda,l\right]}\alpha(H_{\lambda})^{l'_{\lambda}}e^{\lambda'|_{\mathfrak{a}_{\lambda_{0}}}(H_{\lambda})}\Phi^{v,w}_{\lambda',l'}(H_{\lambda}).
\end{equation}
Before establishing the properties of $\Gamma_{\lambda,l}$, let us pause to explain the motivation behind the definitions above. The discussion that follows will be used only in Section~\ref{sec:4}. The reader who prefers to do so can skip to Proposition \ref{Gamma} without any loss of continuity.\\

Let $v_{1},v_{2},v_{3},v_{4}\in H_{K}$. We will be considering integrals of the form 
\begin{equation*}
\lim_{r\rightarrow \infty }\frac{1}{r^{\textbf{d}(\pi)}}\int_{\mathfrak{a}_{<r}^{+}}\phi_{v_{1},v_{2}}(\text{exp}H)\overline{\phi_{v_{3},v_{4}}(\text{exp}H)}\prod_{\beta\in\Delta^{+}}(e^{\beta(H)}-e^{-\beta(H)})^{\text{dim}\mathfrak{g}_{\beta}}\,dH,
\end{equation*}
where 
\begin{equation}\label{defregiona}
\mathfrak{a}_{<r}^{+}:=\mathfrak{a}^{+}\cap\{H\in\mathfrak{a}|\beta(H)<r\text{ for all }\beta\in \Delta^{+}\}.
\end{equation}
Treating these is the content of \cite[Appendix A ]{KYD}. We remark that our region of integration is defined as to exclude the subset of $\overline{\mathfrak{a}^{+}}$ where at least one of the simple roots vanishes. It is a set of measure zero.  \\

We want to interpret  \cite[Lemma A.5]{KYD} in group-theoretic terms. \\

Let us consider the matrix coefficients $\phi_{v_{1},v_{2}}$ and $\phi_{v_{3},v_{4}}$. On $A^{+}:=\text{exp}(\mathfrak{a}^{+})$, they can be expanded as 
\begin{equation*}
\phi_{v_{1},v_{2}}(\text{exp}H)=e^{-\rho(H)}\sum_{\left[\lambda,l\right]\in \mathcal{C}/{\sim}}\alpha(H_{\lambda_{0}})^{l_{\lambda_{0}}}e^{\lambda|_{\mathfrak{a}_{\lambda_{0}}}(H_{\lambda_{0}})}\sum_{(\lambda',l')\in \left[\lambda,l\right]}\Psi^{v_{1},v_{2}}_{\lambda',l'}(H)
\end{equation*}
and 
\begin{equation*}
\phi_{v_{3},v_{4}}(\text{exp}H)=e^{-\rho(H)}\sum_{\left[\mu,m\right]\in \mathcal{C}/{\sim}}\alpha(H_{\mu_{0}})^{m_{\mu_{0}}}e^{\mu|_{\mathfrak{a}_{\mu_{0}}}(H_{\mu_{0}})}\sum_{(\mu',m')\in \left[\mu,m\right]}\Psi^{v_{3},v_{4}}_{\mu',m'}(H).
\end{equation*}
where, for $(\lambda',l')\in\left[\lambda,l\right]$, we set 
\begin{equation*}
\Psi^{v_{1},v_{2}}_{\lambda',l'}(H):=\alpha(H_{\lambda})^{l'_{\lambda}}e^{\lambda'|_{\mathfrak{a}_{\lambda}}(H_{\lambda})}\Phi^{v_{1},v_{2}}_{\lambda',l'}(H)
\end{equation*}
and similarly for $(\mu',m')\in\left[\mu,m\right].$\\
Let $\left[\lambda,l\right]\in\mathcal{C}/{\sim}$ and $\left[\mu,m\right]\in\mathcal{C}/{\sim}$ be such that $I_{\lambda}=I_{\mu}$, $\lambda|_{\mathfrak{a}_{\lambda_{0}}}=\mu|_{\mathfrak{a}_{\lambda_{0}}}$ and 
\begin{equation*}
\textbf{d}(\pi)=|I_{\lambda}|+\sum_{i\in I_{\lambda}}(l_{i}+m_{i}).
\end{equation*}
In view of the first condition, the third is equivalent to the requirement 
\begin{equation*}
\textbf{d}_{P}(\lambda,l)=\textbf{d}(\pi)\text{ and }\textbf{d}_{P}(\mu,m)=\textbf{d}(\pi).
\end{equation*}
Consider the summand 
\begin{equation*}
e^{-2\rho(H)}\alpha(H)^{l'+m'}e^{(\lambda'+\overline{\mu'})(H)}\Phi^{v_{1},v_{2}}_{\lambda',l'}\overline{\Phi^{v_{3},v_{4}}_{\mu',m'}}(H)
\end{equation*}
in the expansion of the product $\phi_{v_{1},v_{2}}\overline{\phi_{v_{3},v_{4}}}$ on $A^{+}$.\\
Taking into account the factor $e^{-2\rho(H)}$ and the fact that the term 
\begin{equation}\label{Omega}
\Omega(H):=\prod_{\beta\in \Delta^{+}}(e^{\beta(H)}-e^{-\beta(H)})^{\text{dim}g_{\beta}}
\end{equation}
is incorporated in the function $\phi$ in \cite[Lemma A.5]{KYD} (compare Section 4.7 in \textit{loc. cit.}), this lemma shows that, as $r\rightarrow \infty$, the integral 
\begin{equation*}
\frac{1}{r^{\textbf{d}(\pi)}}\int_{\mathfrak{a}^{+}_{<r}}e^{-2\rho(H)}\alpha(H)^{l'+m'}e^{(\lambda'+\overline{\mu'})(H)}\Phi^{v_{1},v_{2}}_{\lambda',l'}\overline{\Phi^{v_{3},v_{4}}_{\mu',m'}}(H)\Omega(H)\,dH
\end{equation*}
tends to 
\begin{equation*}
C(\lambda,l,m)\int_{\mathfrak{a}^{+}_{\lambda}}e^{-2\rho_{\lambda}(H_{\lambda})}\left[\Psi^{v_{1},v_{2}}_{\lambda',l'}\overline{\Psi^{v_{3},v_{4}}_{\mu',m'}}\right]|_{\mathfrak{a}_{\lambda}}(H_{\lambda})\Omega_{\lambda}(H_{\lambda})\,
dH_{\lambda},
\end{equation*}
where 
\begin{equation}\label{defomegalambda} \Omega_{\lambda}(H_{\lambda}):=\prod_{\beta\in \Delta^{+}_{\lambda}}(e^{\beta(H_{\lambda})}-e^{-\beta(H_{\lambda})})^{\text{dim}g_{\beta}},
\end{equation}
with
\begin{equation*}
\Delta^{+}_{\lambda}:=\{\beta\in \Delta^{+}|\beta|_{\mathfrak{a}_{\lambda_{0}}}=0\},
\end{equation*}
and the quantity $C(\lambda,l,m)$ is given by 
\begin{equation}\label{defC}
C(\lambda,l,m):=\int_{\{H\in \mathfrak{a}_{{\lambda}_{0}}|\text{ext}^{I^{c}_{\lambda}}(H)\in \mathfrak{a}^{+}_{<1}\}}\alpha(H_{\lambda_{0}})^{l_{\lambda_{0}}+m_{\mu_{0}}}\,dH_{{\lambda_{0}}}.
\end{equation}
Now, summing over all $(\lambda',l')\in \left[\lambda,l\right]$ and over all $(\mu',m')\in\left[\mu,m\right]$, we obtain that the integral over $\mathfrak{a}^{+}_{<r}$ of 
\begin{equation*}
e^{-2\rho(H)}\sum_{(\lambda',l')\in\left[\lambda,l\right]}\text{ }\sum_{(\mu',m')\in\left[\mu,m\right]}\alpha(H)^{l'+m'}e^{(\lambda'+\overline{\mu'})(H)}\Phi^{v_{1},v_{2}}_{\lambda',l'}\overline{\Phi_{\mu',m}^{v_{3},v_{4}}}(H)\Omega(H),
\end{equation*}
upon multiplying by $\frac{1}{r^{\textbf{d}(\pi)}}$ and letting $r\rightarrow \infty$, equals 
\begin{equation*}
C(\lambda,l,m)\int_{\mathfrak{a}^{+}_{\lambda}}e^{-2\rho_{\lambda}(H_{\lambda})}\sum_{(\lambda',l')\in\left[\lambda,l\right]}\text{ }\sum_{(\mu',m')\in\left[\mu,m\right]}\left[\Psi^{v_{1},v_{2}}_{\lambda',l'}\overline{\Psi^{v_{3},v_{4}}_{\mu',m'}}\right]|_{\mathfrak{a}_{\lambda}}(H_{\lambda})\Omega_{\lambda}(H_{\lambda})\,dH_{\lambda}.
\end{equation*}
Finally, since 
\begin{equation*}
\Phi^{v_{1},v_{2}}_{\lambda',l'}|_{\mathfrak{a}_{\lambda}}(H_{\lambda})=\sum_{k\in \mathbb{Z}_{\geq 0}^{I_{\lambda}}}e^{-k(H_{\mathfrak{\lambda}})}\langle c_{\lambda'-k,l'}(v_{1}),v_{2}\rangle,
\end{equation*}
and similarly for $\Phi_{\mu',m,}^{v_{3},v_{4}}$, the integral above equals 
\begin{equation}\label{asympint1}
C(\lambda,l,m)\int_{\mathfrak{a}^{+}_{\lambda}}\Gamma_{\lambda,l}(\text{exp}H_{\lambda},v,w)\overline{\Gamma_{\mu,m}(\text{exp}H_{\lambda},v,w)}\Omega_{\lambda}(H_{\lambda})\,dH_{\lambda}.
\end{equation}
If $\left[\lambda,l\right], \left[\mu,m\right]\in \mathcal{C}/{\sim}$ fail to satisfy any of the three conditions $I_{\lambda}=I_{\mu}$, $\lambda|_{\mathfrak{a}_{\lambda_{0}}}=\mu|_{\mathfrak{a}_{\lambda}}$ and 
\begin{equation*}
\textbf{d}_{P}(\lambda,l)=\textbf{d}(\pi)=\textbf{d}_{P}(\mu,m),
\end{equation*}
then, for every $(\lambda',l')\in\left[\lambda,l\right]$ and for every $(\mu',m')\in \left[\mu,m\right]$, by the considerations in the proof of Claim A.6 and Lemma A.5 in \cite{KYD}, the integral 
\begin{equation*}
\frac{1}{r^{\textbf{d}(\pi)}}\int_{A^{+}_{<r}}e^{-2\rho(H)}\alpha(H)^{l'+m'}e^{(\lambda'+\overline{\mu'})(H)}\Phi^{v_{1},v_{2}}_{\lambda',l'}\overline{\Phi^{v_{3},v_{4}}_{\mu',m'}}(H)\Omega(H)\,dH
\end{equation*}
vanishes as $r\rightarrow \infty$. \\

Therefore, the relevant equivalence classes $\left[\lambda,l\right]\in \mathcal{C}/{\sim}$, those for which the functions of the form $\Gamma_{\lambda,l}$ are defined, are precisely the ones that may contribute a non-zero term to the expression 
\begin{equation*}
\lim_{r\rightarrow \infty }\frac{1}{r^{\textbf{d}(\pi)}}\int_{\mathfrak{a}^{+}_{<r}}\phi_{v_{1},v_{2}}(\text{exp}H)\overline{\phi_{v_{3},v_{4}}(\text{exp}H)}\Omega(H)\,dH.
\end{equation*}
Throughout the rest of this section, we fix a tempered, irreducible representation of a connected, semisimple Lie group $G$ with finite centre. \\

\subsection{Some properties of the functions $\Gamma_{\lambda,l}$}

To study the properties of $\Gamma_{\lambda,l}$, we begin by showing that it is equal to a function of the form $c_{\nu,q}^{P_{\lambda}}$. More precisely, we have:

\begin{prop}\label{Gamma} Let $v,w\in H_{K}$. Let $\left[\lambda,l\right]\in \mathcal{C}/{\sim}$ be a relevant equivalence class.
Set $\nu:=\lambda|_{\mathfrak{a}_{\lambda_{0}}}$
and $q:=l_{\lambda_{0}}$. Then, for every $H_{\lambda}\in \mathfrak{a}^{+}_{\lambda}$, we have 
\begin{equation*}
\Gamma_{\lambda,l}(\text{exp}H_{\lambda},v,w)=c^{P_{\lambda}}_{\nu,q}(\text{exp}H_{\lambda},v,w).
\end{equation*}
\end{prop}
\begin{proof} For every $H_{\lambda}\in\mathfrak{a}^{+}_{\lambda}$, we can find a compact subset $C$ of $M_{\lambda}$ such that $K_{\lambda}CK_{\lambda}=C$ and which contains $H_{\lambda}$, and a positive real $R>0$ such that if $H_{\lambda_{0}}\in \mathfrak{a}^{+}_{\lambda_{0}}$ satisfies $\alpha_{i}(H_{\lambda_{0}})>\text{log}R$ for every $i \in I^{c}_{\lambda}$, then the expansion of $\phi_{v,w}$ with respect to $P$ and the expansion with respect to $P_{\lambda}$ are both valid at $H=H_{\lambda}+H_{\lambda_{0}}$. Comparing them as in \cite[p. 251]{Knapp}, we see that 
\begin{equation*}
c^{P_{\lambda}}_{\nu,q}(\text{exp}H_{\lambda},v,w)=\sum_{\lambda'\in\mathcal{E}_{0}: \lambda'|_{\mathfrak{a}_{\lambda_{0}}}=\nu }\text{ }\sum_{l': |l'|\leq l_{0}\text{ and }l'_{\lambda_{0}}=q}e^{-\rho_{\lambda}(H_{\lambda})}\Psi_{\lambda',l'}^{v,w}(H_{\lambda}).
\end{equation*}
Since, by definition of $\Gamma_{\lambda,l}(\cdot,v,w)$, we have 
\begin{equation*}
\Gamma_{\lambda,l}(\text{exp}H_{\lambda},v,w)=e^{-\rho(H_{\lambda})}\sum_{(\lambda',l')\in \left[\lambda,l\right]}\Psi_{\lambda',l'}^{v,w}(H_{\lambda}),
\end{equation*}
recalling the definition of the equivalence relation that we imposed on $\mathcal{C}$, we only need to show that the set 
\begin{equation*}
\{\lambda'\in\mathcal{E}_{0}|\lambda'|_{\mathfrak{a}_{\lambda_{0}}}=\nu\}
\end{equation*}
is equal to the set 
\begin{equation*}
\{\lambda\in \mathcal{E}_{0}|I_{\lambda'}=I_{\lambda}\text{ and }\lambda'|_{\mathfrak{a}_{\lambda_{0}}}=\lambda|_{\mathfrak{a}_{\lambda_{0}}}\}.
\end{equation*}
Because of the assumption on $\left[\lambda,l\right]$, for every $\lambda'\in \mathcal{E}_{0}$ such that $\lambda'|_{\mathfrak{a}_{\lambda_{0}}}=\nu$, we have $\text{Re}\lambda'_{j}\neq 0$ for every $j\in I_{\lambda}$. Indeed, if there existed a $j\in I_{\lambda}$ for which $\text{Re}\lambda'_{j}=0$, we would have 
\begin{equation*}
|I^{c}_{\lambda'}|\geq1+|I^{c}_{\lambda}|
\end{equation*}
and, since $l'_{\lambda_{0}}=l_{\lambda_{0}}$, this would imply 
\begin{equation*}
\textbf{d}_{P}(\lambda',l')>|I^{c}_{\lambda}|+\sum_{i\in I^{c}_{\lambda'}}2l'_{i}\geq \textbf{d}_{P}(\lambda,l)=\textbf{d}(\pi),
\end{equation*}
contradicting the maximality of $\textbf{d}(\pi)$. Since, by Theorem \ref{tempcriterion}, we have $\text{Re}\lambda_{i}'\leq 0$ for every $i\in \{1,\dots,n\}$, this concludes the proof.  
\end{proof}

Theorem 8.45 in \cite{Knapp} and the discussion at the beginning of p. 251 in loc. cit. now show that $\Gamma_{\lambda,l}(\cdot,v,w)$, being equal to $c_{\nu,q}^{P_{\lambda}}$, extends to an analytic function on $M_{\lambda}$, which we denote again $\Gamma_{\lambda,l}(\cdot,v,w)$. If we decompose $M_{\lambda}$ as 
\begin{equation*}
M_{\lambda}=K_{\lambda}\text{exp}(\overline{\mathfrak{a}^{+}_{\lambda}})K_{\lambda},
\end{equation*}
and if we write $m\in M_{\lambda}$ as $m=\xi_{2}\text{exp}H_{\lambda}\xi_{1}$ for some $\xi_{1},\xi_{2}\in K_{\lambda}$ and some $H_{\lambda}\in\overline{\mathfrak{a}^{+}_{\lambda}}$, then we have 
\begin{equation*}
\Gamma_{\lambda,l}(m,v,w)=\Gamma_{\lambda,l}(\text{exp}H_{\lambda},\pi(\xi_{1})v,\pi(\xi_{2})^{-1}w)
\end{equation*}
because $c_{\nu,q}^{P_{\lambda}}(\cdot,v,w)$ exhibits the same behaviour.\\

We want to prove that $\Gamma_{\lambda,l}(\cdot,v,w)$ belongs to $L^{2}(M_{\lambda})$ and it is $Z(\mathfrak{m_{\lambda}}_{\mathbb{C}})$-finite. An application of Theorem \ref{smoothvector} will imply that $\Gamma_{\lambda,l}(\cdot,v,w)$ is a smooth vector in $L^{2}(M_{\lambda})$. Similar ideas appear in \cite[Chapter VIII]{Knapp}, and in \cite{LanglandsClass}.  \\

\begin{prop}\label{squareint} Let $v,w\in H_{K}$. Let $\left[\lambda,l\right]\in \mathcal{C}/{\sim}$ be a relevant equivalence class. Then $\Gamma_{\lambda,l}(\cdot,v,w)$ belongs to $L^{2}(M_{\lambda})$. 
\end{prop}
\begin{proof} We argue as in the proof of  \cite[Lemma 4.10]{LanglandsClass}. By the proof of Proposition~\ref{Gamma}, we have $\text{Re}\lambda_{i}'<0$ for every $\lambda'$ appearing in the expansion of $\Gamma_{\lambda,l}(\cdot,v,w)$ on $A_{\lambda}^{+}$ and for every $i\in I_{\lambda}$. Since $\Gamma_{\lambda,l}(\cdot,v,w)$ is analytic on $\overline{A_{\lambda}^{+}}$, we can apply \cite[Theorem 4]{Harish-ChandraSupp} and then argue as in  \cite[Theorem 7.5]{CassMil} to establish the desired square-integrability on $\overline{A_{\lambda}^{+}}$. The square-integrability on $M_{\lambda}$ follows from combining the decomposition of $M_{\lambda}$ as $M_{\lambda}=K_{\lambda}\overline{A^{+}_{\lambda}}K_{\lambda}$, the corresponding integral formula and the fact that if $m=\xi_{2}\text{exp}H_{\lambda}\xi_{2}$, for some $H_{\lambda}\in\overline{\mathfrak{a}^{+}_{\lambda}}$ and some $\xi_{1},\xi_{2}\in K_{\lambda}$, then 
\begin{equation*}
\Gamma_{\lambda,l}(m,v,w)=\Gamma_{\lambda,l}(\text{exp}H_{\lambda},\pi(\xi_{1})v,\pi(\xi_{2})^{-1}w).
\end{equation*}
\end{proof}

We recall that there exists an injective algebra homomorphism 
\begin{equation*}
\mu_{P_{\lambda}}\colon \mathrm{Z}(\mathfrak{g}_{\mathbb{C}})\longrightarrow Z((\mathfrak{m}_{\lambda}\oplus\mathfrak{a}_{\lambda_{0}})_{\mathbb{C}})\cong Z(\mathfrak{m}_{\lambda\mathbb{C}})\otimes U(\mathfrak{a}_{\lambda_{0}\mathbb{C}})
\end{equation*}
which turns   $Z(\mathfrak{m}_{\lambda\mathbb{C}})\otimes U(\mathfrak{a}_{\lambda_{0}\mathbb{C}})$ into a free module of finite rank over $\mu_{P_{\lambda}}(\mathrm{Z}(\mathfrak{g}_{\mathbb{C}}))$ by \cite[Lemma 21]{Harish-ChandraInv}.\\ 

\begin{prop}\label{GammaGar} Let $v,w\in H_{K}$. Let $\left[\lambda,l\right]\in \mathcal{C}/{\sim}$ be a relevant equivalence class. Then, for every $X\in U(\mathfrak{m}_{\lambda\mathbb{C}})$ and for every $m\in M_{\lambda}$, we have 
\begin{equation*}
X\Gamma_{\lambda,l}(m,v,w)=\Gamma_{\lambda,l}(m,\dot{\pi}(X)v,w).
\end{equation*}
Moreover, the function $\Gamma_{\lambda,l}(\cdot,v,w)$ is a smooth vector in the right-regular representation $(R,L^{2}(M_{\lambda}))$ of $M_{\lambda}$. 
\end{prop}
\begin{proof} For a given $X\in U(\mathfrak{m}_{\lambda\mathbb{C}})$ and every $g\in G$, we have 
\begin{equation*}
X\phi_{v,w}(g)=\phi_{\dot{\pi}(X)v,w}(g).
\end{equation*}
Therefore, the restriction of $X\phi_{v,w}(\cdot)$ to $M_{\lambda}A_{\lambda_{0}}$ satisfies 
\begin{equation*}
X\phi_{v,w}(ma)=\phi_{\dot{\pi}(X)v,w}(ma).
\end{equation*}

Given $m\in M_{\lambda}$ we can find a compact subset $C$ of $M_{\lambda}$ containing $m$ such that $K_{\lambda}CK_{\lambda}=C$ and a positive $R$ depending on $C$ such that if $a=\text{exp}H\in A_{\lambda_{0}}^{+}$ satisfies $\alpha_{i}(H)>\text{log}R$ for every $i\in I^{c}_{\lambda}$, then $\phi_{\dot{\pi}(X)v,w}(ma)$ may be expanded with respect to $P_{\lambda}$. Since $X\in U(\mathfrak{m}_{\lambda\mathbb{C}})$, the restriction of $X\phi_{v,w}(\cdot)$ to $M_{\lambda}A_{\lambda_{0}}$ can also be computed as the action of the differential operator $X$ on the restriction of $\phi_{v,w}(\cdot)$ to $M_{\lambda}A_{\lambda_{0}}$. For $m\in M_{\lambda}$ and $a\in A_{\lambda_{0}}^{+}$ as above, we expand the function so obtained with respect to $P_{\lambda}$ 
and, as in the proof of (4.8) in \cite{LanglandsClass}, because of the convergence of the series, we can apply the differential operator term by term. By comparing the resulting expansion with the expansion of $\phi_{\dot{\pi}(X)v,w}(ma)$, and 
invoking \cite[Corollary B.26]{Knapp}, we obtain 
\begin{equation*}
Xc^{P_{\lambda}}_{\nu,q}(m,v,w)=c^{P_{\lambda}}_{\nu,q}(m,\dot{\pi}(X)v,w)
\end{equation*}
for every $\nu\in \mathcal{E}_{I}$ and every $q\in\mathbb{Z}^{I^{c}_{\lambda}}_{\geq 0}$. The first statement now follows from choosing $\nu$ and $q$ as in Proposition \ref{Gamma}. \\

For the last statement, we need to show that $\Gamma_{\lambda,l}(\cdot,v,w)$ is annihilated by an ideal of finite codimension in $Z(\mathfrak{m}_{\lambda\mathbb{C}})$; the result will then follow from Theorem \ref{smoothvector}. Let $J$ be the kernel of the infinitesimal character of $(\pi,H)$. Then $J$ is an ideal of finite codimension in $\mathrm{Z}(\mathfrak{g}_{\mathbb{C}})$. As observed in \cite[p.182]{Harish-Chandra}, the inverse image $J_{\mathfrak{m}_{\lambda}}$ along the inclusion 
\begin{equation*}
Z(\mathfrak{m}_{\lambda\mathbb{C}})\longrightarrow Z(\mathfrak{m}_{\lambda\mathbb{C}})\otimes U(\mathfrak{a}_{\lambda_{0}\mathbb{C}}),\text{ } X\mapsto X\otimes 1
\end{equation*}
of the ideal generated by $\mu_{P_{\lambda}}(J)$ in $Z(\mathfrak{m}_{\lambda\mathbb{C}})\otimes U(\mathfrak{a}_{\lambda_{0}\mathbb{C}})$ is an ideal of finite codimension in $Z(\mathfrak{m_{\lambda\mathbb{C}}})$. This follows from the fact that the ideal generated by $\mu_{P_{\lambda}}(J)$ is of finite codimension in $Z(\mathfrak{m}_{\lambda\mathbb{C}})\otimes U(\mathfrak{a}_{\lambda_{0}\mathbb{C}})$, since $Z(\mathfrak{m}_{\lambda\mathbb{C}})\otimes U(\mathfrak{a}_{\lambda_{0}\mathbb{C}})$ is a free module of finite type over $\mu_{P_{\lambda}}(\mathrm{Z}(\mathfrak{g}_{\mathbb{C}}))$ by \cite[Lemma 21]{Harish-ChandraInv}. Denoting $\mu_{P_{\lambda}}(J)^{e}$ the ideal generated by $\mu_{P_{\lambda}}(J)$, we see that $J_{\mathfrak{m}_{\lambda}}$ is precisely the kernel of the homomorphism 
\begin{equation*}
Z(\mathfrak{m}_{\lambda\mathbb{C}})\longrightarrow (Z(\mathfrak{m_{\lambda\mathbb{C}}})\otimes U(\mathfrak{a}_{\lambda_{0}\mathbb{C}}))/\mu_{P_{\lambda}}(J)^{e},\text{ }X\mapsto (X\otimes 1)+\mu_{P_{\lambda}}(J)^{e}.
\end{equation*}
This exhibits $J_{\mathfrak{m}_{\lambda}}$ as an ideal of finite codimension in $Z(\mathfrak{m}_{\lambda\mathbb{C}})$. Now, if $X\in J_{\mathfrak{m}_{\lambda}}$, then $X\otimes 1$ belongs to $\mu_{P_{\lambda}}(J)^{e}$. Hence $X\otimes 1$ can be written as 
\begin{equation*}
X\otimes 1=\sum^{r}_{i=1}Y_{i}\mu_{P_{\lambda}}(Z_{i})
\end{equation*}
with $Y_{i}\in Z(\mathfrak{m}_{\lambda\mathbb{C}})\otimes U(\mathfrak{a}_{\lambda_{0}\mathbb{C}})$ and $Z_{i}\in J$. For every $i\in\{1,\dots,r\}$, by (8.68) in \cite[p. 251]{Knapp}, the differential operator $\mu_{P_{\lambda}}(Z_{i})$ annihilates the function 
\begin{equation*}
F_{\nu-\rho_{\lambda_{0}}}(ma,v,w):=\sum_{q:|q|\leq q_{0}}c^{P_{\lambda}}_{\nu,q}(m,v,w)\alpha(H)^{q}e^{(\nu-\rho_{\lambda_{0}})(H)}.
\end{equation*}
Therefore, $X\otimes 1$ annihilates it, as well. On the other hand, by the first part of the proof, we have 
\begin{equation*}
(X\otimes 1)F_{\nu-\rho_{\lambda_{0}}}(ma,v,w)=\sum_{q:|q|\leq q_{0}}c^{P_{\lambda}}_{\nu,q}(m,\dot{\pi}(X)v,w)\alpha(H)^{q}e^{(\nu-\rho_{\lambda_{0}})(H)}.
\end{equation*}
Since the LHS vanishes identically on $M_{\lambda}A_{\lambda_{0}}$, it follows that 
\begin{equation*}
c^{P_{\lambda}}_{\nu,q}(m,\dot{\pi}(X)v,w)=0
\end{equation*}
for every $m\in M_{\lambda}$. Choosing $\nu$ and $q$ as in Proposition \ref{Gamma}, we find that $\Gamma_{\lambda,l}(\cdot,v,w)$ is annihilated by $J_{\mathfrak{m}_{\lambda}}$.
\end{proof}

\subsection{The functions $\Gamma_{\lambda,l}$ as intertwining operators}

Let $w\in H_{K}$. The next two technical lemmata, together with Proposition \ref{GammaGar}, will be used to prove the $(\mathfrak{m}_{\lambda}\oplus\mathfrak{a},K_{\lambda})$-equivariance of the map 
\begin{equation*}
S_{w}\colon H_{K}\longrightarrow L^{2}(M_{\lambda})\otimes \mathbb{C}_{\lambda|_{\mathfrak{a}_{\lambda_{0}}}-\rho_{\lambda_{0}}},\text{ }S_{w}(v)(m):=\Gamma_{\lambda,l}(m,v,w).
\end{equation*}
We are not claiming that for every $w\in H_{K}$ this map is non-zero: the only thing we need to know is that, whenever $w\in H_{K}$ is such that $S_{w}$ is not identically zero, then $S_{w}$ is $(\mathfrak{m}_{\lambda}\oplus\mathfrak{a},K_{\lambda})$-equivariant. In the final part of this subsection, we show the existence of an admissible, finitely generated, unitary representation of $M_{\lambda}$ which will allow us to apply Theorem \ref{Frobenius} in the way we explained in the Introduction. 
\begin{lem}\label{computationa} Let $v,w\in H_{K}$. Let $\left[\lambda,l\right]\in \mathcal{C}/{\sim}$ be a relevant equivalence class. Then, for every $X\in \mathfrak{a}_{\lambda_{0}}$ and every $m\in M_{\lambda}$, we have 
\begin{equation*}
\Gamma_{\lambda,l}(m,\dot{\pi}(X)v,w)=(\lambda|_{\mathfrak{a}_{\lambda_{0}}}-\rho_{\lambda_{0}})(X)\Gamma_{\lambda,l}(m,v,w).
\end{equation*}
\end{lem}
\begin{proof} We write $m\in M_{\lambda}$ as $m=\xi_{2}a_{\lambda}\xi_{2}$ for some $\xi_{1},\xi_{2}\in K_{\lambda}$ and some $a_{\lambda}\in \overline{A^{+}_{\lambda}}$. Then we have 
\begin{equation*}
\Gamma_{\lambda,l}(m,\dot{\pi}(X)v,w)=\Gamma_{\lambda,l}(a_{\lambda},\pi(\xi_{1})\dot{\pi}(X)v,\pi(\xi^{-1}_{2})w).
\end{equation*}
Recalling that 
\begin{equation*}
\pi(\xi_{1})\dot{\pi}(X)v=\dot{\pi}(\text{Ad}(\xi_{1})X)\pi(\xi_{1})v,
\end{equation*}
since $M_{\lambda}$ centralises $\mathfrak{a}_{\lambda_{0}}$ \cite[ Proposition 7.82]{Knapp2}, and $K_{\lambda}$ is contained in $M_{\lambda}$, we have 
\begin{equation*}
\Gamma_{\lambda,l}(a_{\lambda},\pi(\xi_{1})\dot{\pi}(X)v,\pi(\xi^{-1}_{2})w)=\Gamma_{\lambda,l}(a_{\lambda},\dot{\pi}(X)\pi(\xi_{1})v,\pi(\xi^{-1}_{2})w).
\end{equation*}
Therefore, re-labeling things, it suffices to prove that for every $X\in \mathfrak{a}_{\lambda_{0}}$ and for every $a_{\lambda}\in \overline{A^{+}_{\lambda}}$, we have 
\begin{equation*}
\Gamma_{\lambda,l}(a_{\lambda},\dot{\pi}(X)v,w)=(\lambda|_{\mathfrak{a}_{\lambda_{0}}}-\rho_{\lambda_{0}})(X)\Gamma_{\lambda,l}(a_{\lambda},v,w).
\end{equation*}
Moreover, since $\Gamma_{\lambda,l}(\cdot,v,w)$ is analytic, it suffices to prove the identity for every $a_{\lambda}\in A_{\lambda}^{+}$. \\
Let $a_{\lambda}=\text{exp}H_{\lambda}\in A^{+}_{\lambda}$. Then there exist a compact subset $C$ of $M_{\lambda}$ containing $a_{\lambda}$ and such that $K_{\lambda}CK_{\lambda}=C$, and a positive $R$ depending on $C$ such that, for all $H_{\lambda_{0}}\in \mathfrak{a}^{+}_{\lambda_{0}}$ satisfying $\alpha_{i}(H_{\lambda_{0}})>\text{log}R$ for every $i\in I^{c}_{\lambda}$, the expansion of $\phi_{\dot{\pi}(X)v,w}(a_{\lambda}\text{exp}H_{\lambda_{0}})$ relative to $P$ (Theorem \ref{asymp1}) and the expansion of $\phi_{\dot{\pi}(X)v,w}(a_{\lambda}\text{exp}H_{\lambda_{0}})$ relative to $P_{\lambda}$ (Theorem \ref{asymp2}) are both valid.\\
Setting $H:=H_{\lambda}+H_{\lambda_{0}}$ for $H_{\lambda_{0}}$ as above, the expansion in Theorem \ref{asymp1} gives 
\begin{equation*}
\phi_{\dot{\pi}(X)v,w}(H)=\sum_{\Tilde{\lambda}\in \mathcal{E}}\text{ }\sum_{\Tilde{l}\in \mathbb{Z}^{n}_{\geq 0}:|\Tilde{l}|\leq l_{0}}\alpha(H)^{\Tilde{l}}e^{(\Tilde{\lambda}-\rho)(H)}\langle c_{\Tilde{\lambda},\Tilde{l}}(\dot{\pi}(X)v),w\rangle
\end{equation*}
By linearity we can assume that $X=H_{i}$ for some $i\in I^{c}_{\lambda}$, where $H_{i}$, we recall, is the element in $\mathfrak{a}_{\lambda_{0}}$ dual to to the simple root $\alpha_{i}$.\\ 

Differentiating term by term and taking into account the computation 
\begin{equation*}
H_{i}\left[\alpha(H)^{\Tilde{l}}e^{(\Tilde{\lambda}-\rho)(H)}\right]=\Tilde{l}_{i}\alpha(H)^{\Tilde{l}-e_{i}}e^{(\Tilde{\lambda}-\rho)(H)}+(\Tilde{\lambda}|_{\mathfrak{a}_{\lambda_{0}}}-\rho)(H_{i})\alpha(H)^{\Tilde{l}}e^{(\Tilde{\lambda}-\rho)(H)},
\end{equation*}
where $e_{i}$ is the element in $\mathbb{Z}_{\geq 0}^{n}$ having $1$ as its $i$-th co-ordinate and $0$ as every other co-ordinate, we observe that the only terms in the expansion  
\begin{equation*}
\phi_{v,w}(H)=\sum_{\Tilde{\lambda}\in \mathcal{E}}\text{ }\sum_{\Tilde{l}\in \mathbb{Z}^{n}_{\geq 0}:|\Tilde{l}|\leq l_{0}}\alpha(H)^{\Tilde{l}}e^{(\Tilde{\lambda}-\rho)(H)}\langle c_{\Tilde{\lambda},\Tilde{l}}(v),w\rangle
\end{equation*}
that after differentiation by $H_{i}\in \mathfrak{a}_{\lambda_{0}}$ can contribute a term of the form 
\begin{equation*}
c\alpha(H)^{\Tilde{l}}e^{(\Tilde{\lambda}-\rho)(H)}\langle c_{\Tilde{\lambda},\Tilde{l}}(v),w\rangle,
\end{equation*}
with $c\in \mathbb{C}$, to the expansion of $\phi_{\dot{\pi}(X)v,w}(H)$, is precisely 
\begin{equation*}
\alpha(H)^{\Tilde{l}}e^{(\Tilde{\lambda}-\rho)(H)}\langle c_{\Tilde{\lambda},\Tilde{l}}(v),w\rangle.
\end{equation*}
This reasoning shows that in the expansion 
\begin{equation*}
\phi_{\dot{\pi}(H_{i})v,w}(a_{\lambda}\text{exp}H_{\lambda_{0}})=\sum_{\nu\in \mathcal{E}_{I}}\text{ }\sum_{q\in \mathbb{Z}^{I^{c}_{\lambda}}_{\geq 0}:|q|\leq q_{0}}\alpha(H_{\lambda_{0}})^{q}e^{(\nu-\rho_{\lambda_{0}})(H_{\lambda_{0}})}c_{\nu,q}^{P_{\lambda}}(a_{\lambda},\dot{\pi}(H_{i})v,w)
\end{equation*}
relative to $P_{\lambda}$, the term indexed by $(\nu,q)$ with $\nu=\lambda|_{\mathfrak{a}_{\lambda_{0}}}$ and $q=l_{\lambda_{0}}$ satisfies 
\begin{equation*}
c^{P_{\lambda}}_{\nu,q}(a_{\lambda},\dot{\pi}(H_{i})v,w)=(\lambda|_{\mathfrak{a}_{\lambda_{0}}}-\rho_{\lambda_{0}})(H_{i})c^{P_{\lambda}}_{\nu,q}(a_{\lambda},v,w).
\end{equation*}
Indeed, the comparison in \cite[p. 251]{Knapp}, shows that 
\begin{equation*}
\alpha(H_{\lambda_{0}})^{q}e^{(\nu-\rho_{\lambda_{0}})(H_{\lambda_{0}})}c^{P_{\lambda}}_{\nu,q}(a_{\lambda},\dot{\pi}(H_{i})v,w)
\end{equation*}
is the sum of all the terms in the expansion of $\phi_{\dot{\pi}(H_{i})v,w}(H)$ relative to $P$ which are indexed by couples $(\Tilde{\lambda},\Tilde{l})$ satisfying 
\begin{equation*}
\Tilde{\lambda}|_{\mathfrak{a}_{\lambda_{0}}}=\lambda|_{\mathfrak{a}_{\lambda_{0}}} \text{ and } \Tilde{l}_{\lambda_{0}}=l_{\lambda_{0}}
\end{equation*}
and, as we saw, these are the terms of the form 
\begin{equation*}(\lambda|_{\mathfrak{a}_{\lambda_{0}}}-\rho_{\lambda_{0}})(H_{i})\alpha(H)^{\Tilde{l}}e^{(\Tilde{\lambda}-\rho)(H)}\langle c_{\Tilde{\lambda},\Tilde{l}}(v),w\rangle.
\end{equation*}
Finally, since $$\Gamma_{\lambda,l}(a_{\lambda},v,w)=c^{P_{\lambda}}_{\nu,q}(a_{\lambda},v,w)$$ by Proposition \ref{Gamma}, we obtain 
\begin{equation*}
\Gamma_{\lambda,l}(a_{\lambda},v,w)=(\lambda|_{\mathfrak{a}_{\lambda_{0}}}-\rho_{\lambda_{0}})(H_{i})\Gamma_{\lambda,l}(a_{\lambda},\dot{\pi}(H_{i})v,w).
\end{equation*}
\end{proof}

\begin{lem}\label{computationn} Let $v,w\in H_{K}$. Let $\left[\lambda,l\right]\in \mathcal{C}/{\sim}$ be a relevant equivalence class. Then, for every $X\in \mathfrak{n}_{\lambda_{0}}$ and every $m\in M_{\lambda}$, we have 
\begin{equation*}
\Gamma_{\lambda,l}(m,\dot{\pi}(X)v,w)=0.
\end{equation*}
\end{lem}
\begin{proof} We write $m\in M_{\lambda}$ as $m=\xi_{2}a_{\lambda}\xi_{2}$ for some $\xi_{1},\xi_{2}\in K_{\lambda}$ and some $a_{\lambda}\in \overline{A^{+}_{\lambda}}$. Then we have 
\begin{equation*}
\Gamma_{\lambda,l}(m,\dot{\pi}(X)v,w)=\Gamma_{\lambda,l}(a_{\lambda},\pi(\xi_{1})\dot{\pi}(X)v,\pi(\xi^{-1}_{2})w).
\end{equation*}
Recalling that 
\begin{equation*}
\pi(\xi_{1})\dot{\pi}(X)v=\dot{\pi}(\text{Ad}(\xi_{1})X)\pi(\xi_{1})v,
\end{equation*}
since $M_{\lambda}$ normalises $\mathfrak{n}_{\lambda_{0}}$ \cite[ Proposition 7.83]{Knapp2}, and $K_{\lambda}$ is contained in $M_{\lambda}$, we have 
\begin{equation*}
\Gamma_{\lambda,l}(a_{\lambda},\pi(\xi_{1})\dot{\pi}(X)v,\pi(\xi^{-1}_{2})w)=\Gamma_{\lambda,l}(a_{\lambda},\dot{\pi}(X')\pi(\xi_{1})v,\pi(\xi^{-1}_{2})w)
\end{equation*}
for some $X'\in \mathfrak{n}_{\lambda_{0}}$. Therefore, re-labeling things, it suffices to prove that for every $X\in \mathfrak{a}_{\lambda_{0}}$ and for every $a_{\lambda}\in \overline{A^{+}_{\lambda}}$, we have 
\begin{equation*}\Gamma_{\lambda,l}(a_{\lambda},\dot{\pi}(X)v,w)=0.
\end{equation*}
Again, since $\Gamma_{\lambda,l}(\cdot,v,w)$ is analytic, it suffices to prove the identity for every $a_{\lambda}\in A_{\lambda}^{+}$. \\ 

As in the previous proof, we set $H:=H_{\lambda}+H_{\lambda_{0}}$ for $H_{\lambda_{0}}$ in an appropriate region and the expansion in Theorem \ref{asymp1} gives 
\begin{equation*}
\phi_{\dot{\pi}(X)v,w}(H)=\sum_{\tilde{\lambda}\in \mathcal{E}}\text{ }\sum_{\tilde{l}\in \mathbb{Z}^{n}_{\geq 0}:|\tilde{l}|\leq l_{0}}\alpha(H)^{\tilde{l}}e^{(\tilde{\lambda}-\rho)(H)}\langle c_{\tilde{\lambda},\tilde{l}}(\dot{\pi}(X)v),w\rangle.
\end{equation*}
The expansion in Theorem \ref{asymp2} gives 
\begin{equation*}
\phi_{\dot{\pi}(X)v,w}(a_{\lambda}\text{exp}H_{\lambda_{0}})=\sum_{\nu\in \mathcal{E}_{I}}\text{ }\sum_{q\in \mathbb{Z}^{I^{c}_{\lambda}}_{\geq 0}:|q|\leq q_{0}}\alpha(H_{\lambda_{0}})^{q}e^{(\nu-\rho_{\lambda_{0}})(H_{\lambda_{0}})}c_{\nu,q}^{P_{\lambda}}(a_{\lambda},\dot{\pi}(X)v,w).
\end{equation*}
By \cite[Corollary 8.46]{Knapp}, each $\nu-\rho_{\lambda_{0}}$ in the second expansion is of the form $\Tilde{\lambda}|_{\mathfrak{a}_{\lambda_{0}}}-\rho_{\lambda_{0}}$ for some exponent $\Tilde{\lambda}$ in the first expansion. Therefore, it suffices to prove that if $\lambda\in \mathcal{E}_{0}$ and $l\in \mathbb{Z}^{n}_{\geq 0}$ with $|l| \leq l_{0}$ satisfy 
\begin{equation*}
\textbf{d}_{P}(\lambda,l)=\textbf{d}(\pi),
\end{equation*}
then no term with exponent $\Tilde{\lambda}-\rho$ for which $\Tilde{\lambda}|_{\mathfrak{a}_{\lambda_{0}}}=\lambda|_{\mathfrak{a}_{\lambda_{0}}}$ appears in the first expansion. Indeed, if we can show this, since by the comparison in 
\cite[p. 251]{Knapp}, the term 
\begin{equation*}
\alpha(H_{\lambda_{0}})^{q}e^{(\nu-\rho_{\lambda_{0}})(H_{\lambda_{0}})}c^{P_{\lambda}}_{\nu,q}(a_{\lambda},\dot{\pi}(X)v,w),
\end{equation*}
for $\nu=\lambda|_{\mathfrak{a}_{\lambda_{0}}}$ and $q_{\lambda_{0}}=l_{\lambda_{0}}$ is the sum of all the terms in the expansion of $\phi_{\dot{\pi}(X)v,w}(H)$ relative to $P$ which are indexed by couples $(\Tilde{\lambda},\Tilde{l})$ satisfying 
\begin{equation*}
\Tilde{\lambda}|_{\mathfrak{a}_{\lambda_{0}}}=\lambda|_{\mathfrak{a}_{\lambda_{0}}} \text{ and } \Tilde{l}_{\lambda_{0}}=l_{\lambda_{0}},
\end{equation*}
it would follow that $$c^{P_{\lambda}}_{\nu,q}(a_{\lambda},\dot{\pi}(X)v,w)=0$$
and therefore 
\begin{equation*}
\Gamma_{\lambda,l}(a_{\lambda},\dot{\pi}(X)v,w)=0.
\end{equation*}
By linearity we can assume that $X\in \mathfrak{g}_{-\alpha_{i}}$ for some $i\in I^{c}_{\lambda}$ \cite[Proposition 5.23]{Knapp}.\\

Computing as in \cite[Lemma 8.16]{CassMil}, we have 
\begin{equation*}
\phi_{\dot{\pi}(X)v,w}(a)=\langle \dot{\pi}(\text{Ad}(a)X)\pi(a)v,w\rangle=-e^{-\alpha_{i}(H)}\phi_{ v,\dot{\pi}(X)w}(a).
\end{equation*}
Hence every exponent in the expansion of $\phi_{\dot{\pi}(X)v,w}(a)$ relative to $P$ is of the form $\Tilde{\lambda}=\lambda'-e_{i}$ for some $\lambda'\in \mathcal{E}$. Now, if there existed $\lambda'\in \mathcal{E}$ with 
\begin{equation*}
(\lambda'-e_{i})|_{\mathfrak{a}_{\lambda_{0}}}=\lambda|_{\mathfrak{a}_{\lambda_{0}}},
\end{equation*}
we would have 
\begin{equation*}
\text{Re}(\lambda'-e_{i})_{i}=\text{Re}\lambda_{i}=0
\end{equation*}
since $i\in I^{c}_{\lambda}$. This means that $\text{Re}\lambda'_{i}>0$, a contradiction. Indeed, since $(\pi,H)$ is tempered, the real part of every co-ordinate of each leading exponent is at most zero by Theorem \ref{tempcriterion} and it follows that the same property holds for every element in $\mathcal{E}$. This concludes the proof.
\end{proof} 

\begin{lem}\label{referee} Let $w\in H_{K}$. Let $\left[\lambda,l\right]\in \mathcal{C}/{\sim}$ be a relevant equivalence class. Then the prescription 
\begin{equation*}
S_{w}\colon H_{K}\longrightarrow L^{2}(M_{\lambda}),\text{ }S_{w}(v)(m):=\Gamma_{\lambda,l}(m,v,w)
\end{equation*}
is a well-defined, $(\mathfrak{m}_{\lambda},K_{\lambda})$-equivariant map with image contained in $L^{2}(M_{\lambda})_{K_{\lambda}}$.
\end{lem}
\begin{proof} The map $S_{w}$ is well-defined by Proposition \ref{GammaGar}.
For every $\xi\in K_{\lambda}$ and every $m\in M_{\lambda}$, we have 
\begin{equation*}
S_{w}(\pi(\xi)v)(m)=\Gamma_{\lambda,l}(m,\pi(\xi)v,w)=\Gamma_{\lambda,l}(m\xi,v,w)=R(\xi)S_{w}(v)(m).
\end{equation*}
By Proposition \ref{GammaGar}, for all $X\in \mathfrak{m}_{\lambda}$ and for all $m\in M_{\lambda}$, we have 
\begin{equation*}
S_{w}(\dot{\pi}(X)v)(m)=X\Gamma_{\lambda,l}(m,v,w)
\end{equation*}
and, by Proposition \ref{skewinv}, we have 
\begin{equation*}
X\Gamma_{\lambda,l}(m,v,w)=\dot{R}(X)\Gamma_{\lambda,l}(m,v,w).
\end{equation*}
Therefore 
\begin{equation*}
S_{w}(\dot{\pi}(X)v)(m)=\dot{R}(X)S_{w}(v)(m)
\end{equation*}
and this concludes the proof that $S_{w}$ is $(\mathfrak{m}_{\lambda},K_{\lambda})$-equivariant. \newline
To prove that the image of $S_{w}$ is contained in $L^{2}(M_{\lambda})_{K_{\lambda}}$, we observe that, for every $v\in H_{K}$, the $K_{\lambda}$-finiteness of $v$ implies the existence of finitely many $v_{1},\dots,v_{r}\in H_{K}$ such that 
\begin{equation*}
R(K_{\lambda})\Gamma_{\lambda,l}(\cdot,v,w)\in \text{span}\{\Gamma_{\lambda,l}(\cdot,v_{i},w)|i\in \{1,\dots,r\}\}.
\end{equation*}
Hence, $\Gamma_{\lambda,l}(\cdot,v,w)$ is $K_{\lambda}$-finite and, since it is a smooth vector in $(R,L^{2}(M_{\lambda}))$ by Proposition \ref{GammaGar}, it belongs to $L^{2}(M_{\lambda})_{K_{\lambda}}$. 
\end{proof}

We now construct a sub-representation $(\Theta,H_{\Theta})$ of $(R,L^{2}(M_{\lambda}))$ which, as we will show in the next two results, has precisely those properties that we need to proceed with the strategy outlined in the Introduction. We will show that $(\Theta,H_{\Theta})$ is an admissible, finitely generated, unitary (this follows since it is a sub-representation of $L^{2}(M_{\lambda})$) representation of $M_{\lambda}$ such that the image of the $(\mathfrak{m}_{\lambda},K_{\lambda})$-equivariant map $S_{w}$ is precisely the $(\mathfrak{m}_{\lambda},K_{\lambda})$-module $H_{\Theta,K_{\lambda}}$ and such that the map 

\begin{equation*}
S_{w}:H_{K}\longrightarrow H_{\Theta,K_{\lambda}}\otimes \mathbb{C}_{\lambda|_{\mathfrak{a}_{\lambda_{0}}}-\rho_{\lambda_{0}}},\text{ } S_{w}(v)(m):=\Gamma_{\lambda,l}(m,v,w)
\end{equation*}

is $(\mathfrak{m}_{\lambda}\oplus\mathfrak{a}_{\lambda_{0}},K_{\lambda})$-equivariant.\\

The representation $(\Theta,H_{\Theta})$ depends on the choice of $w\in H_{K}$ and of a relevant $[\lambda,l]\in\mathcal{C}/{\sim}$. However, and this is the important point, the construction can be formed for every choice of $\tilde{w}\in H_{k}$ and for every choice of relevant $[\tilde\lambda,\tilde l]\in\mathcal{C}/{\sim}$. 
In Proposition \ref{intrep} below, we will use this construction to define the representation $(\sigma,H_{\sigma})$ discussed in the Introduction.\\

We adopt the notation of the previous lemma. In the proof of Proposition \ref{GammaGar}, we showed that, for each $v\in H_{K}$, the function $\Gamma_{\lambda,l}(\cdot,v,w)$ is a $Z(\mathfrak{m}_{\lambda})$-finite function in $L^{2}(M_{\lambda})$. By \cite[Corollary 8.42]{Knapp}, there exist finitely many orthogonal irreducible sub-representations of $(R,L^{2}(M_{\lambda}))$ such that $\Gamma_{\lambda,l}(\cdot,v,w)$ is contained in their direct sum. It follows that there exists a (not necessarily finite) collection  $\{(\theta,H_{\theta})\}_{\theta\in \Theta}$ of orthogonal irreducible sub-representations of $(R,L^{2}(M_{\lambda}))$ such that $S_{w}(H_{K})$ is contained in their direct sum. Let $(\Theta,H_{\Theta})$ denote the direct sum of the sub-representations in this collection.\\

\begin{lem}\label{beforeapplyfrob} The $(\mathfrak{m}_{\lambda},K_{\lambda})$-module $H_{\Theta,K_{\lambda}}$ is precisely the image of the $(\mathfrak{m}_{\lambda},K_{\lambda})$-equivariant map 

\begin{equation*}
S_{w}:H_{K}\longrightarrow L^{2}(M_{\lambda}),\text{ }S_{w}(v)(m):=\Gamma_{\lambda,l}(m,v,w).
\end{equation*} 

\end{lem}
\begin{proof} It follows from Lemma \ref{referee} that $S_{w}(H_{K})\subset H_{\Theta}\cap L^{2}(M_{\lambda})_{K_{\lambda}}=H_{\Theta,K_{\lambda}}$.
For the reverse inclusion, the irreducibility of each $(\theta,H_{\theta})$ implies that $S_{w}(H_{K})\cap H_{\theta,K_{\lambda}}=H_{\theta,K_{\lambda}}$. Therefore $H_{\Theta,K_{\lambda}}$ is contained in the image of $S_{w}$, completing the proof. 

\end{proof}

\begin{prop}\label{applyfrob} The representation $(\Theta,H_{\Theta})$ of $M_{\lambda}$ is admissible, finitely generated and  unitary. Moreover, the map
\begin{equation*}
S_{w}:H_{K}\longrightarrow H_{\Theta,K_{\lambda}}\otimes \mathbb{C}_{\lambda|_{\mathfrak{a}_{\lambda_{0}}}-\rho_{\lambda_{0}}},\text{ } S_{w}(v)(m):=\Gamma_{\lambda,l}(m,v,w)
\end{equation*}
is $(\mathfrak{m}_{\lambda}\oplus\mathfrak{a}_{\lambda_{0}},K_{\lambda})$-equivariant.
\end{prop}
\begin{proof} By Lemma \ref{beforeapplyfrob} we have $S_{w}(H_{K})=H_{\Theta,K_{\lambda}}$.  
By Lemma \ref{computationa}, for all $X\in \mathfrak{a}_{\lambda_{0}}$ and for all $m\in M_{\lambda}$, we have 
\begin{equation*}
S_{w}(\dot{\pi}(X)v)(m)=(\lambda|_{\mathfrak{a}_{\lambda_{0}}}-\rho_{\lambda_{0}})(X)\Gamma_{\lambda,l}(m,v,w)=(\lambda|_{\mathfrak{a}_{\lambda_{0}}}-\rho_{\lambda_{0}})(X)S_{w}(v).
\end{equation*}
By Lemma \ref{computationn}, for all $X\in \mathfrak{n}_{\lambda_{0}}$ and for all $m\in M_{\lambda}$, we have 
\begin{equation*}
S_{w}(\dot{\pi}(X)v)(m)=\Gamma_{\lambda,l}(m,\dot{\pi}(X)v,w)=0.
\end{equation*}
We thus obtained an $(\mathfrak{m}_{\lambda}\oplus \mathfrak{a}_{\lambda_{0}},K_{\lambda})$-equivariant map 
\begin{equation*}
S_{w}:H_{K}\longrightarrow H_{\Theta,K_{\lambda}}\otimes \mathbb{C}_{\lambda|_{\mathfrak{a}_{\lambda_{0}}}-\rho_{\lambda_{0}}},\text{ }S_{w}(v)(m):=\Gamma_{\lambda,l}(m,v,w)
\end{equation*}
which factors through the quotient map 
\begin{equation*}
q:H_{K}\longrightarrow H_{K}/\mathfrak{n}_{\lambda_{0}}H_{K}
\end{equation*}
which is $(\mathfrak{m}_{\lambda}\oplus \mathfrak{a}_{\lambda_{0}},K_{\lambda})$-equivariant by Lemma \ref{nV}.  

Since $H_{K}$, being irreducible (and hence admissible by Theorem \ref{HC}), has an infinitesimal character, by Corollary \ref{corquotadfin} the $(\mathfrak{m}_{\lambda}\oplus\mathfrak{a}_{\lambda_{0}},K_{\lambda})$-module $H_{K}/\mathfrak{n}_{\lambda_{0}}H_{K}$ is admissible and finitely generated. It follows that 
\begin{equation*}
S_{w}(H_{K})=H_{\Theta,K_{\lambda}}\otimes \mathbb{C}_{\lambda|_{\mathfrak{a}_{\lambda_{0}}}-\rho_{\lambda_{0}}}
\end{equation*}
is an admissible and finitely generated $(\mathfrak{m}_{\lambda}\oplus\mathfrak{a}_{\lambda_{0}},K_{\lambda})$-module. The fact that $\mathfrak{a}_{\lambda_{0}}$ acts by scalars, implies that $H_{\Theta,K_{\lambda}}$ itself is finitely generated (as $U(\mathfrak{m}_{\lambda\mathbb{C}})$-module) and admissible.
\end{proof}

In the next corollary, we apply Casselman's version of the Frobenius reciprocity to construct $(\mathfrak{g},K)$-intertwining operators from the functions $\Gamma_{\lambda,l}$. We recall that $\overline{P_{\lambda}}$ denotes the parabolic subgroup opposite to $P_{\lambda}$ and that the half-sum of positive roots determined by $\overline{P_{\lambda}}$ is precisely $-\rho_{\lambda_{0}}$.

\begin{cor}\label{CassFrob} The map 
\[T_{w}\colon H_{K}\longrightarrow \text{Ind}_{\overline{P_{\lambda}},K_{\lambda}}(\Theta,\lambda|_{\mathfrak{a}_{\lambda_{0}}}),\;\;
T_{w}(v)(k)(m):=\Gamma_{\lambda,l}(m,\pi(k)v,w)\] is $(\mathfrak{g},K)$-equivariant.
\end{cor}
\begin{proof} The equivariance follows from Proposition \ref{applyfrob}, in combination with Theorem \ref{Frobenius} and the discussion following it. More precisely, we have $T_{w}=\Tilde{S}_{w}$ in the notation of the discussion following Theorem \ref{Frobenius}. 
\end{proof}

The next proposition is the core of the article: it allows us to prove an identity of certain integrals using representation-theoretic methods. In the final section, it will be shown that the identity in question implies Proposition \ref{ginv2}.

\begin{prop}\label{intrep} Let $[\lambda,l],[\mu,m]\in \mathcal{C}/{\sim}$ be relevant equivalence classes such that $I_{\lambda}=I_{\mu}$, $\lambda|_{\mathfrak{a}_{\lambda_{0}}}=\mu|_{\mathfrak{a}_{\lambda_{0}}}$ and $\textbf{d}_{P}(\lambda,l)=\textbf{d}_{P}(\mu,m)$. Then, for all $X\in \mathfrak{g}$, for all $k\in K$, and for all $v_{1},v_{2},v_{3},v_{4}\in H_{K}$, the integral

\begin{equation*}
\int_{K}\langle \Gamma_{\lambda,l}(m_{\lambda},\pi(k)\dot{\pi}(X)v_{1},v_{2}),\Gamma_{\mu,m}(m_{\lambda},\pi(k)v_{3},v_{4})\rangle_{L^{2}(M_{\lambda})}\,dk
\end{equation*}
is equal to the integral
\begin{equation*}
-\int_{K}\langle \Gamma_{\lambda,l}(m_{\lambda},\pi(k)v_{1},v_{2}),\Gamma_{\mu,m}(m_{\lambda},\pi(k)\dot{\pi}(X)v_{3},v_{4})\rangle_{L^{2}(M_{\lambda})}\,dk.
\end{equation*}

\end{prop}
\begin{proof} By Proposition \ref{applyfrob} and the discussion before Lemma \ref{referee}, we can construct a representation $(\sigma_{1},H_{\sigma_{1}})$ of $M_{\lambda}$ that is finitely generated, unitary and such that the image of the $(\mathfrak{m}_{\lambda}\oplus \mathfrak{a}_{\lambda_{0}},K_{\lambda})$-equivariant map 
\begin{equation*}
S_{v_{2}}:H_{K}\longrightarrow L^{2}(M_{\lambda})_{K_{\lambda}}\otimes \mathbb{C}_{\lambda|_{\mathfrak{a}_{\lambda_{0}}}-\rho_{\lambda_{0}}},\text{ }S_{v_{2}}(v)(m):=\Gamma_{\lambda,l}(m,v,v_{2})
\end{equation*}
is precisely $H_{\sigma_{1},K_{\lambda}}\otimes \mathbb{C}_{\lambda|_{\mathfrak{a}_{\lambda_{0}}}-\rho_{\lambda_{0}}}$. Similarly, we can construct an admissible, finitely generated, unitary representation $(\sigma_{2},H_{\sigma_{2}})$ such that the image of the $(\mathfrak{m}_{\lambda}\oplus \mathfrak{a}_{\lambda_{0}},K_{\lambda})$-equivariant map 
\begin{equation*}
S_{v_{4}}:H_{K}\longrightarrow L^{2}(M_{\lambda})_{K_{\lambda}}\otimes \mathbb{C}_{\mu|_{\mathfrak{a}_{\lambda_{0}}}-\rho_{\lambda_{0}}},\text{ }S_{v_{4}}(v)(m):=\Gamma_{\mu,m}(m,v,v_{4})
\end{equation*}
is precisely $H_{\sigma_{2},K_{\lambda}}\otimes \mathbb{C}_{\mu|_{\mathfrak{a}_{\lambda_{0}}}-\rho_{\lambda_{0}}}$. Let $(\sigma,H_{\sigma})$ denote the direct sum of $(\sigma_{1},H_{\sigma_{1}})$ and $(\sigma_{2},H_{\sigma_{2}})$. It is an an admissible, finitely generated, unitary representation which restricts to a unitary representation of $K_{\lambda}$. Since $\lambda|_{\mathfrak{a}_{\lambda_{0}}}=\mu|_{\mathfrak{a}_{\lambda_{0}}}$, by the same computations as in Lemma \ref{referee} and Proposition \ref{applyfrob} we obtain \newline $(\mathfrak{m}_{\lambda}\oplus \mathfrak{a}_{\lambda_{0}},K_{\lambda})$-equivariant maps 
\begin{equation*}
S_{v_{2}}:H_{K}\longrightarrow H_{\sigma,K_{\lambda}}\otimes \mathbb{C}_{\lambda|_{\mathfrak{a}_{\lambda_{0}}}-\rho_{\lambda_{0}}},\text{ }S_{v_{2}}(v)(m):=\Gamma_{\lambda,l}(m,v,v_{2})
\end{equation*}
and 
\begin{equation*}
S_{v_{4}}:H_{K}\longrightarrow H_{\sigma,K_{\lambda}}\otimes \mathbb{C}_{\lambda|_{\mathfrak{a}_{\lambda_{0}}}-\rho_{\lambda_{0}}},\text{ }S_{v_{4}}(v)(m):=\Gamma_{\mu,m}(m,v,v_{4})
\end{equation*}
factoring through the $(\mathfrak{m}_{\lambda}\oplus\mathfrak{a}_{\lambda_{0}},K_{\lambda})$-equivariant quotient map 
\begin{equation*}
q:H_{K}\longrightarrow H_{K}/\mathfrak{n}_{\lambda_{0}}H_{K}.
\end{equation*}
From Corollary \ref{CassFrob}, we obtain $(\mathfrak{g},K)$-equivariant maps
\begin{equation*}
T_{v_{2}}:H_{K}\longrightarrow \text{Ind}_{\overline{P_{\lambda}},K_{\lambda}}(\sigma,\lambda|_{\mathfrak{a}_{\lambda_{0}}}),\text{ }T_{v_{2}}(v)(k)(m):=\Gamma_{\lambda,l}(m,\pi(k)v,v_{2})
\end{equation*}
and
\begin{equation*}
T_{v_{4}}:H_{K}\longrightarrow \text{Ind}_{\overline{P_{\lambda}},K_{\lambda}}(\sigma,\lambda|_{\mathfrak{a}_{\lambda_{0}}}),\text{ }T_{v_{4}}(v)(k)(m):=\Gamma_{\mu,m}(m,\pi(k)v,v_{4}).
\end{equation*}
By definition of the inner product on $\text{Ind}_{\overline{P_{\lambda}}}(\sigma,\lambda|_{\mathfrak{a}_{\lambda_{0}}})$, we see that proving the sought identity is equivalent to proving that 
\begin{equation*}
\langle T_{v_{2}}(\dot{\pi}(X)v_{1}),T_{v_{4}}(v_{3})\rangle_{\text{Ind}_{\overline{P_{\lambda}}}(\sigma,\lambda|_{\mathfrak{a}_{\lambda_{0}}})}=-\langle T_{v_{2}}(v_{1}),T_{v_{4}}(\dot{\pi}(X)v_{3})\rangle_{\text{Ind}_{\overline{P_{\lambda}}}(\sigma,\lambda|_{\mathfrak{a}_{\lambda_{0}}})}.
\end{equation*}
By the $(\mathfrak{g},K)$-equivariance of $T_{v_{2}}$, we have
\begin{equation*}
\langle T_{v_{2}}(\dot{\pi}(X)v_{1}),T_{v_{4}}(v_{3})\rangle_{\text{Ind}_{\overline{P_{\lambda}}}(\sigma,\lambda|_{\mathfrak{a}_{\lambda_{0}}})}=\langle \dot{\text{Ind}}_{\overline{P_{\lambda}}}(\sigma,\lambda|_{\mathfrak{a}_{\lambda_{0}}},X)T_{v_{2}}(v_{1}),T_{v_{4}}(v_{3})\rangle_{\text{Ind}_{\overline{P_{\lambda}}}(\sigma,\lambda|_{\mathfrak{a}_{\lambda_{0}}})}
\end{equation*}
and, since $\lambda|_{\mathfrak{a}_{\lambda_{0}}}$ is totally imaginary, from Corollary \ref{corparabunit} we deduce 
\[\begin{split}\langle \dot{\text{Ind}}_{\overline{P_{\lambda}}}(\sigma,\lambda|_{\mathfrak{a}_{\lambda_{0}}},X)T_{v_{2}}(v_{1}),T_{v_{4}}(v_{3})\rangle_{\text{Ind}_{\overline{P_{\lambda}}}(\sigma,\lambda|_{\mathfrak{a}_{\lambda_{0}}})}&\cr
=-\langle T_{v_{2}}(v_{1}),&\dot{\text{Ind}}_{\overline{P_{\lambda}}}(\sigma,\lambda|_{\mathfrak{a}_{\lambda_{0}}},X)T_{v_{4}}(v_{3})\rangle_{\text{Ind}_{\overline{P_{\lambda}}}(\sigma,\lambda|_{\mathfrak{a}_{\lambda_{0}}})}.
\end{split}\]
The result follows from the $(\mathfrak{g},K)$-equivariance of $T_{v_{4}}$. 
\end{proof}


\section{Asymptotic Orthogonality} \label{sec:4}

For a tempered, irreducible representation $(\pi,H)$ of $G$, for $v,w\in H$, let 
\begin{equation*}
\phi_{v,w}(g):=\langle \pi(g)v,w\rangle
\end{equation*}
denote the associated matrix coefficient. By (2) of Theorem , there exists $\textbf{d}(\pi)\in\mathbb{Z}_{\geq0}$ such that 
\begin{equation*}
\lim_{r\rightarrow \infty}\frac{1}{r^{\textbf{d}(\pi)}}\int_{G_{<r}}|\phi_{v,w}(g)|^{2}\,dg<\infty
\end{equation*}
for all $v,w\in H_{K}$. \\

As in \cite[Section 4.1]{KYD}, by the polarisation identity and by (2) of Theorem \ref{kfin}, the prescription 
\begin{equation*}
D(v_{1},v_{2},v_{3},v_{4}):=\lim_{r\to\infty}\frac{1}{r^{\textbf{d}(\pi)}}\int_{G_{<r}} \phi_{v_{1},v_{2}}(g) \overline{\phi_{v_{3},v_{4}}(g)}\,dg
\end{equation*}
is a well-defined form on $H_{K}$ that is linear in the first and fourth variable, conjugate-linear in the second and the third.\\

We explained in the Introduction that the crucial point is the proof of Proposition \ref{ginv2}. We begin with the following reduction.
\begin{lem}\label{redginv} Let $G$ be a connected, semisimple Lie group with finite centre and let $(\pi,H)$ be a tempered, irreducible representation of $G$. If for all $X\in \mathfrak{g}$ and for all $v_{1},v_{2},v_{3},v_{4}\in H_{K}$ we have 
\begin{equation*}
\lim_{r\rightarrow \infty}\frac{1}{r^{\textbf{d}(\pi)}}\int_{G_{<r}}\phi_{\dot{\pi}(X)v_{1},v_{2}}(g)\overline{\phi_{v_{3},v_{4}}(g)}\,dg=-\lim_{r\rightarrow \infty}\frac{1}{r^{\textbf{d}(\pi)}}\int_{G_{<r}}\phi_{v_{1},v_{2}}(g)\overline{\phi_{\dot{\pi}(X)v_{3},v_{4}}(g)}\,dg,
\end{equation*}
then the equality 
\begin{equation*}
\lim_{r\rightarrow \infty}\frac{1}{r^{\textbf{d}(\pi)}}\int_{G_{<r}}\phi_{v_{1},\dot{\pi}(X)v_{2}}(g)\overline{\phi_{v_{3},v_{4}}(g)}\,dg=-\lim_{r\rightarrow \infty}\frac{1}{r^{\textbf{d}(\pi)}}\int_{G_{<r}}\phi_{v_{1},v_{2}}(g)\overline{\phi_{v_{3},\dot{\pi}v_{4}}(g)}\,dg
\end{equation*}
holds for every $X\in \mathfrak{g}$ and for every $v_{1},v_{2},v_{3},v_{4}\in H_{K}$.
\end{lem}
\begin{proof} We write 
\begin{equation*}
\phi_{v_{1},\dot{\pi}(X)v_{2}}(g)\overline{\phi_{v_{3},v_{4}}(g)}=\langle v_{1},\pi(g^{-1})\dot{\pi}(X)v_{2}\rangle\overline{\langle v_{3},\pi(g^{-1})v_{4}\rangle}
\end{equation*}
and since $\langle\cdot,\cdot\rangle$ is Hermitian we have 
\begin{equation*}
\langle v_{1},\pi(g^{-1})\dot{\pi}(X)v_{2}\rangle\overline{\langle v_{3},\pi(g^{-1})v_{4}\rangle}=\phi_{v_{4},v_{3}}(g^{-1})\overline{ \phi_{\dot{\pi}(X)v_{2},v_{1}}(g^{-1})}.
\end{equation*}
Now, since $G_{<r}$ is invariant under $\iota(g)=g^{-1}$ and $G$ is unimodular, we have 
\begin{equation*}
\int_{G_{<r}}\phi_{v_{4},v_{3}}(g^{-1})\overline{ \phi_{\dot{\pi}(X)v_{2},v_{1}}(g^{-1})}\,dg=\int_{G_{<r}} \phi_{v_{4},v_{3}}(g)\overline{\phi_{\dot{\pi}(X)v_{2},v_{1}}(g)}\,dg
\end{equation*}
and therefore
\begin{equation*}
\int_{G_{<r}}\phi_{v_{1},\dot{\pi}(X)v_{2}}(g)\overline{\phi_{v_{3},v_{4}}(g)}\,dg=\int_{G_{<r}} \phi_{v_{4},v_{3}}(g)\overline{\phi_{\dot{\pi(X)v_{2},v_{1}}}(g)}\,dg.
\end{equation*}
Applying complex conjugation, we obtain 
\begin{equation*}
\overline{\int_{G_{<r}}\phi_{v_{1},\dot{\pi}(X)v_{2}}(g)\overline{\phi_{v_{3},v_{4}}(g)}\,dg}=\int_{G_{<r}}\phi_{\dot{\pi}(X)v_{2},v_{1}}(g)\overline{\phi_{v_{4},v_{3}}(g)}\,dg.
\end{equation*}
Assuming the validity of the first identity in the statement, we can write 
\begin{equation*}
\lim_{r\rightarrow \infty}\frac{1}{r^{\textbf{d}(\pi)}}\int_{G_{<r}}\phi_{\dot{\pi}(X)v_{2},v_{1}}(g)\overline{\phi_{v_{4},v_{3}}(g)}\,dg=-\lim_{r\rightarrow \infty}\frac{1}{r^{\textbf{d}(\pi)}}\int_{G_{<r}}\phi_{v_{2},v_{1}}(g)\overline{\phi_{\dot{\pi}(X)v_{4},v_{3}}(g)}\,dg.
\end{equation*}
Now, since 
\begin{equation*}
\int_{G_{<r}}\phi_{\dot{\pi}(X)v_{2},v_{1}}(g)\overline{\phi_{v_{4},v_{3}}(g)}\,dg=\overline{\int_{G_{<r}}\phi_{v_{1},\dot{\pi}(X)v_{2}}(g)\overline{\phi_{v_{3},v_{4}}(g)}\,dg},
\end{equation*}
it follows that 
\begin{equation*}
\lim_{r\rightarrow \infty}\frac{1}{r^{\textbf{d}(\pi)}}\int_{G_{<r}}\phi_{v_{1},\dot{\pi}(X)v_{2}}(g)\overline{\phi_{v_{3},v_{4}}(g)}\,dg=-\lim_{r\rightarrow \infty}\overline{\int_{G_{<r}}\phi_{v_{2},v_{1}}(g)\overline{\phi_{\dot{\pi}(X)v_{4},v_{3}}(g)}\,dg}.
\end{equation*}
Observing that 
\begin{equation*}
\overline{\int_{G_{<r}}\phi_{v_{2},v_{1}}(g)\overline{\phi_{\dot{\pi}(X)v_{4},v_{3}}(g)}\,dg}=\int_{G_{<r}}\phi_{\dot{\pi}(X)v_{4},v_{3}}(g)\overline{\phi_{v_{2},v_{1}}(g)}\,dg
\end{equation*}
and that, using the invariance of $G_{<r}$ under $\iota(g)=g^{-1}$ and the unimodularity of $G$, we have 
\begin{equation*}
\int_{G_{<r}}\phi_{\dot{\pi}(X)v_{4},v_{3}}(g)\overline{\phi_{v_{2},v_{1}}(g)}\,dg=\int_{G_{<r}}\phi_{v_{1},v_{2}}(g)\overline{\phi_{v_{3},\dot{\pi}(X)v_{4}}(g)}\,dg,
\end{equation*}
we finally obtain 
\begin{equation*}
\lim_{r\rightarrow \infty}\frac{1}{r^{\textbf{d}(\pi)}}\int_{G_{<r}}\phi_{v_{1},\dot{\pi}(X)v_{2}}(g)\overline{\phi_{v_{3},v_{4}}(g)}\,dg=-\lim_{r\rightarrow \infty}\frac{1}{r^{\textbf{d}(\pi)}}\int_{G_{<r}}\phi_{v_{1},v_{2}}(g)\overline{\phi_{v_{3},\dot{\pi}(X)v_{4}}(g)}\,dg. 
\end{equation*}
\end{proof}

\begin{prop}\label{ginv} Let $G$ be a connected, semisimple Lie group with finite centre and let $(\pi,H)$ be a tempered, irreducible representation of $G$. Then, for all $X\in\mathfrak{g}$ and for all $v_{1},v_{2},v_{3},v_{4}\in H_{K}$, we have 
\begin{equation*}\lim_{r\rightarrow \infty}\frac{1}{r^{\textbf{d}(\pi)}}\int_{G_{<r}}\phi_{\dot{\pi}(X)v_{1},v_{2}}(g)\overline{\phi_{v_{3},v_{4}}(g)}\,dg=-\lim_{r\rightarrow \infty}\frac{1}{r^{\textbf{d}(\pi)}}\int_{G_{<r}}\phi_{v_{1},v_{2}}(g)\overline{\phi_{\dot{\pi}(X)v_{3},v_{4}}(g)}\,dg.
\end{equation*}
\end{prop}
\begin{rem} Some of the integral manipulations in the proof require careful justification. We decided to provided it in Lemma \ref{Interchange} after the proof of Proposition \ref{ginv}. 
\end{rem}
\begin{proof} The integral formula for the Cartan decomposition, taking into account the fact that, except for a set of measure zero, every $g\in G_{<r}$ can be written as $g=k_{2}\,\mathrm{exp}H\,k_{1}$, for some $k_{1},k_{2}\in K$ and some $H\in \mathfrak{a}_{<r}^{+}$, with $\mathfrak{a}_{<r}^{+}$ as in \eqref{defregiona}, gives 
\[\begin{split}
\int_{G_{<r}}\phi_{\dot{\pi}(X)v_{1},v_{2}}(g)\overline{\phi_{v_{3},v_{4}}(g)}\,dg& \cr
=\int_{K}\int_{\mathfrak{a}_{<r}^{+}}\int_{K}&\phi_{\dot{\pi}(X)v_{1},v_{2}}(k_{2}\,\text{exp}H\,k_{1})\overline{\phi_{v_{3},v_{4}}(k_{2}\,\text{exp}H\,k_{1})}\Omega(H)\,dk_{1}\,dH\,dk_{2}
\end{split}\]
with $\Omega(H)$ defined in \eqref{Omega}.

Arguing as in \cite[p. 258]{KYD}, we can interchange the two innermost integrals in the RHS and, upon multiplying both sides by $\displaystyle\frac{1}{r^{\textbf{d}(\pi)}}$ and taking the limit as $r\rightarrow \infty$, the RHS can be computed as the integral over $K\times K$ of 
\begin{equation*}
\lim_{r\rightarrow \infty}\frac{1}{r^{\textbf{d}(\pi)}}\int_{\mathfrak{a}^{+}_{<r}}\phi_{\dot{\pi}(X)v_{1},v_{2}}(k_{2}\,\text{exp}H\,k_{1})\overline{\phi_{v_{3},v_{4}}(k_{2}\,\text{exp}H\,k_{1})}\Omega(H)\,dH.
\end{equation*}
We expand $\phi_{v_{1},v_{2}}$ and $\phi_{v_{3},v_{4}}$ as 
\begin{equation*}
\phi_{v_{1},v_{2}}(k_{2}\,\text{exp}H\,k_{1})=e^{-\rho_{\mathfrak{p}}(H)}\sum_{\left[\lambda,l\right]\in \mathcal{C}/{\sim}}\alpha(H_{\lambda_{0}})^{l_{\lambda_{0}}}e^{\lambda|_{\mathfrak{a}_{\lambda_{0}}}(H_{\lambda_{0}})}\sum_{(\lambda',l')\in \left[\lambda,l\right]}\Psi^{\pi(k_{1})v_{1},\pi(k^{-1}_{2})v_{2}}_{\lambda',l'}(H)
\end{equation*}
and
\[\begin{split}
\phi_{v_{3},v_{4}}(k_{2}\,\text{exp}H\,k_{1})\,=\,&e^{-\rho_{\mathfrak{p}}(H)}\sum_{\left[\mu,m\right]\in \mathcal{C}/{\sim}}\alpha(H_{\mu_{0}})^{m_{\mu_{0}}}e^{\mu|_{\mathfrak{a}_{\mu_{0}}}(H_{\mu_{0}})}\cr
&\sum_{(\mu',m')\in \left[\mu,m\right]}\Psi^{\pi(k_{1})v_{1},\pi(k^{-1}_{2})v_{2}}_{\mu',m'}(H).
\end{split}\]
By \cite[Lemma A.5 and Claim A.6]{KYD}, the only non-zero contributions to 
\begin{equation}\label{asympint2}
\lim_{r\rightarrow \infty}\frac{1}{r^{\textbf{d}(\pi)}}\int_{\mathfrak{a}^{+}_{<r}}\phi_{\dot{\pi}(X)v_{1},v_{2}}(k_{2}\,\text{exp}H\,k_{1})\overline{\phi_{v_{3},v_{4}}(k_{2}\,\text{exp}H\,k_{1})}\Omega(H)\,dH
\end{equation}
may come from those $\left[\lambda,l\right]\in \mathcal{C}/{\sim}$ and those $\left[\mu,m\right]\in \mathcal{C}/{\sim}$ for which
\[I_{\lambda}=I_{\mu},\quad \lambda|_{\mathfrak{a}_{\lambda_{0}}}=\mu|_{\mathfrak{a}_{\lambda_{0}}}\quad\text{and}\quad
\textbf{d}(\pi)=|I_{\lambda}|+\sum_{i\in I_{\lambda}}(l_{i}+m_{i}).
\]
In view of the first condition, the third is equivalent to requiring that 
\begin{equation*}
\textbf{d}(\pi)=\textbf{d}_{P}(\lambda,l)=\textbf{d}_{P}(\mu,m),
\end{equation*}
where $\textbf{d}_{P}(\lambda,l)$ and $\textbf{d}_{P}(\mu,m)$ are defined by \eqref{defdp}.\\

By the discussion in Section~\ref{sec:3} and by Proposition~\ref{Gamma}, the expression 
\begin{equation*}
\lim_{r\rightarrow \infty}\frac{1}{r^{\textbf{d}(\pi)}}\int_{\mathfrak{a}^{+}_{<r}}\phi_{\dot{\pi}(X)v_{1},v_{2}}(k_{2}\,\text{exp}H\,k_{1})\overline{\phi_{v_{3},v_{4}}(k_{2}\,\text{exp}H\,k_{1})}\Omega(H)\,dH
\end{equation*}
is equal to a finite sum terms of the form 
\[\begin{split}
C(\lambda,l,m)\int_{\mathfrak{a}^{+}_{\lambda}}\Gamma_{\lambda,l}(\text{exp}H_{\lambda},\pi(k_{1})\dot{\pi}(X)v_{1},\pi(k_{2}^{-1})v_{2})&\cr
\overline{\Gamma_{\mu,m}(\text{exp}H_{\lambda},\pi(k_{1})v_{3},\pi(k_{2}^{-1})v_{4})}&\Omega_{\lambda}(H_{\lambda})\,dH_{\lambda},
\end{split}\]
with $C(\lambda,l,m)$ as in \eqref{defC}, the functions $\Gamma_{\lambda,l}$ and $\Gamma_{\mu,m}$ defined as in \eqref{defgamma} and $\Omega_{\lambda}(H_{\lambda})$ defined as in \eqref{defomegalambda}.\\

Taking into account the integration over $K\times K$, we proved that 
\begin{equation*}
\lim_{r\rightarrow \infty}\frac{1}{r^{\textbf{d}(\pi)}}\int_{K}\int_{\mathfrak{a}_{<r}^{+}}\int_{K}\phi_{\dot{\pi}(X)v_{1},v_{2}}(k_{2}\,\text{exp}H\,k_{1})\overline{\phi_{v_{3},v_{4}}(k_{2}\,\text{exp}H\,k_{1})}\Omega(H)\,dk_{1}\,dH\,dk_{2}
\end{equation*}
is equal to a finite sum of terms of the form 
\[\begin{split}
C(\lambda,l,m)\int_{K}\int_{K}\int_{\mathfrak{a}^{+}_{\lambda}}\Gamma_{\lambda,l}(\text{exp}H_{\lambda},\pi(k_{1})\dot{\pi}(X)v_{1},\pi(k_{2}^{-1})v_{2})&\cr
\overline{\Gamma_{\mu,m}(\text{exp}H_{\lambda},\pi(k_{1})v_{3},\pi(k_{2}^{-1})v_{4})}&\Omega_{\lambda}(H_{\lambda})\,dH_{\lambda}\,dk_{1}\,dk_{2}.
\end{split}\]
By (1) of Lemma \ref{Interchange} and applying the Fubini-Tonelli theorem, we can interchange the two innermost integral and we therefore need to prove that 
\[\begin{split}
\int_{K}\int_{\mathfrak{a}^{+}_{\lambda}}\int_{K}\Gamma_{\lambda,l}(\text{exp}H_{\lambda},\pi(k_{1})\dot{\pi}(X)v_{1},\pi(k_{2}^{-1})v_{2})&\cr
\overline{\Gamma_{\mu,m}(\text{exp}H_{\lambda},\pi(k_{1})v_{3},\pi(k_{2}^{-1})v_{4})}&\Omega_{\lambda}(H_{\lambda})\,dk_{1}\,dH_{\lambda}\,dk_{2}
\end{split}\]
is equal to 
\[\begin{split}
-\int_{K}\int_{\mathfrak{a}^{+}_{\lambda}}\int_{K}\Gamma_{\lambda,l}(\text{exp}H_{\lambda},\pi(k_{1})v_{1},\pi(k_{2}^{-1})v_{2})&\cr
\overline{\Gamma_{\mu,m}(\text{exp}H_{\lambda},\pi(k_{1})\dot{\pi}(X)v_{3},\pi(k_{2}^{-1})v_{4})}&\Omega_{\lambda}(H_{\lambda})\,dk_{1}\,dH_{\lambda}\,dk_{2}.
\end{split}\]
Set 
\begin{equation*}
\mathcal{I}(\text{exp}H_{\lambda},k_{1},k_{2}^{-1}):=\Gamma_{\lambda,l}(\text{exp}_{\lambda},\pi(k_{1})\dot{\pi}(X)v_{1},\pi(k_{2}^{-1})v_{2})\overline{\Gamma_{\mu,m}(\text{exp}H_{\lambda},\pi(k_{1})v_{3},\pi(k_{2}^{-1})v_{4})}.
\end{equation*}
We apply the quotient integral formula \cite[Theorem 2.51]{Folland}, to write the integral 
\begin{equation*}
\int_{K}\int_{\mathfrak{a}^{+}_{\lambda}}\int_{K}\mathcal{I}(\text{exp}H_{\lambda},k_{1},k_{2}^{-1})\Omega(H_{\lambda})\,dk_{1}\,dH_{\lambda}\,dk_{2}
\end{equation*}
as 
\begin{equation*}
\int_{K}\int_{\mathfrak{a}^{+}_{\lambda}}\int_{K_{\lambda}\backslash K}\int_{K_{\lambda}}\mathcal{I}(\text{exp}H_{\lambda},\xi_{1}k_{1},k_{2}^{-1})\Omega_{\lambda}(H_{\lambda})\,d\xi_{1}\,d\dot{k}_{1}\,dH_{\lambda}\,dk_{2}
\end{equation*}
and again to write it as 
\begin{equation*}
\int_{K/K_{\lambda}}\int_{K_{\lambda}}\int_{\mathfrak{a}^{+}_{\lambda}}\int_{K_{\lambda}\backslash K}\int_{K_{\lambda}}\mathcal{I}(\text{exp}H_{\lambda},\xi_{1}k_{1},\xi_{2}^{-1}k_{2}^{-1})\Omega_{\lambda}(H_{\lambda})\,d\xi_{1}\,d\dot{k}_{1}\,dH_{\lambda}\,d\xi_{2}\,d\dot{k}_{2}.
\end{equation*}
By (3) of Lemma \ref{Interchange}, we can appeal to the Fubini-Tonelli theorem to interchange the two innermost integrals and to obtain 
\begin{equation*}
\int_{K/K_{\lambda}}\int_{K_{\lambda}}\int_{\mathfrak{a}^{+}_{\lambda}}\int_{K_{\lambda}}\int_{K_{\lambda}\backslash K}\mathcal{I}(\text{exp}H_{\lambda},\xi_{1}k_{1},\xi_{2}^{-1}k_{2}^{-1})\Omega_{\lambda}(H_{\lambda})\,d\dot{k}_{1}\,d\xi_{1}\,dH_{\lambda}\,d\xi_{2}\,d\dot{k}_{2}.
\end{equation*}
Now, combining the fact that $M^{\text{reg}}_{\lambda}=K_{\lambda}A^{+}_{\lambda}K_{\lambda}$, the relevant integral formula and the fact that the complement of $M^{\text{reg}}$ has measure zero in $M$, it follows that the integral 
\begin{equation*}
\int_{K_{\lambda}}\int_{\mathfrak{a}^{+}_{\lambda}}\int_{K_{\lambda}}\int_{K_{\lambda}\backslash K}\mathcal{I}(\text{exp}H_{\lambda},\xi_{1}k_{1},\xi_{2}^{-1}k_{2}^{-1})\Omega_{\lambda}(H_{\lambda})\,d\dot{k}_{1} \,d\xi_{1}\,dH_{\lambda}\,d\xi_{2}
\end{equation*}
is equal to 
\begin{equation*}
\int_{{M}_{\lambda}}\int_{K_{\lambda}\backslash K}\Gamma_{\lambda,l}(m_{\lambda},\pi(k_{1})\dot{\pi}(X)v_{1},\pi(k_{2}^{-1})v_{2})\overline{\Gamma_{\mu,m}(m_{\lambda},\pi(k_{1})v_{3},\pi(k_{2}^{-1})v_{4})}\,d\dot{k}_{1}\,dm_{\lambda}.
\end{equation*}
For $k_{1}\in K$, we define 
\begin{equation*}
f(k_{1}):=\langle \Gamma_{\lambda,l}(m_{\lambda},\pi(k_{1})\dot{\pi}(X)v_{1},\pi(k^{-1}_{2})v_{2}),\Gamma_{\mu,m}(m_{\lambda},\pi(k_{1})v_{3},\pi(k^{-1}_{2})v_{4})\rangle_{L^{2}(M_{\lambda})}.
\end{equation*}
The function $f$ is invariant under left-multiplication by $K_{\lambda}$. Indeed, if \\ $\xi\in K_{\lambda}$, then 
\begin{equation*}
\Gamma_{\lambda,l}(m_{\lambda},\pi(\xi k_{1})\dot{\pi}(X)v_{1},\pi(k_{2})v_{2})=\Gamma_{\lambda,l}(m_{\lambda}\xi,\pi(k_{1})\dot{\pi}(X)v_{1},\pi(k_{2})v_{2})
\end{equation*}
and similarly for the $\Gamma_{\mu,m}$-term. Since the right-regular representation of $M_{\lambda}$ is unitary, we have 
\begin{equation*}
\langle \Gamma_{\lambda,l}(m_{\lambda}\xi,\pi(k_{1})\dot{\pi}(X)v_{1},\pi(k^{-1}_{2})v_{2}),\Gamma_{\mu,m}(m_{\lambda}\xi,\pi(k_{1})v_{3},\pi(k^{-1}_{2})v_{4})\rangle_{L^{2}(M_{\lambda})}=f(k).
\end{equation*}
An application of the quotient integral formula \cite[Theorem 2.51]{Folland}, gives 
\begin{equation*}
\int_{K}f(k_{1})\,dk_{1}=\int_{K_{\lambda}\backslash K}\int_{K}f(\xi k_{1})\,d\xi\,d\dot{k}_{1}=\text{vol}(K_{\lambda})\int_{K_{\lambda}\backslash K}f(k_{1}) \,d\dot{k}_{1}.
\end{equation*}
By (2) in Lemma~\ref{Interchange} and appealing again to the Fubini-Tonelli theorem, we interchange the integrals over $M_{\lambda}$ and $K_{\lambda}\backslash K$ to obtain that 
\begin{equation*}
\int_{K/K_{\lambda}}\int_{M_{\lambda}}\int_{K_{\lambda}\backslash K}\Gamma_{\lambda,l}(m_{\lambda},\pi(k_{1})\dot{\pi}(X)v_{1},\pi(k_{2}^{-1})v_{2})\overline{\Gamma_{\mu,m}(m_{\lambda},\pi(k_{1})v_{3},\pi(k_{2}^{-1})v_{4})}\,d\xi_{1}dm_{\lambda}\,d\xi_{2}
\end{equation*}
equals
\begin{equation*}
\frac{1}{\text{vol}(K_{\lambda})}\int_{K/K_{\lambda}}\int_{K}f(k_{1})\,dk_{1}\,d\dot{k}_{2}
\end{equation*}
which, in turn, equals 
\[\begin{split}
\frac{1}{\text{vol}(K_{\lambda})}\int_{K/K_{\lambda}}\int_{K}\qquad&\cr
\langle \Gamma_{\lambda,l}(m_{\lambda},\pi(k_{1})\dot{\pi}(X)v_{1},&\pi(k^{-1}_{2})v_{2}),\Gamma_{\mu,m}(m_{\lambda},\pi(k_{1})v_{3},\pi(k^{-1}_{2})v_{4})\rangle_{L^{2}(M_{\lambda})}\,dk_{1}\,d\dot{k}_{2}.
\end{split}\]
For fixed $k_{2}\in K$, set $w_{2}:=\pi(k_{2}^{-1})v_{2}$ and $w_{4}:=\pi(k_{2}^{-1})v_{4}$.
We reduced the problem to proving that 
\begin{equation*}
\int_{K}\langle \Gamma_{\lambda,l}(m_{\lambda},\pi(k_{1})\dot{\pi}(X)v_{1},w_{2}),\Gamma_{\mu,m}(m_{\lambda},\pi(k_{1})v_{3},w_{4})\rangle_{L^{2}(M_{\lambda})}\,dk_{1}
\end{equation*}
equals 
\begin{equation*}
-\int_{K}\langle \Gamma_{\lambda,l}(m_{\lambda},\pi(k_{1})v_{1},w_{2}),\Gamma_{\mu,m}(m_{\lambda},\pi(k_{1})\dot{\pi}(X)v_{3},w_{4})\rangle_{L^{2}(M_{\lambda})}\,dk_{1}.
\end{equation*}
The result is therefore a consequence of Proposition \ref{intrep}.
\end{proof}

\begin{lem}\label{Interchange} Let $v_{1},w_{2},v_{3},w_{4}\in H_{K}$. Let $\left[\lambda,l\right],\left[\mu,m\right]\in \mathcal{C}/{\sim}$ be such that $I_{\lambda}=I_{\mu}$, $\lambda|_{\mathfrak{a}_{\lambda_{0}}}=\mu|_{\mathfrak{a}_{\lambda_{0}}}$ and $\textbf{d}(\pi)=|I_{\lambda}|+\sum_{i\in I_{\lambda}}(l_{i}+m_{i})$. Then the following holds: 
\begin{enumerate}
\item[(1)] 
\begin{equation*}
\int_{K}\int_{\mathfrak{a}_{\lambda}^{+}}|\Gamma_{\lambda,l}(\text{exp}H_{\lambda},\pi(k_{1})v_{1},w_{2})\overline{\Gamma_{\mu,m}(\text{exp}H_{\lambda},\pi(k_{1})v_{3},w_{4})}|\,dH_{\lambda}\,dk_{1}<\infty,
\end{equation*}
\item[(2)] 
\begin{equation*}
\int_{K_{\lambda}\backslash K}\int_{M_{\lambda}}|\Gamma_{\lambda,l}(m_{\lambda},\pi( k)v_{1},w_{2})\overline{\Gamma_{\mu,m}(m_{\lambda},\pi( k)v_{3},w_{4})}|\,dm_{\lambda}\,d\dot{k}<\infty.
\end{equation*}
\item[(3)] For any fixed $H_{\lambda}\in \mathfrak{a}^{+}_{\lambda}$, we have 
\begin{equation*}
\int_{K_{\lambda}\backslash K}\int_{K_{\lambda}}|\Gamma_{\lambda,l}(\text{exp}H_{\lambda},\pi(\xi k)v_{1},w_{2})\overline{\Gamma_{\mu,m}(\text{exp}H_{\lambda},\pi(\xi k)v_{3},w_{4})}|\,d\xi\,d\dot{k}<\infty.
\end{equation*}
\end{enumerate}
\end{lem}
\begin{proof} To prove (1), we begin by observing that, for a fixed element $k$ of $K$, the functions $\Gamma_{\lambda,l}(\text{exp}H_{\lambda},\pi(k)v_{1},v_{2})$ and $\Gamma_{\mu,m}(\text{exp}H_{\lambda},\pi(k)v_{3},v_{4})$ are square-integrable on $\mathfrak{a}^{+}_{\lambda}$ by Proposition \ref{squareint}.
Therefore, we have 
\begin{equation*}
\int_{\mathfrak{a}^{+}_{\lambda}}|\Gamma_{\lambda,l}(\text{exp}H_{\lambda},\pi(k)v_{1},w_{2})\overline{\Gamma_{\mu,m}(\text{exp}H_{\lambda},\pi(k)v_{3},w_{4})}|\,dm_{\lambda}<\infty.
\end{equation*}
Hence, we can define the function 
\begin{equation*}
h\colon K\longrightarrow \mathbb{R}_{\geq 0}, \text{ }h(k)=\int_{\mathfrak{a}^{+}_{\lambda}}|\Gamma_{\lambda,l}(\text{exp}H_{\lambda},\pi(k)v_{1},w_{2})\overline{\Gamma_{\mu,m}(\text{exp}H_{\lambda},\pi(k)v_{3},w_{4})}|\,dH_{\lambda}
\end{equation*}
and the result will follow if we establish the continuity of $h$. The $K$-finiteness of $v_{1}$ and $v_{3}$ implies the existence of finitely many $K$-finite vectors $v^{(1)}_{1},\dots,v_{1}^{(p)}$ and finitely many $K$-finite vectors $v^{(1)}_{3},\dots,v_{3}^{(q)}$ such that 
\begin{equation*}
\pi(k)v_{1}=\sum^{p}_{i=1}a_{i}(k)v^{(i)}_{1}, \text{ and }\pi(k)v_{3}=\sum^{q}_{j=1}b_{j}(k)v^{(j)}_{3}
\end{equation*}
for continuous complex-valued functions $a_{i}$ and $b_{j}$.  
Let $k_{0}\in K$. Then 
\begin{equation*}
|h(k)-h(k_{0})|
\end{equation*}
is majorised by the integral over $\mathfrak{a}^{+}_{\lambda}$ of 
\[\begin{split}
||\Gamma_{\lambda,l}(\text{exp}H_{\lambda},\pi(k)v_{1},w_{2})\overline{\Gamma_{\mu,m}(\text{exp}H_{\lambda},\pi(k)v_{3},w_{4})}|&\cr
-|\Gamma_{\lambda,l}(\text{exp}H_{\lambda},\pi(k_{0})v_{1},w_{2})&\overline{\Gamma_{\mu,m}(\text{exp}H_{\lambda},\pi(k_{0})v_{3},w_{4})}||.
\end{split}\]
By reverse triangle inequality, the integrand is majorised by 
\[\begin{split}
|\Gamma_{\lambda,l}(\text{exp}_{\lambda},\pi(k)v_{1},w_{2})\overline{\Gamma_{\mu,m}(\text{exp}H_{\lambda},\pi(k)v_{3},w_{4})}&\cr
-\Gamma_{\lambda,l}(\text{exp}H_{\lambda},\pi(k_{0})v_{1},w_{2})&\overline{\Gamma_{\mu,m}(\text{exp}H_{\lambda},\pi(k_{0})v_{3},w_{4})}|,
\end{split}\]
which, in turn, is less than or equal to 
\begin{equation*}
\sum^{p}_{i=1}\sum^{q}_{j=1}|a_{i}(k)\overline{b_{j}(k)}-a_{i}(k_{0})\overline{b_{j}(k_{0})}||\Gamma_{\lambda,l}(\text{exp}H_{\lambda},v^{(i)}_{1},w_{2})\overline{\Gamma_{\mu,m}(\text{exp}H_{\lambda},v_{3}^{(j)},w_{4})}|.
\end{equation*}
We obtained 
\[\begin{split}
|h(k)-h(k_{0})|\leq \sum^{p}_{i=1}\sum^{q}_{j=1}|a_{i}(k)\overline{b_{j}(k)}-a_{i}(k_{0})\overline{b_{j}(k_{0})}&\cr
|\int_{\mathfrak{a}^{+}_{\lambda}}|\Gamma_{\lambda,l}(\text{exp}H_{\lambda},v^{(i)}_{1},w_{2})&\overline{\Gamma_{\mu,m}(\text{exp}H_{\lambda},v_{3}^{(j)},w_{4})}|\,dH_{\lambda},
\end{split}\]
and the continuity follows from the continuity of the $a_{i}$'s and $b_{j}$'s. \\ 

For (2), we begin by observing that for fixed $k\in K$, the functions $\Gamma_{\lambda,l}(m_{\lambda},\pi(k)v_{1},w_{2})$ and $\Gamma_{\mu,m}(m_{\lambda},\pi(k)v_{3},w_{4})$ are square-integrable on $M_{\lambda}$ by Proposition \ref{squareint}. Therefore, we have 
\begin{equation*}
\int_{M_{\lambda}}|\Gamma_{\lambda,l}(m_{\lambda},\pi(k)v_{1},w_{2})\overline{\Gamma_{\mu,m}(m_{\lambda},\pi(k)v_{3},w_{4})}|\,dm_{\lambda}<\infty.
\end{equation*}
Hence, we can define the function 
\begin{equation*}
h\colon K\longrightarrow \mathbb{R}_{\geq 0}, \text{ }h(k)=\int_{M_{\lambda}}|\Gamma_{\lambda,l}(m_{\lambda},\pi(k)v_{1},w_{2})\overline{\Gamma_{\mu,m}(m_{\lambda},\pi(k)v_{3},w_{4})}|\,dm_{\lambda}.
\end{equation*}
Arguing as for (1), we obtain that $h$ is continuous.\\

By the right-invariance of the Haar measure on $M_{\lambda}$ and since 
\begin{equation*}
\Gamma_{\lambda,l}(m_{\lambda},\pi(\xi k)v_{1},w_{2})=\Gamma_{\lambda,l}(m_{\lambda}\xi,\pi(k)v_{1},w_{2})
\end{equation*}
for every $\xi\in K_{\lambda}$ (and similarly for the $\Gamma_{\mu,m}$-term), the function $h$ is invariant under multiplication on the left by elements in $K_{\lambda}$ and it therefore descends to a continuous function on $K_{\lambda}\backslash K$, concluding the proof of (2).\\

For (3), given a fixed $H_{\lambda}\in \mathfrak{a}^{+}_{\lambda}$ the function 
\begin{equation*}
K_{\lambda}\longrightarrow \mathbb{C},\text{ } \xi\mapsto \Gamma_{\lambda,l}(\text{exp}H_{\lambda},\pi(\xi k)v_{1},w_{2})
\end{equation*}
is continuous. Indeed, let $\xi_{0}\in K_{\lambda}$. Since $\pi(k)v_{1}$ is $K$-finite, it is in particular $K_{\lambda}$-finite. Hence, there exist finitely many $K_{\lambda}$-finite vectors $v^{(1)}_{1},\dots,v^{(r)}_{1}$ such that 
\begin{equation*}
\pi(\xi)\pi(k)v=\sum^{r}_{i=1}c_{i}(\xi)v^{(i)}_{1},
\end{equation*}
where each $c_{i}$ is a complex-valued continuous function on $K_{\lambda}$. Therefore, the quantity 
\begin{equation*}
|\Gamma_{\lambda,l}(\text{exp}H_{\lambda},\pi(\xi k)v_{1},w_{2})-\Gamma_{\lambda,l}(\text{exp}H_{\lambda},\pi(\xi_{0}k)v_{1},w_{2})|
\end{equation*}
is bounded by 
\begin{equation*}
\sum^{r}_{i=1}|c_{i}(\xi)-c_{i}(\xi_{0})||\Gamma_{\lambda,l}(\text{exp}H_{\lambda},v^{(i)}_{1},w_{2})|
\end{equation*}
and the claim follows from the continuity of the $c_{i}$'s. \\

The same argument shows that, for fixed $H_{\lambda}\in\mathfrak{a}^{+}_{\lambda}$, the function 
\begin{equation*}
K_{\lambda}\longrightarrow \mathbb{C},\text{ } \xi\mapsto \Gamma_{\mu,m}(\text{exp}H_{\lambda},\pi(\xi )v_{3},w_{4})
\end{equation*}
is continuous and it follows that 
\begin{equation*}
\int_{K_{\lambda}}|\Gamma_{\lambda,l}(\text{exp}H_{\lambda},\pi(\xi k)v_{1},w_{2})\overline{\Gamma_{\mu,m}(\text{exp}H_{\lambda},\pi(\xi k)v_{3},w_{4})}|\,d\xi<\infty.
\end{equation*}
Hence, we can define the function 
\begin{equation*}
f\colon K\rightarrow \mathbb{R}_{\geq 0},\text{ } f(k)=\int_{K_{\lambda}}|\Gamma_{\lambda,l}(\text{exp}H_{\lambda},\pi(\xi k)v_{1},w_{2})\overline{\Gamma_{\mu,m}(\text{exp}H_{\lambda},\pi(\xi k)v_{3},w_{4})}|\,d\xi
\end{equation*}
and argue as in the proof of (2).
\end{proof}

We can now complete the strategy outlined in the Introduction. For fixed $v_{2},v_{4}\in H_{K}$, we define 
\begin{equation*}
A_{v_{2},v_{4}}:=D(\cdot,v_{2},\cdot,v_{4}),
\end{equation*}
which is linear in the first variable and conjugate linear in the second. For fixed $v_{1},v_{3}\in H_{K}$, we define 
\begin{equation*}
B_{v_{1},v_{3}}:=D(v_{1},\cdot,v_{3},\cdot),
\end{equation*}
which is conjugate-linear in the first variable and linear in the second.\\

\begin{thm}\label{main3} Let $G$ be a connected, semisimple Lie group with finite centre. Let $(\pi,H)$ be a tempered, irreducible representation of $G$. Then there exists $\textbf{f}(\pi)\in\mathbb{R}_{>0}$ such that, for all $v_{1},v_{2},v_{3},v_{4}\in H_{K}$, we have 
\begin{equation*}
\lim_{r\to\infty}\frac{1}{r^{\textbf{d}(\pi)}}\int_{G_{<r}}\langle \pi(g) v_{1},v_{2}\rangle \overline{\langle \pi(g)v_{3},v_{4}\rangle}\,dg=\frac{1}{\textbf{f}(\pi)}\langle v_{1},v_{3}\rangle \overline{\langle v_{2},v_{4}\rangle}.
\end{equation*}
\end{thm}
\begin{proof} Fix $v_{2},v_{4}\in H_{K}$. By Proposition \ref{ginv}, we can apply Corollary~\ref{Schur} to the form $A_{v_{2},v_{4}}$. Hence there exists $c_{v_{2},v_{4}}\in\mathbb{C}$ such that for all $v_{1},v_{3}\in H_{K}$ we have 
\begin{equation*}
A_{v_{2},v_{4}}(v_{1},v_{3})=c_{v_{2},v_{4}}\langle v_{1},v_{3}\rangle.
\end{equation*}
Similarly, fixing $v_{1},v_{3}\in H_{K}$, by Proposition \ref{ginv} and Lemma \ref{redginv} there exists a $d_{v_{1},v_{3}}\in\mathbb{C}$ such that 
\begin{equation*}
\overline{B_{v_{3},v_{1}}(v_{4},v_{2})}=d_{v_{1},v_{3}}\langle v_{4},v_{2}\rangle,
\end{equation*}
since the left-hand side is conjugate-linear in the first variable. Hence, since 
\begin{equation*}
\overline{B_{v_{3},v_{1}}(v_{4},v_{2})}=B_{v_{1},v_{3}}(v_{2},v_{4}),
\end{equation*}
we obtain 
\begin{equation*}
B_{v_{1},v_{3}}(v_{2},v_{4})=d_{v_{1},v_{3}}\overline{\langle v_{2},v_{4}\rangle}
\end{equation*}
By definition, we have 
\begin{equation*}
D(v_{1},v_{2},v_{3},v_{4})=A_{v_{2},v_{4}}(v_{1},v_{3})=B_{v_{1},v_{3}}(v_{2},v_{4}),
\end{equation*}
so, for a vector $v_{0}\in H_{K}$ of norm $1$, using (2) of Theorem 1.2, we obtain a real number $C(v_{0},v_{0})>0$ such that 
\begin{equation*}
D(v_{0},v_{0},v_{0},v_{0})=C(v_{0},v_{0})=c_{v_{0},v_{0}}=d_{v_{0},v_{0}}.
\end{equation*}
Computing $D(v_{1},v_{0},v_{3},v_{0})$, we have 
\begin{equation*}
d_{v_{1},v_{3}}=c_{v_{0},v_{0}}\langle v_{1},v_{3}\rangle.
\end{equation*}
Therefore, we obtained 
\begin{equation*}
D(v_{1},v_{2},v_{3},v_{4})=c_{v_{0},v_{0}}\langle v_{1},v_{3}\rangle\overline{\langle v_{2},v_{4}\rangle},
\end{equation*}
showing that $\textbf{f}(\pi):=\displaystyle\frac{1}{C(v_{0},v_{0})}$ does not depend on the choice of $v_{0}$, as required.
\end{proof}

\end{document}